\documentclass{article}
\usepackage{amssymb,latexsym,amsmath,amsthm,makeidx}
\usepackage[dvips]{graphics}

\newtheorem{theorem}{Theorem}[section]
\newtheorem{lemma}[theorem]{Lemma}
\newtheorem{corollary}[theorem]{Corollary}
\newtheorem{proposition}[theorem]{Proposition}

\theoremstyle{definition}
\newtheorem{definition}[theorem]{Definition}

\newtheorem{remark}[theorem]{Remark}

\newcommand{\Hom}{{ \rm Hom }}

\newcommand{\Mod}{{ \rm Mod }}

\newcommand{\Ind}{{ \rm Ind }}

\newcommand{\DGC}{{\rm DGComod}}

\newcommand{\GC}{{\rm GComod}}

\newcommand{\Coder}{{\rm Coder}}

\newcommand{\Alg}{{\rm Alg}}

\newcommand{\GModi}{{\rm G}\hueca{M}{\rm od}} 
\newcommand{\TGModi}{{\rm TG}\hueca{M}{\rm od}} 
\newcommand{\GModit}{{\rm G}\hueca{M}{\rm od}\hskip-1.5pt\hbox{\rm I}} 
\newcommand{\GModitsub}{{\rm G}\hueca{M}{\rm od}\hskip-1.5pt\hbox{\scriptsize\rm I}} 
\newcommand{\TGModit}{{\rm TG}\hueca{M}{\rm od}\hskip-1.5pt\hbox{\rm I}} 
\newcommand{\GM}{{\rm GMod}}

\newcommand{\TGM}{{\rm TGMod}}
\newcommand{\GInd}{{\rm GCoind}}
\newcommand{\DGInd}{{\rm DGCoind}}
\newcommand{\DGCoalg}{{\rm DGCoalg}}
\newcommand{\DGComod}{{\rm DGComod}}

\newcommand{\sgn}{{ \rm sgn }}

\newcommand{\g}{\hbox{-}}
\newcommand{\uddots}{\mathinner{\mkern1mu\raise1pt\vbox{\kern7pt\hbox{.}}
\mkern2mu\raise4pt\hbox{.}\mkern2mu\raise7pt\hbox{.}\mkern1mu}}

\newcommand{\I}{{\rm I\hskip-2pt I}}

\newcommand{\hueca}[1]{\mathbb{#1}}
\newcommand{\lddots}{
\mathinner{
\mkern1mu\raise1pt}\vbox{\kern7pt\hbox{.}}
\mkern2mu\raise3pt\hbox{.}
\mkern2mu\raise7pt\hbox{.}\mkern1mu}

\newcommand{\rightdashmap}[1]{\smash{\mathop{\hbox to 
20pt{-\,-\,-\,\rightarrowfill}}\limits^{#1}}}

\newcommand{\rightmap}[1]{\smash{\mathop{\hbox to 
20pt{\rightarrowfill}}\limits^{#1}}}
\newcommand{\leftmap}[1]{\smash{\mathop{\hbox to 
20pt{\leftarrowfill}}\limits^{#1}}}
\newcommand{\lmapdown}[1]{\llap{$\vcenter{\hbox{$\scriptstyle#1$}}$}
\Bigg\downarrow}
\newcommand{\rmapdown}[1]{\Bigg\downarrow\rlap{$\vcenter{\hbox{$\scriptstyle#1$}
}$}}

\newcommand{\longrightmap}[1]{\smash{\mathop{\hbox to 
4cm{\rightarrowfill}}\limits^{#1}}}
\newcommand{\longleftmap}[1]{\smash{\mathop{\hbox to 
4cm{\leftarrowfill}}\limits^{#1}}}
\newcommand{\medrightmap}[1]{\smash{\mathop{\hbox to 
2cm{\rightarrowfill}}\limits^{#1}}}
\newcommand{\medleftmap}[1]{\smash{\mathop{\hbox to 
2cm{\leftarrowfill}}\limits^{#1}}}

\newcommand{\shortlmapdown}[1]
{\downarrow\rlap{$\vcenter{\hbox{$\scriptstyle#1$}}$}}

\newcommand{\idmapdown}[1]
{\hskip-8pt\mathop{\hskip-5pt\raise6pt
\hbox{$\scriptstyle#1$}\hskip-5pt\swarrow}}

\newcommand{\ddmapdown}[1]
{\hskip-5pt\mathop{\searrow\hskip-6pt\raise5pt\hbox{$\scriptstyle#1$}}}

\newcommand{\idmapup}[1]
{\hskip-5pt\mathop{\nwarrow\hskip-6pt \raise5pt\hbox{$\scriptstyle#1$}}}

\newcommand{\ddmapup}[1]
{\hskip-8pt\mathop{\hskip-5pt\raise6pt\hbox{$\scriptstyle#1$}\hskip-5pt\nearrow}
}

\newcommand{\flechypunt}[2]{\ \smash{\mathop{
   \raise 3pt \hbox to 40pt{\rightarrow}\hskip-40pt \lower 3pt
   \hbox to 40pt{\dashrightarrow}}\limits^{#1}_{#2}}\ }
\newcommand{\dobleflechavieneva}[2]{\ \smash{\mathop{
   \raise 3pt \hbox to 40pt{\leftarrowfill}\hskip-40pt \lower 3pt
   \hbox to 40pt{\rightarrowfill}}\limits^{#1}_{#2}}\ }
\newcommand{\dobleflechavienevabis}[2]{\ \smash{\mathop{
   \raise 3pt \hbox to 20pt{\leftarrowfill}\hskip-20pt \lower 3pt
   \hbox to 40pt{\rightarrowfill}}\limits^{#1}_{#2}}\ }

\newcommand{\dobleflecha}[2]{\ \smash{\mathop{
   \raise 3pt \hbox to 40pt{\rightarrowfill}\hskip-40pt \lower 3pt
   \hbox to 40pt{\rightarrowfill}}\limits^{#1}_{#2}}\ }

\newcommand{\longequal}{\ \smash{\mathop{
   \raise 5pt \hbox to 35pt{\hrulefill}\hskip-35pt \lower 0pt
   \hbox to 35pt{\hrulefill}}}\ }

\newcommand{\raya}[1]{\ \smash{\mathop{\raise 2pt \hbox to 
10pt{\hrulefill}}\limits^{#1}}\ }

\title{\bf \Large Differential graded bocses and $A_\infty$-modules\\ 
\, $ $\\ \, $ $\\ }
\author{R. Bautista, E. P\'erez and L. Salmer\'on}
\makeindex

\begin{document}

 \maketitle
 \renewcommand{\thefootnote}{}



\centerline{\emph{Dedicated to J.A. De la Pe\~na on the occasion of his 60th birthday}}

\begin{abstract}
We introduce and study the category of twisted modules over a triangular differential graded  bocs. We show that in this category idempotents split, that it admits a natural structure of a Frobenius category, that a twisted module is homotopically trivial iff its underlying complex is acyclic, and that any homotopy equivalence of differential graded bocses determines an equivalence of the corresponding homotopy categories of twisted modules. The category of modules over an $A_\infty$-algebra is equivalent to the category of twisted modules over a triangular differential graded bocs, so all the preceding statements lift to the former category.
\end{abstract}

   
   \section{Introduction}
   
   There is some analogy between the theory of modules over bocses  and the theory of $A_\infty$-modules over an $A_\infty$-algebra. This can be noticed, for instance, in the fact that 
   the morphisms $f:M\rightmap{}N$ of modules over a given triangular bocs can be handled as pairs $(f^0,f^1)$ of morphisms, where the properties of the first component $f^0:M\rightmap{}N$ reflect on the properties of the whole morphism $f$, similarly, the properties of a morphism $g:M\rightmap{}N$ of 
   $A_\infty$-modules over a given $A_\infty$-algebra depend on the properties of the first component $g_1:M\rightmap{}N$.
   
   For instance, the natural exact structure on the category of modules of a given triangular bocs is determined by the first components of morphisms, while in the category of modules over a given $A_\infty$-algebra the natural exact structure is determined by the first components of morphisms. 
   
   Motivated by this analogy, we work here with the natural notions of  differential graded bocses and their  category of twisted graded modules, as defined below in (\ref{D: dgb y su cat de mods}) and (\ref{D: category of twisted graded modules of a bocs}). The study of these algebraic structures led us to simpler proofs of some well known facts on categories of $A_\infty$-modules. 
   
   In order to describe more precisely our constribution, let us fix some notation and recall some well known concepts. Throughout this article, we denote by  $k$  a fixed ground field, which will act centrally on any $k$-algebra and on any $k$-$k$-bimodule we consider. We also fix a finite-dimensional  semisimple $k$-algebra $S$.  Moreover, we ask that 
   $S\otimes_k S^{op}$ is also semisimple. This holds, for instance when $k$ is a perfect field or if $S$ is a finite product of copies of the ground field.  We will make  explicit any further requirement on $S$, when needed.   We consider $S$ as a $\hueca{Z}$-graded $k$-algebra concentrated at degree $0$. 
   
   We will consider graded right $S$-modules (or graded left $S$-modules, or graded $S$-$S$-bimodules) $M$, so $M$ is equipped with a direct sum decomposition 
$$M=\bigoplus_{j\in \hueca{Z}}M_j$$
of right $S$-submodules  $M_j$ (resp.  left $S$-submodules, or $S$-$S$-subbimodules). The elements $x\in M_j$ are called \emph{homogeneous of degree $j$}, and we indicate this fact by $\vert x\vert=j$. Given two graded right $S$-modules, a \emph{homogeneous morphism $f:M\rightmap{}N$ of degree $\vert f\vert=d$} satisfies that $f(M_j)\subseteq N_{j+d}$, for any $j\in \hueca{Z}$. We denote by 
$\Hom_{\GM\g S}^d(M,N)$ the space of homogeneous morphisms of graded right $S$-modules $f:M\rightmap{}N$ of degree $d$. Then, we have the graded category 
$\GM\g S$ of graded right $S$-modules with hom spaces
$$\Hom_{\GM\g S}(M,N)=\bigoplus_{d\in \hueca{Z}}\Hom_{\GM\g S}^d(M,N).$$

Similarly, we have the graded category $\GM\g S\g S$  of  
graded $S$-$S$-bimodules.

The tensor product $M\otimes_SN$ of two graded objects $M$ and $N$ is equipped with the standard grading given by the homogeneous components 
$$(M\otimes_S N)_j=\bigoplus_{s+t=j}M_s\otimes_SN_t.$$

   Assume that $f: M\rightmap{}N$ is a homogeneous morphism of graded right $S$-modules and 
   $g:M'\rightmap{}N'$  is a homogeneous morphism of graded left $S$-modules. 
    Then, their \emph{tensor product}
   $f\otimes g$ is the homogeneous linear map 
   $M\otimes_S M'\rightmap{}N\otimes_S N'$ of degree $\vert f\otimes g\vert =\vert f\vert+\vert g\vert$ defined, for any homogeneous elements $x\in M$ and $y\in M'$, 
   by the following formula. 
   $$[f\otimes g](x\otimes y):=(-1)^{\vert g\vert\vert x\vert}f(x)\otimes g(y).$$
  So, if  $f:M\rightmap{}N$,  $h:N\rightmap{}L$ are  morphisms of graded  right $S$-modules, 
  and        $g:M'\rightmap{}N'$, $t:N'\rightmap{}L'$ are morphisms of graded left $S$-modules, we have  
  $$(h\otimes t)(f\otimes g)=
  (-1)^{\vert t\vert\vert f\vert}hf\otimes tg.$$
   
  Now we recall some basic definitions of the theory of $A_\infty$-algebras. 
   
    \begin{definition}\label{A-infinite algebra} 
    An \emph{$A_{\infty}$-algebra} $A$ is a graded $S$-$S$-bimodule  
   $A$, equipped with a sequence of homogeneous 
   morphisms of $S$-$S$-bimodules  
   $$\{m_n:A^{\otimes n}\rightmap{}A\}_{n\in \hueca{N}},$$
   where each $m_n$ has degree $\vert m_n\vert=2-n$, such that, for each 
   $n\in \hueca{N}$, the following \emph{Stasheff identity} holds.
   $$S_n:\sum_{\scriptsize\begin{matrix}r+s+t=n\\ s\geq 1;r,t\geq 0\end{matrix}} 
   (-1)^{r+st}m_{r+1+t}(id^{\otimes r}\otimes m_s\otimes id^{\otimes t})=0.$$
   
Given two $A_\infty$-algebras $(A,\{m_n\})$ and $(B,\{m'_n\})$, 
  a \emph{morphism of
  $A_\infty$-algebras} $f:(A,\{m_n\})\rightmap{}(B,\{m'_n\})$  
  is a family of homogeneous $S$-$S$-bimodule morphisms  
  $\{f_n:A^{\otimes n}\rightmap{}B\}_{n\in \hueca{N}}$ 
  such that each $f_n$ is homogeneous of degree $\vert f_n\vert=1-n$ and the equality
  $\Sigma_n=\Sigma'_n$ holds for all $n\in \hueca{N}$, 
  where
  $$\Sigma_n=\sum_{\scriptsize
  \begin{matrix}r+s+t=n\\ r,t\geq 0; s\geq 1\end{matrix}}(-1)^{r+st}
 f_{r+1+t}(id^{\otimes r}\otimes m_s\otimes id^{\otimes t}),$$  
 $$\Sigma'_n=\sum_{\scriptsize\begin{matrix}1\leq r\leq n\\ i_1+\cdots+ i_r=n\end{matrix}}(-1)^{\sgn(i_1,\ldots,i_r)}
 m'_{r}(f_{i_1}\otimes\cdots\otimes f_{i_r}), $$
 where $i_1,\ldots,i_r\geq 1$ and 
 $$\sgn(i_1,\ldots,i_r)=(r-1)(i_1-1)+(r-2)(i_2-1)+\cdots +2(i_{r-2}-1)+(i_{r-1}-1).$$
   
  Given two morphisms of $A_\infty$-algebras $f:A\rightmap{}B$ 
 and $g:B\rightmap{}C$, 
 their composition $g\circ f:A\rightmap{}C$ is the family 
 $(g\circ f)=\{(g\circ f)_n\}_{n\in \hueca{N}}$
 defined, for each $n\in \hueca{N}$, by 
 $$(g\circ f)_n=\sum_{\scriptsize
  \begin{matrix} 1\leq s\leq n\\
           i_1+\cdots +i_s=n\\
          \end{matrix}}(-1)^{\sgn(i_1,\ldots,i_s)}g_s(f_{i_1}\otimes\cdots\otimes f_{i_s}).$$ 
 \end{definition}
 
 The preceding notions give rise to a $k$-category $\Alg_\infty$ of $A_\infty$-algebras with morphisms of $A_\infty$-algebras. 
 There is a well known full and faithful functor  $\Psi:\Alg_\infty \rightmap{}\DGCoalg$, where $\DGCoalg$ denotes the category of  differential graded $S$-coalgebras,  which has been used successfully to study the category $\Alg_\infty$, see \cite{K1}, \cite{L-H} and their references. The functor $\Psi$ is given by the bar construction: it maps each $A_\infty$-algebra $(A,\{m_n\})$ onto its reduced tensor $S$-coalgebra $\overline{T}_S(A[1])$, equipped with  a differential $\delta$ induced by the family $\{m_n\}$, see \cite{K1}(3.6).
 
 There is some notational difference with \cite{K1} and \cite{L-H}, which is a minor one as explained by the following.

\begin{lemma}\label{L: S ene y S' ene}
Let $A$ be a graded $S$-$S$-bimodule 
equipped with a sequence of homogeneous morphisms of $S$-$S$-bimodules 
   $\{m_n:A^{\otimes n}\rightmap{}A\}_{n\in \hueca{N}}$. For $n\in \hueca{N}$, define 
   $$m'_n:=(-1)^{\frac{n(n-1)}{2}} m_n.$$
   Then, the equation $S_n$ appearing 
   in the last definition holds iff the following equation holds
   $$S'_n:\sum_{\scriptsize\begin{matrix}r+s+t=n\\ s\geq 1;r,t\geq 0\end{matrix}} 
   (-1)^{rs+t}m'_{r+t+1}(id^{\otimes r}\otimes m_s\otimes id^{\otimes t})=0.$$
 Moreover, if we make 
 $$Z_n=\sum_{\scriptsize\begin{matrix}r+s+t=n\\ s\geq 1;r,t\geq 0\end{matrix}} 
   (-1)^{r+st}m_{r+t+1}(id^{\otimes r}\otimes m_s\otimes id^{\otimes t})$$
 and 
 $$Z'_n=\sum_{\scriptsize\begin{matrix}r+s+t=n\\ s\geq 1;r,t\geq 0\end{matrix}} 
   (-1)^{rs+t}m'_{r+t+1}(id^{\otimes r}\otimes m'_s\otimes id^{\otimes t}),$$
 we have that 
 $$Z'_n=(-1)^{\frac{n(n-1)}{2}} Z_n.$$
\end{lemma}

\begin{proof} By definition, 
$$Z'_n=\sum_{\scriptsize\begin{matrix}r+s+t=n\\ s\geq 1;r,t\geq 0\end{matrix}} 
   (-1)^{rs+t}(-1)^{\frac{(r+t+1)(r+t)}{2}}(-1)^{\frac{s(s-1)}{2}}m_{r+t+1}
   (id^{\otimes r}\otimes m_s\otimes id^{\otimes t}).$$
  The non-negative integers $r,s,t$ satisfy $r+s+t=n$ and $s\geq 1$.  
   Then, our  statement follows from the next 
   congruence modulo $2$
  $$rs+t+\frac{(r+t+1)(r+t)}{2}+\frac{s(s-1)}{2}\equiv r+st+\frac{n(n-1)}{2}.$$
\end{proof}

Thus, the choice for the signs in the formulas we made in the last definition, 
which is also adopted  by other authors, is not essential. It 
 will reflect on some signs in the formulas involving $A_\infty$-objects.

 \begin{definition}\label{D: A-infinto modulos derechos}
 Let $A=(A,\{m_n\})$ be an $A_\infty$-algebra. 
 Then,  a graded right $S$-module $M$ is called a 
 \emph{right $A_\infty$-module over $A$}\index{module!right $A_\infty$-module} iff $M$ is equipped 
 with a family $\{m_n^M\}_{n\in \hueca{N}}$ of morphism of right $S$-modules such that, 
 each map  
 $$m_n^M:M\otimes A^{\otimes (n-1)}\rightmap{}M$$
 is homogeneous of degree $\vert m^M_n\vert=2-n$ and, for every $n\in \hueca{N}$,
 the condition 
 $\Sigma_n^++ \Sigma_n^0=0$ is satisfied, where  
   $$\Sigma_n^+=\sum_{\scriptsize\begin{matrix}r+s+t=n\\ t\geq 0; r,s\geq 1\end{matrix}}(-1)^{r+st}
 m^M_{r+1+t}(id^{\otimes r}\otimes m_s\otimes id^{\otimes t})$$
 and
 $$\Sigma_n^0=\sum_{\scriptsize\begin{matrix}s+t=n\\ t\geq 0; s\geq 1\end{matrix}}(-1)^{st}
 m^M_{1+t}( m^M_s\otimes id^{\otimes t}).$$

 In the  sum $\Sigma_n^+$, for simplicity,  we abuse of the language 
 writing $id^{\otimes r}$ instead of $id_M\otimes id_A^{\otimes (r-1)}$.
 In case $n=1$, the condition reduces to $\Sigma^0_1=0$ which is equivalent 
 to say that $m^M_1$ satisfies $(m^M_1)^2=0$. 
 
  Given $M=(M,\{m^M_n\})$ and $N=(N,\{m_n^N\})$  right $A_\infty$-modules over $A$. 
  Then, a \emph{morphism of right $A_\infty$-modules $f:M\rightmap{}N$ 
  over} $A$ is a family $f=\{f_n\}_{n\in \hueca{N}}$, where each 
  $$f_n:M\otimes_S A^{\otimes (n-1)}\rightmap{}N$$
  is a homogeneous morphism of graded right $S$-modules of degree $\vert f_n\vert=1-n$ such that, 
  for each $n\in \hueca{N}$, 
  the equality $\Sigma_n^{f +}+\Sigma_n^{f 0}+\Sigma_n^{f -}=0$, where 
  $$\Sigma_n^{f +}=\sum_{\scriptsize\begin{matrix}r+s+t=n\\ t\geq 0; r,s\geq 1\end{matrix}}
  (-1)^{r+st}
 f_{r+1+t}(id^{\otimes r}\otimes m_s\otimes id^{\otimes t}),$$
$$\Sigma_n^{f 0}=\sum_{\scriptsize\begin{matrix}s+t=n\\ t\geq 0; s\geq 1\end{matrix}}
(-1)^{st}
 f_{1+t}( m^M_s\otimes id^{\otimes t}),$$
and
$$\Sigma_n^{f-}=-\sum_{\scriptsize\begin{matrix}r+s=n\\ r\geq 1; s\geq 0\end{matrix}}
(-1)^{(r+1)s}
 m^N_{1+s}(f_r\otimes id^{\otimes s}).$$
 The condition in case $n=1$ is equivalent to $m_1^Nf_1=f_1m_1^M$, 
 that is, to the requirement that the map $f_1:M\rightmap{}N$ is a morphism 
 of complexes of right $S$-modules. 
 The morphism $f$ is called a \emph{quasi-isomorphism} iff $f_1$ is so. 
 
 The class of 
 right $A_\infty$-modules  together with the morphisms between them is 
 a category with the following composition. If $f:M\rightmap{}N$ and 
 $g:N\rightmap{}L$ are morphisms of right $A_\infty$-modules over $A$, 
 their composition 
 $$g\circ f=\{(g\circ f)_n\}_{n\in \hueca{N}}:M\rightmap{}L$$ 
 is defined, for each  $n\in \hueca{N}$, by 
 $$(g\circ f)_n=\sum_{\scriptsize\begin{matrix}r+s=n\\ r\geq 1;
 s\geq 0\end{matrix}}(-1)^{(r+1)s}
 g_{1+s}( f_r\otimes id^{\otimes s}).$$
 Given a left $A_\infty$-module $M$, the identity 
 $\I_M=\{h_n\}:M\rightmap{}M$ is given by $h_1=id_M$ and $h_n=0$,
 for all $n\geq 2$. We shall denote this category by 
 $\Mod_\infty\g A$.
\end{definition}
 
 For the study of the category $\Mod_\infty\g A$ one of the main tools is the graded category $\DGComod\g {\cal B}_A$, the category of differential graded comodules over the differential tensor $S$-coalgebra ${\cal B}_A=(T_S(A[1]),\mu,\epsilon,\delta)$, with $\mu$  the comultiplication given by the bar construction, 
 $\epsilon:T_S(A[1])\rightmap{}S$ is the canonical projection, and $\delta$ is the differential induced by the family of operations $\{m_n\}_{n\in \hueca{N}}$ of the $A_\infty$-algebra $A$. This is so, because of there is a 
 full and faithful functor 
 $\Phi:\Mod_\infty\g A\rightmap{}\DGComod^0\g {\cal B}_A$. 
 The superindex $0$ indicates the subcategory of $\DGComod\g {\cal B}_A$ with the same objects but with 
 only degree zero homogeneous morphisms.

Finally, we recall the appropriate notions of homotopy for the preceding categories $\Alg_\infty$ and $\Mod_\infty\g A$. 

\begin{definition}\label{D: homotopy for Mod-inf-A}
Let $A$ be an $A_\infty$-algebra and  $f,g:M\rightmap{}N$  morphisms of right 
$A_\infty$-modules over $A$. Then, a \emph{homotopy from $f$ to $g$}  is 
a family of maps $\{h_n\}_{n\in \hueca{N}}$  where, for each $n\in \hueca{N}$, 
$$h_n:M\otimes_SA^{\otimes (n-1)}\rightmap{} N$$
is a homogeneous morphism of right $S$-modules of degree $\vert h_n\vert=-n$,
such that, for all $n\in \hueca{N}$, we have  
$f_n-g_n=H_n^{(1)} + H_n^{(2)}+H_n^{(3)}$, where 
$$H_n^{(1)}=\sum_{\scriptsize\begin{matrix}r+s=n\\ r\geq 1;s\geq 0\end{matrix}} 
   (-1)^{rs}m^N_{1+s}( h_r\otimes id^{\otimes s}),$$
  $$H_n^{(2)}=\sum_{\scriptsize\begin{matrix}s+t=n\\ s\geq 1;t\geq 0\end{matrix}} 
   (-1)^{st}h_{1+t}( m^M_s\otimes id^{\otimes t}),$$
   and
   $$H_n^{(3)}=\sum_{\scriptsize\begin{matrix}r+s+t=n\\ r,s\geq 1;t\geq 0\end{matrix}} 
   (-1)^{r+st}h_{r+1+t}(id^{\otimes r}\otimes m_s\otimes id^{\otimes t}),$$
   
   A morphism of $A_\infty$-modules $f:M\rightmap{}N$ is called \emph{null-homotopic} iff there is a homotopy from $f$ to $0$.
\end{definition}
 
  \begin{definition}\label{D: homotopy for Alg-inf}
   Let $A$ and $B$ be $A_\infty$-algebras and let $f,g:A\rightmap{}B$
   be morphisms of $A_\infty$-algebras. A \emph{homotopy $h$ from $f$
   to $g$} is a family $h=\{h_n\}_{n\in \hueca{N}}$, where 
each 
$$h_n:A^{\otimes n}\rightmap{}B$$ 
is a homogeneous morphism of  $S\g S$-bimodules with degree $\vert h_n\vert=-n$,  
 such that, for each $n\in \hueca{N}$, 
   we have 
  $$f_n-g_n=H(h)_n+H_{f,g}(h)_n,$$
  where 
    $$ H(h)_n=\sum_{\scriptsize
  \begin{matrix} r+s+t=n\\
            r,t\geq 0; s\geq 1\\
         \end{matrix}}\hbox{\,}(-1)^{r+st}h_{r+1+t}(id^{\otimes r}\otimes m^A_s\otimes id^{\otimes t})
         $$
         and 
  $$H_{f,g}(h)_n=\sum_{\scriptsize
  \begin{matrix} 
            r,t\geq 0; s\geq 1\\
           i_1+\cdots +i_r+s\\
            +j_1+\cdots +j_t=n\\
         \end{matrix}}(-1)^{\sgn}
         m^B_{r+1+t}
         (f_{i_1}\otimes\cdots\otimes f_{i_r}\otimes h_s\otimes 
         g_{j_1}\otimes\cdots\otimes g_{j_t})$$
         where $i_1,\ldots,i_r,j_1,\ldots,j_t\geq 1$ and 
         $\sgn=\sgn(i_1,\ldots,i_r,s,j_1,\ldots,j_t)$
         is given by the sum 
         $$r(t+1)+st+\sum_{\scriptsize
  \begin{matrix} 
            2\leq \alpha \leq r\\
            1\leq u< \alpha\\
         \end{matrix}}(1-i_u)
         +
         t\sum_{u=1}^ri_u
         +
         \sum_{\scriptsize
  \begin{matrix} 
            2\leq \beta \leq t\\
            1\leq v< \beta\\
         \end{matrix}}(1-j_v).$$
 \end{definition}
 
 The preceding notions of homotopy are known to be equivalence relations on the corresponding categories and give rise to the \emph{homotopy categories} $\underline{\Mod}_\infty\g A$ and $\underline{\Alg}_\infty$, respectively. 
 
 The categories $\DGCoalg$ and $\DGComod^0\g {\cal B}_A$ have their own classical homotopy relations. 
 The functors 
 $$\Psi:\Alg_\infty\rightmap{}\DGCoalg \hbox{  \ and  \ }
 \Phi:\Mod_\infty\g A\rightmap{}\DGComod^0\g {\cal B}_A$$ preserve and reflect the preceding homotopy relations and, therefore, induce full and faithful functors $\underline{\Psi}:\underline{\Alg}_\infty\rightmap{}\underline{\DGCoalg}$ and 
 $\underline{\Phi}:\underline{\Mod}_\infty\g A\rightmap{}\underline{\DGComod}^0\g{\cal B}_A$.

 Many important properties of $\underline{\Alg}_\infty$ and $\underline{\Mod}_\infty\g A$ are derived, using the full and faithful functors $\Psi$ and $\Phi$, from the corresponding properties for $\underline{\DGCoalg}$ and $\underline{\DGComod}^0\g{\cal B}_A$ which are better understood categories.  
 
 In this paper, we propose that the use of the categories of twisted modules over differential graded  $S$-bocses can be  a fresh and simpler tool for some  studies of $\Mod_\infty\g A$. We illustrate 
 this with  three applications. 
 
 Namely, we can describe with quite good precision the structure of $\Mod_\infty\g A$ as a Frobenius category. In particular, we prove that idempotents split in $\Mod_\infty\g A$, a fact which seems to have remained unnoticed. We obtain this as an application of the study of the category $\TGM\g {\cal B}$, for a general triangular differential  graded $S$-bocs, see (\ref{T: TGMod-A estable es triangulada}). By ``general'' we mean that it is not necessarily a differential graded tensor $S$-coalgebra. We compare our description with the one given in \cite{K1}(Proposition 5.2 and 8.4).

 The important fact that each quasi-isomorphism of $A_\infty$-modules is a homotopy equivalence, see \cite{K1}(Theorem 4.2), is obtained here, for the case where $S$ is a finite product of copies of the ground field $k$, as a consequence of our theorem (\ref{P: (M,u0) aciclico sii (M,u) homotopicamente trivial}). The latter is proved by an induction argument using the 
 triangularity of the given triangular differential graded $S$-bocs. The proof does not resort to any model theoretical considerations. 
 
 Finally, we show the fact that any homotopy equivalence $f:A\rightmap{}B$ of $A_\infty$-algebras determines a restriction functor $R_f:\underline{\Mod}_\infty\g A\rightmap{}\underline{\Mod}_\infty\g A$ which is an equivalence of categories, see \cite{K1}(Proposition 6.2). We obtain this as a consequence of the corresponding result for triangular differential graded  $S$-bocses, see 
 (\ref{C: equiv homotopica de bocses triangulares induce equiv de cats homotópicas}), which essentially relies on the remarkable property that a morphism $f:M\rightmap{}N$ in the category of modules over a triangular differential graded $S$-bocs is an isomorphism if and only if its first component is so, see (\ref{C: f iso sii f0 es iso}). 
 
 The twisted modules we consider in (\ref{D: category of twisted graded modules of a bocs}) are constructed from the differential graded category $\GM\g {\cal B}$,  of graded modules over a differential graded $S$-bocs ${\cal B}$. These twisted modules are a special kind of the twisted complexes over a general differential graded category considered in \cite{BK}. Here we focus on a naive and very concrete approach to the study of the category $\TGM\g {\cal B}$ for a triangular differential graded $S$-bocs ${\cal B}$ and extract from this the preceding applications to the study of $A_\infty$-modules.

   \section{Differential graded bocses,  twisted modules}
   
   In this paper, we use the word \emph{bocs} in the following specific sense. 
   
   \begin{definition}\label{D: dgb y su cat de mods}
    A \emph{graded $S$-bocs} ${\cal B}$ is a triple ${\cal B}=(C,\mu,\epsilon)$, where $C$ is a graded $S$-$S$-bimodule, 
   so we have a decomposition $C=\bigoplus_{j\in \hueca{Z}}C_j$ as a direct 
 sum of $S$-$S$-bimodules, and $\mu:C\rightmap{}C\otimes_SC$ and $\epsilon:C\rightmap{}S$ are 
 homogeneous $S$-$S$-bimodule  maps of degree 0 such that 
 the following diagrams commute
 $$\begin{matrix}
 C&\rightmap{\mu}&C\otimes_SC\\
 \lmapdown{\mu}&&\lmapdown{\mu\otimes id_C}\\
 C\otimes_SC&\rightmap{id_C\otimes \mu}&C\otimes_SC\otimes_SC\\
 \end{matrix}
 \hbox{ and }
 \begin{matrix}
 C\otimes_SC&\leftmap{\mu}&C&\rightmap{\mu}&C\otimes_SC\\
 \rmapdown{id_C\otimes \epsilon}&&\rmapdown{id_C}&&\rmapdown{\epsilon\otimes id_C}\\
 C\otimes_SS&\rightmap{\rho_C}&C&\leftmap{\lambda_C}&S\otimes_SC\\
 \end{matrix}
 $$
where $\lambda_C$ and $\rho_C$ are the left and right $S$-multiplications on $C$, respectively. 
The $S$-$S$-bimodule $S$ is considered as a graded $S$-bimodule concentrated at 0. 

A \emph{coderivation $\delta$ on} 
 ${\cal B}$ is a homogeneous morphism of $S$-$S$-bimodules $\delta:C\rightmap{}C$, of degree 1, 
 such that the following square commutes.
 $$\begin{matrix}
    C&\rightmap{\mu}&C\otimes_SC\hfill\\
    \lmapdown{\delta}&&\rmapdown{id_C\otimes \delta+\delta\otimes id_C}\\
    C&\rightmap{\mu}&C\otimes_S C.\hfill\\
   \end{matrix}$$ 
 A coderivation  $\delta$ on ${\cal B}$ is called a \emph{differential}
 iff furthermore $\delta^2=0$.
 In this case, ${\cal B}=(C,\mu,\epsilon,\delta)$ is called a 
 \emph{differential graded $S$-bocs}. A 
 \emph{morphism $f:(C,\mu,\epsilon,\delta)\rightmap{}
 (C',\mu',\epsilon',\delta')$ of differential 
 graded $S$-bocses} 
 is a homogeneous morphism $f:C\rightmap{}C'$ of graded $S$-$S$-bimodules of degree $0$ such that $\mu' f=(f\otimes f)\mu$, 
 $\epsilon'f=\epsilon$ and $\delta'f=f\delta$. 
   \end{definition}
   
   So a (differential) graded $S$-bocs is exactly the same concept as a (differential) graded $S$-coalgebra. A morphism of (differential) graded $S$-bocses is the same as a morphism of (differential) graded $S$-coalgebras. 

   The following useful property is well known, but we include a proof for the sake of the reader. 
   
 \begin{lemma}\label{L: epsilon d = 0}
 For any differential graded $S$-bocs ${\cal B}=(C,\mu,\epsilon,\delta)$ we have $\epsilon\delta=0$.
\end{lemma}
 
 \begin{proof} We know that $\rho(id_C\otimes \epsilon)\mu=id_C$ and 
 $\lambda(\epsilon\otimes id_C)\mu=id_C$. Thus,
 for $c\in C$, if $\mu(c)=\sum_ic_i\otimes c'_i$,  we get 
 $$\sum_i\epsilon(c_i)c'_i=c=\sum_ic_i\epsilon(c'_i).$$
 Since $\delta$ is a coderivation, we have 
 $[(id_C\otimes \delta)+(\delta\otimes id_C)]\mu=\mu \delta$. As a consequence, 
 $\rho(id_C\otimes \epsilon)(id_C\otimes \delta)\mu+
 \rho(id_C\otimes \epsilon)(\delta\otimes id_C)\mu=
 \rho(id_C\otimes \epsilon)\mu \delta=\delta$. Hence, 
 $$\rho(id_C\otimes \epsilon \delta)\mu+
 \rho(\delta\otimes \epsilon)\mu=\delta.$$
 Evaluate at the element $c$ to obtain
 $\sum_i (-1)^{\vert c_i\vert}c_i\epsilon \delta(c'_i)+\sum_i\delta(c_i)\epsilon(c'_i)=\delta(c)$. 
 Since we also have 
 $$\delta(c)=\delta(\sum_ic_i\epsilon(c'_i))=\sum_i\delta(c_i)\epsilon(c'_i),$$
 we know that $\sum_i(-1)^{\vert c_i\vert}c_i\epsilon \delta(c'_i)=0$. Applying the morphism $\epsilon$,  we have 
 $\sum_i(-1)^{\vert c_i\vert}\epsilon(c_i) \epsilon \delta(c'_i)=0$. Since $\epsilon:C\rightmap{}S$ is homogeneous of degree $0$, we have $\epsilon(c_i)=0$ whenever $\vert c_i\vert\not=0$, so  we finally get 
 $$\epsilon \delta(c)=\epsilon \delta (\sum_i\epsilon(c_i)c_i')=
 \sum_i\epsilon(c_i)\epsilon \delta(c_i')=0.$$
 \end{proof}
   
   \begin{definition}\label{D: la cat GMod-B}
   Given a  graded $S$-bocs ${\cal B}=(C,\mu,\epsilon)$, we can consider the \emph{$k$-category of graded right ${\cal B}$-modules} denoted by $\GM\g {\cal B}$. A \emph{graded right ${\cal B}$-module} is by definition a graded right $S$-module.
   Given graded ${\cal B}$-modules $M$ and $N$, and $d\in \hueca{Z}$, a \emph{morphism $f:M\rightmap{}N$ of right ${\cal B}$-modules of degree $d$} is a homogeneous morphism of right $S$-modules $f:M\otimes_SC\rightmap{}N$ of degree $d$. 
   So, their set of morphisms in the category $\GM\g{\cal B}$ is by definition 
   $$\Hom_{\GM\g{\cal B}}(M,N):=\bigoplus_{d\in \hueca{Z}}\Hom^d_{\GM\g{\cal B}}(M,N)=\bigoplus_{d\in \hueca{Z}}\Hom^d_{\GM\g S}(M\otimes_SC,N).$$
   If we have a pair $f:M\rightmap{}N$, $g:N\rightmap{}L$ of composable morphisms in $\GM\g {\cal B}$, their composition in $\GM\g {\cal B}$ is defined as the following composition of morphisms of right $S$-modules
   $$g*f=\left(M\otimes_SC\rightmap{ \  \ id_M\otimes \mu \  \ }M\otimes_SC\otimes_SC\rightmap{ \  \ f\otimes id_C \  \ }N\otimes_SC\rightmap{g}L\right).$$
   \end{definition}
   
   It is not hard to show that $\GM\g {\cal B}$ is indeed a graded $k$-category (the unit morphism at a graded ${\cal B}$-module $M$, denoted by $\hueca{I}_M$, is  the morphism of right $S$-modules  
   $\rho_M(id_M\otimes \epsilon)$, where 
   $\rho_M:M\otimes_S S\rightmap{}M$ is the right multiplication on $M$. We will denote by $\GM^0\g{\cal B}$ the subcategory of 
    $\GM\g{\cal B}$ with the same objects but only the degree zero morphisms of $\GM\g {\cal B}$. 
   
   \begin{proposition}\label{P: DG-B es cat dif graduada}
   Let ${\cal B}=(C,\mu,\epsilon,\delta)$ be a differential graded $S$-bocs. 
   Given $M,N\in \GM\g {\cal B}$, define on the graded hom space 
   $$\Hom_{\GM\g {\cal B}}(M,N)=\bigoplus_{d\in \hueca{Z}}\Hom^d_{\GM\g {\cal B}}(M,N)$$
   the homogeneous linear map 
   $\hat{\delta}:\Hom_{\GM\g {\cal B}}(M,N)\rightmap{}\Hom_{\GM\g {\cal B}}(M,N)$ 
   of degree
   1 by the following recipe, for any homogeneous morphism 
   $f:M\rightmap{}N$, 
   $$\hat{\delta}(f)=(-1)^{\vert f\vert+1}f(id_M\otimes \delta).$$
   Then, the category $\GM\g B$ is a \emph{differential graded category}. Namely, 
   $\hat{\delta}^2=0$, $\hat{\delta}(\hueca{I}_M)=0$, for all $M\in \GM\g {\cal B}$, and the
   following Leibniz rule holds
   $$\hat{\delta}(g*f)=\hat{\delta}(g)*f+(-1)^{\vert g\vert}g*\hat{\delta}(f),$$
   for any homogeneous morphisms $f:M\rightmap{}N$ and $g:N\rightmap{}L$ in $\GM\g{\cal B}$.
   \end{proposition}

   \begin{proof} For this proof make $\hat{\underline{\delta}}(f)=f(id_M\otimes \delta)$. 
   Then, we have 
   $$\begin{matrix}
      \hat{\underline{\delta}}(g*f)&=&(g*f)(id_M\otimes \delta)\hfill\\
      &=&
      g(f\otimes id_C)(id_M\otimes \mu)(id_M\otimes \delta)\hfill\\
      &=&
       g(f\otimes id_C)(id_M\otimes \mu \delta)\hfill\\
      &=&
       g(f\otimes id_C)(id_M\otimes[\delta \otimes id_C+id_C\otimes \delta]\mu)\hfill\\
      &=&
       g(f\otimes id_C)(id_M\otimes \delta \otimes id_C)(id_M\otimes\mu)\hfill\\
       &&+\,g(f\otimes id_C)(id_M\otimes id_C\otimes \delta)(id_M\otimes\mu)\hfill\\
      &=&
      g(f(id_M\otimes \delta)\otimes id_C)(id_M\otimes\mu)\hfill\\
       &&+\,g(f(id_M\otimes id_C)\otimes \delta)(id_M\otimes\mu)\hfill\\
       &=&
          g(f(id_M\otimes \delta)\otimes id_C)(id_M\otimes\mu)\hfill\\
          &&+\,(-1)^{\vert f\vert}g(id_N\otimes \delta)(f\otimes id_C)(id_M\otimes \mu)\hfill\\
       &=&
        g*f(id_M\otimes \delta)+(-1)^{\vert f\vert}g(id_N\otimes \delta)*f\hfill\\
       &=&
       g*\hat{\underline{\delta}}(f)+(-1)^{\vert f\vert}\hat{\underline{\delta}}(g)*f.\hfill\\
       \end{matrix}$$
       Then,
       $$\begin{matrix}
          \hat{\delta}(g*f)&=&(-1)^{\vert f\vert+\vert g\vert+1}\hat{\underline{\delta}}(g*f)\hfill\\
       &=&
       (-1)^{\vert f\vert+\vert g\vert+1}[g*\hat{\underline{\delta}}(f)
       +
        (-1)^{\vert f\vert}\hat{\underline{\delta}}(g)*f]\hfill\\
       &=&
       \hat{\delta}(g)*f+(-1)^{\vert g\vert}g*\hat{\delta}(f).\hfill\\
         \end{matrix}$$
   \end{proof}

   \begin{definition}\label{D: category of twisted graded modules of a bocs}
   Let ${\cal B}=(C,\mu,\epsilon,\delta)$  be a differential graded $S$-bocs.
   Then, a \emph{twisted 
   ${\cal B}$-module} is a pair $(M,u)$, where $M\in \GM\g {\cal B}$ and $u$ is a homogeneous 
   morphism $u\in \Hom_{\GM\g{\cal B}}^1(M,M)$  such that the following 
   \emph{Maurer-Cartan equation for $u$} holds  
   $$\hat{\delta}(u)+u*u=0.$$
   If $(M,u)$ and $(N,v)$ are twisted ${\cal B}$-modules, then a 
   \emph{homogeneous morphism of twisted ${\cal B}$-modules $f:(M,u)\rightmap{}(N,v)$ of degree $d$} is a 
   homogeneous morphism $f:M\rightmap{}N$  in $\GM\g {\cal B}$ of degree $d$ such that
   $$\hat{\delta}(f)+v*f-(-1)^df*u=0.$$
   
   We will denote by $\Hom_{\GM\g {\cal B}}^d((M,u),(N,v))$ the space of  homogeneous morphisms of twisted 
   ${\cal B}$-modules of degree $d$. Moreover, we make 
   $$\Hom_{\TGM\g{\cal B}}((M,u),(N,v)):=\bigoplus_{d\in \hueca{Z}}\Hom^d_{\TGM\g{\cal B}}((M,u),(N,v)).$$ 
   \end{definition}

   \begin{lemma}\label{L: la cat de  modulos torcidos}
   With the preceding notations, we can form 
   \emph{the graded category of twisted ${\cal B}$-modules $\TGM\g{\cal B}$} with objects
   the twisted ${\cal B}$-modules, morphisms of twisted ${\cal B}$-modules, and 
   the same composition $*$ of the category $\GM\g {\cal B}$.
   \end{lemma}

   \begin{proof} Let $f:(M,u)\rightmap{}(N,v)$ and $g:(N,v)\rightmap{}(L,w)$ be
   homogeneous morphisms of twisted ${\cal B}$-modules. 
   We have $\hat{\delta}(f)+v*f-(-1)^{\vert f\vert} f*u=0$ 
   and $\hat{\delta}(g)+w*g-(-1)^{\vert g\vert}g*v=0$. 
   Then, we have
   $$\hat{\delta}(g*f)+w*g*f-(-1)^{\vert g\vert +\vert f\vert}g*f*u =$$
   $$\hat{\delta}(g)*f+(-1)^{\vert g\vert}g*\hat{\delta}(f)+w*g*f-(-1)^{\vert g\vert +\vert f\vert}g*f*u=$$
   $$\hat{\delta}(g)*f+(-1)^{\vert g\vert}g*[(-1)^{\vert f\vert}f*u-v*f]+w*g*f-(-1)^{\vert g\vert+\vert f\vert}g*f*u=$$
   $$\hat{\delta}(g)*f+w*g*f-(-1)^{\vert g\vert}g*v*f=$$
   $$[\hat{\delta}(g)+w*g-(-1)^{\vert g\vert}g*v]*f=0.$$
   \end{proof}

   \begin{proposition}\label{P: restrictions on graded modules}
     Let ${\cal B}_1=(C_1,\mu_1,\epsilon_1,\delta_1)$ and 
     ${\cal B}_2=(C_2,\mu_2,\epsilon_2,\delta_2)$  be  differential graded $S$-bocses 
     and $\psi:{\cal B}_1\rightmap{}{\cal B}_2$  a morphism of differential 
     graded $S$-bocses. Then, there is a functor of differential graded categories 
     (called 
     \emph{the restriction functor associated to the morphism $\psi$})
     $$R_\psi:\GM\g{\cal B}_2\rightmap{}\GM\g {\cal B}_1$$
     such that $R_\psi(M)=M$, for $M\in \GM\g{\cal B}_2$, and 
     $R_\psi(f)=f(id_M\otimes \psi)$, for any morphism 
     $f:M\rightmap{}N$ of $\GM\g {\cal B}_2$.
   \end{proposition}
   
   \begin{proof} The identity morphism of $M\in \GM\g {\cal B}_2$ is 
   $\rho(id_M\otimes \epsilon_2):M\otimes_SC_2\rightmap{}M$, 
   where $\rho:M\otimes_SS\rightmap{}M$ is the multiplication map. 
   Then, 
   $$R_\psi(\rho(id_M\otimes \epsilon_2))
   =
   \rho(id_M\otimes \epsilon_2)(id_M\otimes \psi)
   =
   \rho(id_M\otimes \epsilon_2\psi)= \rho(id_M\otimes \epsilon_1).
  $$
   So, $R_\psi$ preserves identities.  
   
   Given composable morphisms $f:M\rightmap{}N$
   and $g:N\rightmap{}L$ in $\GM\g {\cal B}_2$, we have 
   $$\begin{matrix}
R_\psi(g*f)
&=&
(g*f)(id_M\otimes \psi)\hfill\\
&=&
g(f\otimes id_{C_2})(id_M\otimes \mu_2)(id_M\otimes \psi)\hfill\\
&=&
g(f\otimes id_{C_2})(id_M\otimes \mu_2\psi)\hfill\\
&=&
g(f\otimes id_{C_2})(id_M\otimes(\psi\otimes\psi)\mu_1)\hfill\\
&=&
g(f\otimes id_{C_2})(id_M\otimes(\psi\otimes\psi))(id_M\otimes \mu_1)\hfill\\
&=&
g(id_N\otimes \psi)(f(id_M\otimes \psi)\otimes id_{C_1})(id_M\otimes \mu_1)\hfill\\
&=&
g(id_N\otimes \psi)*f(id_M\otimes \psi)\hfill\\
&=&
R_\psi(g)*R_\psi(f).\hfill\\
\end{matrix}$$
Finally, given a homogeneous morphism $f:M\rightmap{}N$ in $\GM\g {\cal B}_2$, 
we have
$$\begin{matrix}
   R_\psi(\hat{\delta}_2(f))
   &=&
   (-1)^{\vert f\vert +1}f(id_M\otimes \delta_2)(id_M\otimes \psi)\hfill\\
   &=&
   (-1)^{\vert f\vert +1}f(id_M\otimes \delta_2\psi)\hfill\\
   &=&
    (-1)^{\vert f\vert +1}f(id_M\otimes \psi \delta_1)\hfill\\
   &=&
   (-1)^{\vert R_\psi(f)\vert +1}f(id_M\otimes \psi)(id_M\otimes \delta_1)=
   \hat{\delta}_1(R_\psi(f)).\hfill\\
  \end{matrix}$$
\end{proof} 
   
    \begin{proposition}\label{P: restrictions on twisted graded modules}
     Let ${\cal B}_1=(C_1,\mu_1,\epsilon_1,\delta_1)$ and 
     ${\cal B}_2=(C_2,\mu_2,\epsilon_2,\delta_2)$  be  differential graded $S$-bocses 
     and $\psi:{\cal B}_1\rightmap{}{\cal B}_2$  a morphism of differential 
     graded $S$-bocses. Then, the restriction functor 
     $R_\psi:\GM\g{\cal B}_2\rightmap{}\GM\g{\cal B}_1$ induces a $k$-functor 
     (called again 
     \emph{the restriction functor associated to the morphism $\psi$})
     $$R_\psi:\TGM\g{\cal B}_2\rightmap{}\TGM\g {\cal B}_1$$
     such that $R_\psi(M,u)=(M,R_\psi(u))$, for $(M,u)\in \TGM\g{\cal B}_2$, and 
     $R_\psi(f)=f(id_M\otimes \psi)$, for any morphism 
     $f:(M,u)\rightmap{}(N,v)$ of $\TGM\g {\cal B}_2$.
   \end{proposition}
   
   \begin{proof} Given $(M,u)\in \TGM\g {\cal B}_2$, we have 
   $$\hat{\delta}_1(R_\psi(u))+R_\psi(u)*R_\psi(u)=
   R_\psi(\hat{\delta}_2(u)+u*u)=0,$$
   so, $(M,R_\psi(u))\in \TGM\g{\cal B}_1$. 
   
       Given a homogeneous morphism $f:(M,u)\rightmap{}(N,v)$ in $\TGM\g{\cal B}_2$,
    we have that 
    $$\hat{\delta}_1(R_\psi(f))+R_\psi(v)*R_\psi(f)-
    (-1)^{\vert R_\psi(f)\vert}R_\psi(f)*R_\psi(u)$$
    coincides with 
    $$R_\psi(\hat{\delta}_2(f)+v*f-(-1)^{\vert f\vert}f*u)=0.$$
    So $R_\psi(f):(M,R_\psi(u))\rightmap{}(N,R_\psi(v))$ is a homogeneous 
    morphism of twisted ${\cal B}_1$-modules. 
   \end{proof}

   \section{Bocses, coalgebras and homotopy}
   
  Let us see how that the preceding notions relate to the category of comodules and differential comodules over a given differential graded $S$-bocs ${\cal B}$. We first recall some basic notions.  
   
   \begin{definition}\label{D: GComod-B}
   Given a graded $S$-coalgebra ${\cal B}=(C,\mu,\epsilon)$, that is a graded $S$-bocs, we will denote by $\GC\g {\cal B}$ \emph{the category of the graded right  ${\cal B}$-comodules.} Recall that a \emph{graded right ${\cal B}$-comodule} is a 
   pair  $(M,\mu_M)$, where $M$ is a graded right $S$-module and $\mu_M:M\rightmap{}M\otimes_SC$ is a homogeneous morphism of right $S$-modules of degree 0 such that the following diagrams commute 
 $$\begin{matrix}
 M&\rightmap{\mu_M}&M\otimes_SC\\
 \lmapdown{\mu_M}&&\lmapdown{\mu_M\otimes id_C}\\
 M\otimes_SC&\rightmap{id_M\otimes \mu}&M\otimes_SC\otimes_SC\\
 \end{matrix}
 \hbox{ and \ }
 \begin{matrix}
 M\otimes_SC&\leftmap{\mu_M}&M\\
 \rmapdown{id_M\otimes \epsilon}&&\rmapdown{id_M}\\
 M\otimes_SS&\rightmap{\rho_M}&M.\\
 \end{matrix}
 $$
 If $(M,\mu_M)$ and $(N,\mu_N)$ are graded right ${\cal B}$-comodules and $d\in \hueca{Z}$, 
 a \emph{morphism of graded right ${\cal B}$-comodules}
 $h:(M,\mu_M)\rightmap{}(N,\mu_N)$ \emph{of degree $d$} is a homogeneous morphism 
 of right $S$-modules
 $h:M\rightmap{}N$ with degree $d$, such that 
 the following square commutes 
 $$\begin{matrix}
 M&\rightmap{\mu_M}&M\otimes_SC\\
 \lmapdown{h}&&\rmapdown{h\otimes id_C}\\
 N&\rightmap{\mu_N}&N\otimes_SC.\\
 \end{matrix}$$
 We denote by $\Hom^d_{\GC\g {\cal B}}(M,N)$ the space of morphisms of graded
 right ${\cal B}$-comodules of degree $d$ from $M$ to $N$ and we make 
 $$\Hom_{\GC\g {\cal B}}(M,N):=\bigoplus_{d\in \hueca{Z}}\Hom^d_{\GC\g {\cal B}}(M,N).$$
 We denote by $\GC^0\g {\cal B}$ the subcategory of $\GC\g {\cal B}$ with the same objects and only degree zero morphisms. 
 \end{definition}
   
 \begin{definition}\label{D: DGComod-B}
 Given a differential graded $S$-coalgebra ${\cal B}=(C,\mu,\epsilon,\delta)$, 
 that is a differential graded $S$-bocs, we will denote by $\DGC\g {\cal B}$ 
 \emph{the graded category of the differential graded right  ${\cal B}$-comodules.} 
 Recall that a 
 \emph{differential graded right ${\cal B}$-comodule} is a 
   triple  $(M,\mu_M,\delta_M)$, where $(M,\mu_M)$ is a graded right ${\cal B}$-comodule and $\delta_M:M\rightmap{}M$ is a differential on $(M,\mu_M)$. Recall that a map $\delta_M:M\rightmap{}M$ is a \emph{coderivation on $(M,\mu_M)$} iff   it is a  homogeneous morphism of graded right $S$-modules of degree $1$ such that the following diagram commutes 
 $$\begin{matrix}
    M&\rightmap{\mu_M}&M\otimes_SC\hfill\\
    \lmapdown{\delta_M}&&\rmapdown{id_M\otimes \delta+\delta_M\otimes id_C}\\
    M&\rightmap{\mu_M}&M\otimes_S C.\hfill\\
   \end{matrix}$$
   Such a map $\delta_M$ is called a \emph{differential} if, furthermore, $\delta_M^2=0$.

 Whenever $(M,\mu_M,\delta_M)$ and $(N,\mu_N,\delta_N)$ are differential graded right 
 ${\cal B}$-como\-dules and $d\in \hueca{Z}$, we agree that 
 a \emph{morphism of differential graded right ${\cal B}$-como\-dules}
 $h:(M,\mu_M,\delta_M)\rightmap{}(N,\mu_N,\delta_N)$ \emph{of degree $d$} is a homogeneous morphism 
 of graded right ${\cal B}$-comodules
 $h:M\rightmap{}N$ with degree $d$, such that $\delta_Nh=(-1)^dh\delta_M$. 
 We denote by $\Hom^d_{\DGC\g {\cal B}}(M,N)$ the space of homogeneous morphisms of differential graded
 right ${\cal B}$-comodules of degree $d$ from $M$ to $N$ and we make  
 $$\Hom_{\DGC\g {\cal B}}(M,N):=\bigoplus_{d\in \hueca{Z}}\Hom^d_{\DGC\g {\cal B}}(M,N).$$
 As before, with $\DGC^0\g {\cal B}$ we denote the subcategory of $\DGC\g {\cal B}$ with the same objects but only degree zero morphisms. 
 \end{definition}

   \begin{remark}\label{L: h determina morfismo dNh-(-1)grad h hd_N}
   Let ${\cal B}=(C,\mu,\epsilon,\delta)$ be a differential graded  $S$-coalgebra, $M,N\in  \DGC\g {\cal B}$,
   and $h:M\rightmap{}N$ a homogeneous morphism of graded right ${\cal B}$-comodules. Then,
   $$\delta_Nh-(-1)^{\vert h\vert}h\delta_M\in \Hom^{\vert h\vert+1}_{\DGC\g {\cal B}}(M,N).$$
   \end{remark}

   \begin{proposition}\label{P: comods coalgebras <--F--> mods bocses}
   Let ${\cal B}=(C,\mu,\epsilon,\delta)$ be a 
    differential graded $S$-bocs. 
    \begin{enumerate}
     \item There is a full and faithful functor of graded categories
     $$\Phi: \GM\g {\cal B}\rightmap{}\GC\g {\cal B}$$
     such that $\Phi(M)=\Ind_{\cal B}(M):=(M\otimes_SC,id_M\otimes \mu)$, for $M\in \GM\g {\cal B}$, and, for 
     any $f\in \Hom_{\GM\g{\cal B}}(M,N)$, 
     we have  $\Phi(f)=(f\otimes id_C)(id_M\otimes \mu)$.   
     \item There is a full and faithful functor of graded categories
     $$\Phi: \TGM\g {\cal B}\rightmap{}\DGC\g {\cal B}$$
     such that $\Phi(M,u)=(\Ind_{\cal B}(M),\delta_{\Phi(M,u)})$, where
     $\delta_{\Phi(M,u)}=id_M\otimes \delta+ \Phi(u)$, 
     for each $(M,u)\in \TGM\g {\cal B}$, and 
     $\Phi(f)=(f\otimes id_C)(id_M\otimes \mu)$,  for any $f\in \Hom_{\TGM\g{\cal B}}((M,u),(N,v))$. 
    \end{enumerate}
   \end{proposition}

   \begin{proof} (1): It is easy to see that the association $\Phi$ is well defined. 
In order to show that $\Phi$ preserves composition, 
take $f\in \Hom_{\GM\g {\cal B}}(M,N)$ and 
$g\in \Hom_{\GM\g{\cal B}}(N,L)$. 
Then, 
$$\begin{matrix}
  \Phi(g)\Phi(f)
  &=&
    (g\otimes id_C)(id_N\otimes\mu)(f\otimes id_C)(id_M\otimes\mu)\hfill\\
  &=&
     (g\otimes id_C)(f\otimes\mu)(id_M\otimes\mu)\hfill\\
  &=&
       (g\otimes id_C)(f\otimes id_C\otimes id_C)(id_M\otimes id_C\otimes\mu)(id_M\otimes\mu)\hfill\\
  &=&
       (g\otimes id_C)(f\otimes id_C\otimes id_C)(id_M\otimes(id_C\otimes\mu)\mu)\hfill\\
  &=&
         (g\otimes id_C)(f\otimes id_C\otimes id_C)(id_M\otimes(\mu\otimes id_C)\mu)\hfill\\
  &=&
        (g\otimes id_C)(f\otimes id_C\otimes id_C)(id_M\otimes \mu\otimes id_C)(id_M\otimes \mu)\hfill\\
  &=&
         (g(f\otimes id_C)(id_M\otimes \mu)\otimes id_C)(id_M\otimes \mu)\hfill\\
  &=&
       (g*f\otimes id_C)(id_M\otimes \mu) = \Phi(g*f).\hfill\\
  \end{matrix}$$
Given $M\in \GM\g{\cal B}$, we consider the identity morphism 
$\hueca{I}_{M}=\rho_M(id_M\otimes \epsilon)$ in the category $\GM\g{\cal B}$. Then, 
$$\begin{matrix}
   \Phi(\hueca{I}_M)
   &=&
   (\rho_M(id_M\otimes \epsilon)\otimes id_C)(id_M\otimes \mu)\hfill\\
  
   &=&(id_M\otimes \lambda(\epsilon\otimes id_C))(id_M\otimes \mu)\hfill\\
   &=& id_M\otimes \lambda(\epsilon\otimes id_C)\mu=id_M\otimes id_C\hfill\\
   &=&id_{\Ind_{\cal B}(M)}=id_{\Phi(M)}.\hfill\\
  \end{matrix}$$
  Thus $\Phi$ is a degree preserving functor.  It is full and faithful because, for $d\in \hueca{Z}$,  
 $$\Phi_d:\Hom_{\GM\g S}^d(M\otimes_S C,N)\rightmap{}\Hom_{\GC\g {\cal B}}^d(\Ind_{\cal B}(M),\Ind_{\cal B}(N))$$
is an isomorphism, with inverse $\Phi'_d$ given by $\Phi'_d(\phi)=q_N\phi$, where $q_N:=\rho_N(id_N\otimes \epsilon).$

    (2): Take $(M,u)\in \TGM\g {\cal B}$. If we denote by $\Coder(\Ind_{\cal B}(M))$ the set of coderivations of $\Ind_{\cal B}(M)$, we have a bijection 
    $$\Theta:\Hom_{\GC\g {\cal B}}^1(\Ind_{\cal B}(M),\Ind_{\cal B}(M))\rightmap{}\Coder(\Ind_{\cal B}(M))$$
    given by $\Theta(h)=id_M\otimes \delta+h$. Then, we know that $\delta_{\Phi(M,u)}=\Theta\Phi_1(u)$ is a coderivation on $\Ind_{\cal B}(M)$. Since $\delta_{\Phi(M,u)}^2$ is a homogeneous morphism of graded right ${\cal B}$-comodules of degree $2$, to show that $\delta_{\Phi(M,u)}^2=0$ is equivalent  to show that 
    $q_M\delta_{\Phi(M,u)}^2=0$. 
    
    From (\ref{L: epsilon d = 0}),  we have 
$$\begin{matrix}
q_M\delta_{\Phi(M,u)}^2
&=&
q_M[(id_M\otimes \delta)\Phi(u)+\Phi(u)(id_M\otimes \delta)+\Phi(u)\Phi(u)]\hfill\\
&=&
u(id_M\otimes \delta)+u\Phi(u)=\hat{\delta}(u)+u*u=0.\hfill\end{matrix}$$

Now, take  a homogeneous morphism $f:(M,u)\rightmap{}(N,v)$ in $\TGM\g {\cal B}$. 
We already know that $\Phi(f):\Ind_{\cal B}(M)\rightmap{}\Ind_{\cal B}(N)$ is a 
homogeneous morphism of graded ${\cal B}$-comodules. By definition, the map 
$\Phi(f):\Phi(M,u)\rightmap{}\Phi(N,v)$ is a morphism of differential 
graded ${\cal B}$-comodules 
if and only if the following difference $D$ is zero. 
$$D:=(id_N\otimes \delta+\Phi(v))\Phi(f)-(-1)^{\vert f\vert}\Phi(f)(id_M\otimes \delta+\Phi(u)).$$
By (\ref{L: h determina morfismo dNh-(-1)grad h hd_N}), we have that $D=0$ is 
equivalent to $q_ND=0$. Since $f$ is a morphism of twisted ${\cal B}$-modules, we have 
$$\begin{matrix}
  q_ND&=&v\Phi(f)-(-1)^{\vert f\vert}f(id_M\otimes \delta)-(-1)^{\vert f\vert}f\Phi(u)\hfill\\
  &=&
  \hat{\delta}(f)+v\Phi(f)-(-1)^{\vert f\vert}f\Phi(u)\hfill\\
  &=&
  \hat{\delta}(f)+v*f-(-1)^{\vert f\vert}f*u=0.\hfill\\
  \end{matrix}$$
Hence $\Phi$ is a well defined faithful functor. In order to show that it is a full functor, 
we use again the fact that  $\Phi_d$ is  a bijection and reverse the preceding argument.  
   \end{proof}

 \begin{definition}\label{D: homotopy for morphisms of twisted cal B-modules}
  Let ${\cal B}=(C,\mu,\epsilon,\delta)$ be a differential graded $S$-bocs. Given morphisms
  $f,g\in \Hom^0_{\TGM\g{\cal B}}((M,u),(N,v))$, \emph{a homotopy  $h$ from $f$ to $g$} is a morphism 
  $h\in \Hom_{\GM\g{\cal B}}^{-1}(M,N)$ such that 
  $$f-g=\hat{\delta}(h)+v*h+h*u.$$
  A morphism 
  $f\in \Hom^0_{\TGM\g{\cal B}}((M,u),(N,v))$ is \emph{null-homotopic} iff there is a homotopy from $f$ to $0$.

 We denote by $\TGM^0\g {\cal B}$ the subcategory $\TGM\g {\cal B}$ with the same objects and only zero degree homogeneous morphisms. 
 Thus, the notion of homotopy is an equivalence relation in the category $\TGM^0\g {\cal B}$.
  \end{definition}
  
  \begin{proposition}\label{P: equiv de GModB --> GC-C preserva homotopia}
   Let ${\cal B}=(C,\mu,\epsilon,\delta)$ be a differential graded $S$-bocs. 
   Consider the full and faithful functor  $\Phi:\TGM\g {\cal B}\rightmap{}\DGC\g {\cal B}$ of (\ref{P: comods coalgebras <--F--> mods bocses}). Then, 
   a morphism $f\in \Hom^0_{\TGM\g{\cal B}}((M,u),(N,v))$ in 
   $\TGM^0\g {\cal B}$ is 
   is null-homotopic iff $\Phi(f):\Phi(M,u)\rightmap{}\Phi(N,v)$ is null-homotopic in $\DGC^0\g {\cal B}$. 
   As a consequence, there is an induced  full and faithful functor on the homotopy categories 
   $$\underline{\Phi}:\underline{\TGM}^0\g {\cal B}\rightmap{}\underline{\DGC}^0\g {\cal B}.$$
  \end{proposition}

   \begin{proof} Recall that, by definition,  $\Phi(M,u)=(\Ind_{\cal B}(M),id_M\otimes \delta+\Phi(u))$ and  
   $\Phi(N,v)=(\Ind_{\cal B}(N),id_N\otimes \delta+\Phi(v))$. The morphism  $\Phi(f)$ is
   null-homotopic in $\DGC^0\g {\cal B}$ iff there is a morphism $\Phi(h)\in \Hom_{\GC\g {\cal B}}^{-1}(\Ind_{\cal B}(M),\Ind_{\cal B}(N))$ with 
   $$\Phi(f)=(id_N\otimes \delta+\Phi(v))\Phi(h)+\Phi(h)(id_M\otimes \delta+\Phi(u)).$$
   This is equivalent to 
   $$q_N\Phi(f)=q_N[(id_N\otimes \delta+\Phi(v))\Phi(h)+\Phi(h)(id_M\otimes \delta+\Phi(u))],$$
   and this is equivalent to 
    $$f=v\Phi(h)+h(id_M\otimes \delta)+h\Phi(u)=\hat{\delta}(h)+v*h+h*u,$$
    that is, $f$ is null-homotopic in $\TGM^0\g {\cal B}$.     
   \end{proof}
   
   \begin{lemma}\label{L: restriction functors and homotopy}
      Let ${\cal B}_1=(C_1,\mu_1,\epsilon_1,\delta_1)$ and 
     ${\cal B}_2=(C_2,\mu_2,\epsilon_2,\delta_2)$  be  differential graded $S$-bocses 
     and $\psi:{\cal B}_1\rightmap{}{\cal B}_2$  a morphism of differential 
     graded $S$-bocses. Then, the restriction functor 
     $R_\psi:\TGM^0\g{\cal B}_2\rightmap{}\TGM^0\g{\cal B}_1$ maps null-homotopic 
     morphisms onto null-homotopic morphisms. Hence it induces a $k$-functor on the corresponding 
     homotopy categories 
     $$\underline{R}_\psi:\underline{\TGM}^0\g{\cal B}_2\rightmap{}\underline{\TGM}^0\g {\cal B}_1$$
   \end{lemma}

   \begin{proof} Given a null-homotopic morphism 
   $f\in \Hom^0_{\TGM\g{\cal B}_2}((M,u),(N,v))$, there is 
   $h\in \Hom_{\GM\g{\cal B}_2}^{-1}(M,N)$ with $f=\hat{\delta}_2(h)+v*h+h*u$. Then, applying $R_\psi$, we obtain 
   $$\begin{matrix}
     R_\psi(f)&=&R_\psi( \hat{\delta}_2(h))+R_\psi(v*h)+R_\psi(h*u)\hfill\\
     &=&
     \hat{\delta}_1(R_\psi(h))+R_\psi(v)*R_\psi(h)+R_\psi(h)*R_\psi(u).\hfill\\
     \end{matrix}$$
Hence, $R_\psi(f)$ is null-homotopic.    
   \end{proof}

\section{Triangular bocses}   
  
 In this section we show some properties of the category $\TGM\g {\cal B}$, where ${\cal B}$ is a differential graded $S$-bocs  of a special type, which we call triangular. We stress the fact that some proofs are inspired in the 
 study of differential graded tensor algebras (or ditalgebras) and their module categories initiated by the Kiev school of representation theory, see \cite{BSZ}. 
  
 \begin{definition}\label{D: normal bocs}
 A \emph{normal graded $S$-bocs} ${\cal B}$ is a differential unitary graded $S$-coalgebra
 $(C,\mu,\epsilon,\delta)$, with 
 an $S$-$S$-bimodule decomposition $C=S\bigoplus \overline{C}$ 
such that $\epsilon:C\rightmap{}S$ is the projection map and $\mu(1)=1\otimes 1$. 
 \end{definition}
 
 \begin{lemma}\label{L: la coalgebra restringida}
 If ${\cal B}=(C,\mu,\epsilon,\delta)$ is a normal graded $S$-bocs, then we have a well defined map 
 $$\overline{\mu}:\overline{C}\rightmap{}\overline{C}\otimes_S\overline{C}\hbox{  \ with   \ } 
 \overline{\mu}(x)=\mu(x)-1\otimes x-x\otimes 1,$$
 and the differential  $\delta$ restricts to a map $\overline{\delta}:\overline{C}\rightmap{}\overline{C}$
 of degree $1$. Moreover, the triple $\overline{\cal B}=(\overline{C},\overline{\mu},\overline{\delta})$ is a differential graded 
 $S$-coalgebra (without counit). We call $\overline{\cal B}$ the \emph{reduced differential graded $S$-bocs of} ${\cal B}$. 
 \end{lemma}
 
 \begin{proof} See, for instance, \cite{BSZ}(3.2).
 \end{proof}

The following reinterpretation of the $k$-category $\GM\g {\cal B}$ is useful.

\begin{definition}\label{D: hueca(M)od-B}
 Given a normal graded $S$-bocs ${\cal B}=(C,\mu,\epsilon,\delta)$, we can form the following graded category $\GModi\g{\cal B}$. Its objects are the graded right $S$-modules and, given $d\in \hueca{Z}$, a \emph{morphism $f:M\rightmap{}N$ of degree $d$} in $\GModi\g{\cal B}$ is a pair of maps $f=(f^0,f^1)$,  where $f^0:M\rightmap{}N$ and $f^1:M\otimes_S \overline{C}\rightmap{}N$ are homogeneous morphisms  of right $S$-modules of degree $d$. Thus, its hom spaces are given by 
 $$\Hom_{\GModi\g {\cal B}}(M,N)=\Hom_{\GM\g S}(M,N)\times\Hom_{\GM\g S}(M\otimes_S\overline{C},N).$$
 If $f=(f^0,f^1)\in \Hom_{\GModi\g{\cal B}}(M,N)$ and $g=(g^0,g^1)\in\Hom_{\GModi\g{\cal B}}(N,L)$, 
their composition $g*f=((g*f)^0,(g*f)^1)\in 
\Hom_{\GModi\g {\cal B}}(M,L)$ is defined 
by $(g*f)^0=g^0f^0$ and $(g*f)^1$ is the following composition in $\GM\g S$
$$(g*f)^1=g^1(f^0\otimes id_{\overline{C}})+g^0f^1+
g^1[(f^1\otimes id_{\overline{C}})(id_M\otimes \overline{\mu})].$$
\end{definition}

It is not hard to show that $\GModi\g {\cal B}$ is indeed a graded $k$-category, 
where the identity morphism $\hueca{I}_M$ on each $M\in \GModi\g{\cal B}$ is just $(id_M,0)$. Again, with $\GModi^0\g{\cal B}$, we denote the subcategory of $\GModi\g{\cal B}$ with the same objects but with only zero degree morphisms. 

\begin{lemma}\label{L: equiv directa entre GMod-B y G hueca(M)od-B}
 For any normal graded $S$-bocs  ${\cal B}$, there is an equivalence of graded categories 
 $$F:\GM\g{\cal B}\rightmap{}\GModi\g {\cal B}$$
 such that $F(M)=M$, for each $M\in \GM\g {\cal B}$ and,
 given any morphism $f\in \Hom_{\GM\g{\cal B}}(M,N)=\Hom_{\GM\g S}(M\otimes_S C,N)$, we have 
 $F(f)=(f^0,f^1)$, where $f^0:M\rightmap{} N$ is given by $f^0(m)=f(m\otimes 1)$, for $m\in M$,
 and the morphism 
 $f^1:M\otimes_S\overline{C}\rightmap{}N$ is just the restriction of the map $f$.
\end{lemma}

\begin{proof} It is clear that $F$ defined as above determines a linear isomorphism  
$$\Hom_{\GM\g{\cal B}}(M,N)\rightmap{}\Hom_{\GModi\g{\cal B}}(M,N).$$
So we only have to show that $F$ preserves compositions and identities. 
Assume that $f:M\rightmap{}N$ and $g:N\rightmap{}L$ are morphisms in $\GM\g {\cal B}$. Then, 
for $m\in M$, we have 
$$\begin{matrix}
  F(g*f)^0(m)&=&
  g(f\otimes id_C)(id_M\otimes \mu)(m\otimes 1)\hfill\\
  &=&
  g(f\otimes id_C)(m\otimes 1\otimes 1)\hfill\\
  &=&
  g(f(m\otimes 1)\otimes 1)=g^0f^0(m).\hfill\\
  \end{matrix}$$
For $x\in \overline{C}$, we have  $\mu(x)=\overline{\mu}(x)+1\otimes x+x\otimes 1$
and $\overline{\mu}(x)=\sum_jy_j\otimes z_j$, so 
$$\begin{matrix}
  F(g*f)^1(m\otimes x)
  &=& 
  g(f\otimes id_C)(id_M\otimes \mu)(m\otimes x)\hfill\\
  &=&
  g(f\otimes id_C)(m\otimes \mu(x))\hfill\\
  &=&
  g(f\otimes id_C)(m\otimes \overline{\mu}(x))\hfill\\
  &&+\,
  g(f\otimes id_C)(m\otimes x\otimes 1)+g(f\otimes id_C)(m\otimes 1\otimes x)\hfill\\
   &=&
  \sum_jg(f(m\otimes y_j)\otimes z_j)\hfill\\
  &&+\,
  g(f(m\otimes x)\otimes 1)+g(f(m\otimes 1)\otimes x)\hfill\\
  &=&
  \sum_jg^1(f^1(m\otimes y_j)\otimes z_j)
  +
  g^0(f^1(m\otimes x))+g^1(f^0(m)\otimes x)\hfill\\
  &=&
  [g^1(f^1\otimes id_{\overline{C}})(id_M\otimes \overline{\mu})+g^0f^1+
  g^1(f^0\otimes id_{\overline{C}})](m\otimes x)\hfill\\
  &=&
  [F(g)*F(f)]^1(m\otimes x).\hfill\\
  \end{matrix}$$
Thus, we obtain that $F(g*f)=F(g)*F(f)$. Given $M\in \GM\g {\cal B}$, recall that 
$\hueca{I}_M=\rho_M (id_M\otimes \epsilon)$, thus 
$F(\hueca{I}_M)^0(m) =\rho(id_M\otimes \epsilon)(m\otimes 1)=m.$ 
  For $x\in \overline{C}$, we have 
 $F(\hueca{I}_M)^1(m\otimes x)= \rho(id_M\otimes \epsilon)(m\otimes x)=0$. Hence, 
 $F(\hueca{I}_M)=(id_M,0)=\hueca{I}_{F(M)}$. 
\end{proof}

\begin{remark} Given $f:M\rightmap{}N$ in $\GM\g {\cal B}$, 
we often refer to the maps $f^0$ and $f^1$ in $F(f)=(f^0,f^1)$
 as \emph{the components of $f$}. 

Notice that any morphism $(f^0,f^1)\in \Hom_{\GModi\g{\cal B}}(M,N)$ decomposes as a sum of morphisms $(f^0,f^1)=(f^0,0)+(0,f^1)$ in $\Hom_{\GModi\g {\cal B}}(M,N)$. 

This notation as pairs of maps was relevant in the study of categories of modules over differential tensor algebras in \cite{BSZ}. We exploit it here in the context of graded $S$-bocses. 
\end{remark}

\begin{definition}\label{D: triangular bocs}
A normal graded $S$-bocs ${\cal B}=(C,\mu,\epsilon,\delta)$ is called \emph{triangular} iff 
there is a sequence 
$$0=\overline{C}_0\subseteq \overline{C}_1\subseteq\cdots\subseteq
\overline{C}_i\subseteq\overline{C}_{i+1}\subseteq \cdots $$
of graded $S$-$S$-subbimodules of $\overline{C}$ such that $\overline{C}=\bigcup_{i\in \hueca{N}}\overline{C}_i$, 

$$ \overline{\mu}(\overline{C}_i)\subseteq \overline{C}_{i-1}\otimes_S\overline{C}_{i-1} \hbox{ \  and \  } 
\overline{\delta}(\overline{C}_i)\subseteq \overline{C}_i,
\hbox{ for all } i\in \hueca{N}.$$ 
\end{definition}

\begin{proposition}\label{P: morfismos localmente nilpotentes en bocses}
Let ${\cal B}=(C,\mu,\epsilon,\delta)$ be a triangular graded $S$-bocs. Consider any endomorphism in $\GModi^0\g {\cal B}$ of the form 
$f=(0,f^1):M\rightmap{}M$. Then, 
for each $x\in \overline{C}$, there is a natural number  $n_x$ such that $(f^{n})^1[m\otimes x]=0$, 
for all $n\geq n_x$ and $m\in M$. In this case, we can consider the morphism in $\GModi^0\g {\cal B}$ 
$$\sum_{n=1}^\infty f^n:M\rightmap{}M.$$
\end{proposition}

\begin{proof} Adopt the notations of (\ref{D: triangular bocs}). We proceed to show, by induction on $i\in \hueca{N}$,  
 that there is a natural number $n_i$ such that for any $x\in \overline{C}_i$ and $m\in M$, 
we have,  $(f^{n_i})^1[m\otimes  x]=0$. 

From the definition of the composition 
of morphisms is $\GModi\g {\cal B}$, we have that 
$(f^{n+1})^1(m\otimes x)=f^n(f^1\otimes id_{\overline{C}})(id_M\otimes \overline{\mu})[m\otimes x]=0$, for $x\in \overline{C}_1$ and $n\geq 1$, and we can take $n_1=2$. 

Assume that we have proved that $n_i$ exists for a fixed $i\in \hueca{N}$, such that 
for any $x\in \overline{C}_{i}$ and $m\in M$, we have 
$(f^{n_i})^1[m\otimes x]=0$. Then, if $x\in \overline{C}_{i+1}$, we have that 
$\overline{\mu}(x)=\sum_{j=1}^tx_j\otimes y_j$, with $x_j,y_j\in \overline{C}_{i}$. Thus, for $m\in M$, 
we have 
$$\begin{matrix}(f^{n_i+1})^1[m\otimes x]&=&
f^{n_i}(f^1\otimes id_{\overline{C}})(id_M\otimes \overline{\mu})[m\otimes x]\hfill\\
&=&
f^{n_i}(f^1\otimes id_{\overline{C}})(m\otimes \overline{\mu}(x))\hfill\\
&=&
\sum_jf^{n_i}(f^1\otimes id_{\overline{C}})(m\otimes x_j\otimes y_j)\hfill\\
&=&
\sum_jf^{n_i}(f^1(m\otimes x_j)\otimes y_j)=0.\hfill\\
\end{matrix}
$$
The statement of the Proposition follows from this. 
\end{proof}

\begin{corollary}\label{C: f iso sii f0 es iso}
Let ${\cal B}$ be a triangular graded $S$-bocs. Take any morphism $f=(f^0,f^1):M\rightmap{}N$ in $\GModi^0\g {\cal B}$. Then, 
$f:M\rightmap{}N$ is an isomorphism in $\GModi^0\g {\cal B}$ iff $f^0:M\rightmap{}N$ is 
an isomorphism of $\GM\g S$. 
\end{corollary}

\begin{proof} We first show that an endomorphism of the form 
$f=(id_M,f^1):M\rightmap{}M$ is an automorphism. 
 Indeed, we can write $f=\hueca{I}_M-h$ where 
$h=(0,-f^1)$. Then 
$$(\hueca{I}_M-h)*(\sum_{i=0}^\infty h^i)=
\sum_{i=0}^\infty h^i-h*(\sum_{i=0}^\infty h^i)=h^0=\hueca{I}_M.$$
Similarly, we have $(\sum_{i=0}^\infty h^i)*(\hueca{I}_M-h)=\hueca{I}_M$. Thus $f$ is an automorphism. 

Now, consider any morphism $f=(f^0,f^1):M\rightmap{}N$ with $f^0:M\rightmap{}N$ 
isomorphism and consider the inverse $g^0:N\rightmap{}M$ of $f^0$. 
Then, $f*(g^0,0)=(id_N,u)$ and $(g^0,0)*f=(id_M,v)$ are isomorphisms. So, 
$f$ is a section and a retraction in $\GModi\g {\cal B}$. It follows 
that $f$ is an isomorphism in $\GModi\g{\cal B}$. 
\end{proof}

\begin{corollary}\label{C: f iso sii f0 es iso para TwMod-B}
 Let ${\cal B}$ be a triangular graded $S$-bocs. Take any  morphism $f:(M,u)\rightmap{}(N,v)$ in $\TGM^0\g {\cal B}$. Then, 
$f:M\rightmap{}N$ is an isomorphism in $\GModi^0\g {\cal B}$ iff its first component $f^0:M\rightmap{}N$ is 
an isomorphism of $\GM\g S$. 
\end{corollary}

\begin{proof} 
From (\ref{C: f iso sii f0 es iso}),  it will be enough to show the following. 

\medskip
\noindent\emph{Claim:  Fix any homogeneous morphism $f:(M,u)\rightmap{}(N,v)$ in $\TGM^0\g {\cal B}$. Then, $f:(M,u)\rightmap{}(N,v)$ is an isomorphism in $\TGM^0\g{\cal B}$ iff $f:M\rightmap{}N$ is an isomorphism in $\GM\g{\cal B}$. }
\medskip

We now prove the claim. Since $f:(M,u)\rightmap{}(N,v)$ is a  morphism in $\TGM^0\g {\cal B}$, 
we have that $\hat{\delta}(f)+v*f-f*u=0$. If $f$ is an isomorphism in $\GM^0\g{\cal B}$, it has a homogeneous two-sided inverse $g:N\rightmap{}M$ with $\vert g\vert=0$. Since $g*f=\hueca{I}_M$, we have $0=\hat{\delta}(\hueca{I}_M)=\hat{\delta}(g*f)=\hat{\delta}(g)*f+g*\hat{\delta}(f)$. So,  
$$
\hat{\delta}(g)
=
-g*\hat{\delta}(f)*g
=
g*(v*f)*g-g*(f*u)*g
=
g*v-u*g.
$$
So, $g:(N,v)\rightmap{}(M,u)$ is a morphism in $\TGM^0\g{\cal B}$, an  inverse for $f$. 
\end{proof}

 \begin{proposition}\label{P: conflations in GMod-B} Suppose that ${\cal B}$  is a triangular differential graded $S$-bocs. 
Consider a composable pair of morphisms  in $\GModi^0\g{\cal B}$ 
 $$M\rightmap{f} E\rightmap{g}N$$
such that 
 $g*f=0$ and the sequence 
 $0\rightmap{}M\rightmap{f^0} E\rightmap{g^0}N\rightmap{}0$ 
 is an exact sequence in $\GM\g S$. Then, there is 
 a commutative diagram in $\GModi^0\g {\cal B}$
 $$\begin{matrix}
    M&\rightmap{f} &E&\rightmap{g}&N\\
    \parallel&&\lmapdown{h}&&\parallel\\
M&\rightmap{ \ (r^0,0) \ } &E&\rightmap{ \ (s^0,0) \ }&N,\\
   \end{matrix}$$
where $h$ is an automorphism.
 \end{proposition}
 
 \begin{proof} \emph{Step 1:} There is an automorphism $h$ of $E$ such that $(h*f)^1=0$.
 \medskip
 
 We have $f=\underline{f}^0+\underline{f}^1$, where  $\underline{f}^0=(f^0,0)$ and 
 $\underline{f}^1=(0,f^1)$ are morphisms in $\GModi^0\g{\cal B}$. Since $S$ is semisimple, the exact 
 sequence in the statement of our proposition splits and we have a morphism $p:E\rightmap{}M$ in $\GM\g S$ 
 such that $pf^0=id_M$. Consider the morphism $\underline{p}=(p,0):E\rightmap{}M$ 
 in $\GModi^0\g {\cal B}$. Then, we have $\underline{p}*\underline{f}^0=\hueca{I}_M$ in $\GModi\g {\cal B}$. 
 Make $u:=\underline{f}^1*\underline{p}$ and notice that $u^0=0$. 
 Then, we have the automorphism $h$ of $E$  defined by 
 $$h=\hueca{I}_E+v:E\rightmap{}E,\hbox{ \ where \ } v=\sum_{n=1}^{\infty}(-1)^nu^n.$$
 We have that $h*f=\underline{f}^0+\underline{f}^1+v*\underline{f}^0+
 v*\underline{f}^1$, where 
 $$\begin{matrix}
  \underline{f}^1+v*\underline{f}^0+v*\underline{f}^1
  &=&
  \underline{f}^1-\underline{f}^1*\underline{p}*\underline{f}^0+
  \sum_{n=1}^{\infty}(-1)^{n+1}(\underline{f}^1*\underline{p})^n*(\underline{f}^1*\underline{p})*\underline{f}^0\hfill\\
  &&\hfill\\
  &&+\, \sum_{n=1}^{\infty}(-1)^n(\underline{f}^1*\underline{p})^n*\underline{f}^1\hfill\\
  &&\hfill\\
  &=&
  \underline{f}^1-\underline{f}^1+
  \sum_{n=1}^{\infty}(-1)^{n+1}(\underline{f}^1*\underline{p})^n*\underline{f}^1\hfill\\
  &&\hfill\\
  &&+\,
  \sum_{n=1}^{\infty}(-1)^n(\underline{f}^1*\underline{p})^n*\underline{f}^1
  
  =
  0.\hfill\\
   \end{matrix}$$
 Therefore, we obtain $(h*f)^1=0$.   
 
 \medskip
 \emph{Step 2:} Now, we may assume that $f=(f^0,0)$ and finish the proof.
 \medskip
 
 As before, we have  $g=\underline{g}^0+\underline{g}^1$, where  $\underline{g}^0=(g^0,0)$ and 
 $\underline{g}^1=(0,g^1)$ are morphisms in $\GModi^0\g{\cal B}$. Since $S$ is semisimple, 
 there is a morphism $s:N\rightmap{}E$ in $\GM\g S$ such that $g^0s=id_N$. If we take 
 $\underline{s}=(s,0)$ we have $\underline{g}^0*\underline{s}=\hueca{I}_N$ in $\GModi\g {\cal B}$. 
 Now, make $u=\underline{s}*\underline{g}^1$ and notice that $u^0=0$. Now, 
 consider the automorphism $h$ of $E$ defined by 
 $$h=\hueca{I}_E+w:E\rightmap{}E, \hbox{ \ where \ } w=\sum_{i=1}^\infty(-1)^nu^n.$$
We have 
  that $g*h=\underline{g}^0+\underline{g}^1+\underline{g}^0*w
  +\underline{g}^1*w$, where 
 $$\begin{matrix}
  \underline{g}^1+\underline{g}^0*w+\underline{g}^1*w
  &=&
  \underline{g}^1-\underline{g}^0*\underline{s}*\underline{g}^1
  +
  \sum_{n=1}^{\infty}(-1)^{n+1}\underline{g}^0*\underline{s}*\underline{g}^1*(\underline{s}*
  \underline{g}^1)^n\hfill\\
  &&\hfill\\
  &&+\, 
  \sum_{n=1}^{\infty}(-1)^n\underline{g}^1*(\underline{s}*\underline{g}^1)^n\hfill\\
  &&\\
  &=&
    \underline{g}^1-\underline{g}^1 +
  \sum_{n=1}^{\infty}(-1)^{n+1}\underline{g}^1*(\underline{s}*\underline{g}^1)^n\hfill\\
  &&\hfill\\
  &&+\, \sum_{n=1}^{\infty}(-1)^n\underline{g}^1*(\underline{s}*\underline{g}^1)^n
  =  0.\hfill\\
   \end{matrix}$$
 Therefore, we obtain $(g*h)^1=0$.   Notice that 
 $$u+w+w*u = u+\sum_{n=1}^{\infty}(-1)^nu^n+\sum_{n=1}^{\infty}(-1)^nu^{n+1}=0.$$
Hence, $h^{-1}=\hueca{I}_E+u$. Since $g*f=0$, we have that $\underline{g}^1*f=0$ and, therefore, 
$h^{-1}*f=f+u*f=f$. So, $(h^{-1}*f)^1=0$ and we are done. 
 \end{proof}

 \begin{lemma}\label{L: iso en GMod-B de M en N copia la u de M en una de N}
 Let ${\cal B}$ be a differential graded $S$-bocs, assume that $(M,u)\in \TGM\g {\cal B}$,
 and $N\in \GM\g{\cal B}$. Then, for any homogeneous isomorphism $h:M\rightmap{}N$ in $\GM\g {\cal B}$, 
 there is a unique morphism $v\in \Hom_{\GM\g {\cal B}}^1(N,N)$ such that $h:(M,u)\rightmap{}(N,v)$ is an isomorphism in $\TGM\g{\cal B}$.
 \end{lemma}

 \begin{proof} By assumption, $\hat{\delta}(u)+u*u=0$ and we are looking for a morphism $v\in \Hom_{\GM\g{\cal B}}^1(N,N)$ such that 
 $\hat{\delta}(h)+v*h-(-1)^{\vert h\vert}h*u=0$. Then, the only possible choice is 
 $v=(-1)^{\vert h\vert}h*u*h^{-1}-\hat{\delta}(h)*h^{-1}$. It remains to show that $\hat{\delta}(v)+v*v=0$. We shall compute each term separately.

 From Leibniz rule, we have that $0=\hat{\delta}(\hueca{I}_N)=\hat{\delta}(h*h^{-1})=\hat{\delta}(h)*h^{-1}+(-1)^{\vert h\vert}h*\hat{\delta}(h^{-1})$.
 Then, we have $\hat{\delta}(h^{-1})=-(-1)^{\vert h\vert}h^{-1}*\hat{\delta}(h)*h^{-1}$, and 
 $$\begin{matrix}
   \hat{\delta}(v)
   &=& 
   (-1)^{\vert h\vert}\hat{\delta}(h*u*h^{-1})
   -\hat{\delta}(\hat{\delta}(h)*h^{-1})\hfill\\
   &=&
   (-1)^{\vert h\vert}\hat{\delta}(h)*u*h^{-1}+
   h*\hat{\delta}(u)*h^{-1}\hfill\\
   &&-\,h*u*\hat{\delta}(h^{-1})
+  (-1)^{\vert h\vert}\hat{\delta}(h)*\hat{\delta}(h^{-1})\hfill\\
   &=&
   (-1)^{\vert h\vert}\hat{\delta}(h)*u*h^{-1}+
   h*\hat{\delta}(u)*h^{-1}\hfill\\
   &&+\,
   (-1)^{\vert h\vert} h*u*h^{-1}*\hat{\delta}(h)*h^{-1}\hfill\\
   &&-\, 
  \hat{\delta}(h)*h^{-1}*\hat{\delta}(h)*h^{-1}.\hfill\\
   \end{matrix}$$
Moreover,  we have 
$$\begin{matrix}
   v*v
   &=&
   [(-1)^{\vert h\vert}h*u*h^{-1}-\hat{\delta}(h)*h^{-1}]*
   [(-1)^{\vert h\vert}h*u*h^{-1}-\hat{\delta}(h)*h^{-1}]\hfill\\
   &=&
   h*u*u*h^{-1}
   -(-1)^{\vert h\vert}h*u*h^{-1}*\hat{\delta}(h)*h^{-1}\hfill\\
   &&-\,
   (-1)^{\vert h\vert}\hat{\delta}(h)*u*h^{-1}
   +\hat{\delta}(h)*h^{-1}*\hat{\delta}(h)*h^{-1}.\hfill\\
  \end{matrix}$$
Since $\hat{\delta}(u)+u*u=0$, we obtain $\hat{\delta}(v)=-v*v$, as we wanted.  
\end{proof}

 \begin{proposition}\label{P: conflations homog en TGMod-B}
 Suppose that ${\cal B}$  is a triangular graded $S$-bocs. 
 Consider a composable pair of homogeneous morphisms in $\TGM^0\g{\cal B}$ 
 $$(M,u_M)\rightmap{f} (E,u_E)\rightmap{g}(N,u_N)$$
 such that 
 $g*f=0$ and the sequence of first components 
 $$0\rightmap{}M\rightmap{f^0} E\rightmap{g^0}N\rightmap{}0$$ 
 is exact in $\GM\g S$. Then, there is 
 a commutative diagram in $\TGM^0\g {\cal B}$
 $$\begin{matrix}
    (M,u_M)&\rightmap{f} &(E,u_E)&\rightmap{g}&(N,u_N)\\
    \parallel&&\lmapdown{h}&&\parallel\\
(M,u_M)&\rightmap{ \ r \ } &(E,u'_E)&\rightmap{ \ s \ }&(N,u_N),\\
   \end{matrix}$$
where $h$ is an isomorphism and the second components of $r$ and $s$ are zero. 
 \end{proposition}
 
 \begin{proof}  Consider the underlying diagrams in $\GM^0\g{\cal B}$ of our hypothesis. Then,  
 there is an isomorphism $h:E\rightmap{}E$ in $\GM^0\g{\cal B}$ as in (\ref{P: conflations in GMod-B}). 
 We have  $(E,u_E)\in \TGM\g{\cal B}$ and, so, we can apply 
 (\ref{L: iso en GMod-B de M en N copia la u de M en una de N}) to $h$ and obtain a morphism $u'_E\in \Hom_{\GM\g{\cal B}}^1(E,E)$ such that $h:(E,u_E)\rightmap{}(E,u'_E)$ is an isomorphism in $\TGM^0\g{\cal B}$.
 \end{proof}

 \section{Scalar restriction functors and homotopy}
 
 Let us recall the classical definition of homotopy of differential graded $S$-coalge\-bras (or $S$-bocses).
 
 \begin{definition}\label{D: para morfismos de bocses homotopia}
 Let $\overline{\cal B}_A=(\overline{C}_A,\overline{\mu}_A,\overline{\delta}_A)$ and  
 $\overline{\cal B}_B=(\overline{C}_B,\overline{\mu}_B,\overline{\delta}_B)$ be differential 
 graded $S$-bocses (without counit). 
 Let $\overline{\phi},\overline{\psi}:\overline{C}_A\rightmap{}\overline{C}_B$ be morphisms of 
 differential graded  $S$-bocses. A \emph{homotopy $\overline{h}$
 from $\overline{\phi}$ to $\overline{\psi}$}
 is a morphism $\overline{h}\in \Hom^{-1}_{\GM\g S\g S}(\overline{C}_A,\overline{C}_B)$ such that 
 $$\overline{\mu}_B\overline{h}=(\overline{\phi}\otimes \overline{h}+
 \overline{h}\otimes \overline{\psi})\overline{\mu}_A \hbox{ \   \ and \  \ }\overline{\phi}-\overline{\psi}
 =\overline{\delta}_B\overline{h}+\overline{h}\,\overline{\delta}_A.$$
 The morphisms $\overline{\phi}$ and $\overline{\psi}$ are called \emph{homotopic} iff 
 there is such a homotopy $\overline{h}$. A morphism $\overline{h}\in \Hom^{-1}_{\GM\g S\g S}(\overline{C}_A,\overline{C}_B)$ is called a \emph{$\overline{\phi}\g\overline{\psi}$-coderivation} when the preceding left equality is satisfied. We denote by $\Coder^{-1}_{\overline{\phi}\g \overline{\psi}}(\overline{C}_A,\overline{C}_B)$ the space of  $\overline{\phi}\g\overline{\psi}$-coderivations.
 
 The definition of \emph{homotopy of morphisms of differential
 counitary graded $S$-bocses} is similar, we just remove the overlines 
 and add $\epsilon_A$ and $\epsilon_B$ to the given triples, respectively. 
 \end{definition}

 The proof of the following statement is straightforward. 
 
 \begin{lemma}\label{L: homotopias de morfismos de overline C a C}
 Let   ${\cal B}_A=({C}_A,{\mu}_A,\epsilon_A,{\delta}_A)$ and  
 ${\cal B}_B=({C}_B,{\mu}_B,\epsilon_B,{\delta}_B)$ be differential normal 
 graded $S$-bocses. 
 Let $\overline{\phi},\overline{\psi}:\overline{\cal B}_A\rightmap{}\overline{\cal B}_B$ be morphisms of 
 the corresponding reduced differential graded $S$-bocses and assume that $\overline{h}$ is a 
 homotopy from $\overline{\phi}$ to $\overline{\psi}$. Extend $\overline{\phi}$ and $\overline{\psi}$ 
 to morphisms of differential unitary graded $S$-bocses $\phi,\psi:{\cal B}_A\rightmap{}{\cal B}_B$ defining 
 $\phi(s+x)=s+\overline{\phi}(x)$ and 
 $\psi(s+x)=s+\overline{\psi}(x)$, for $s\in S$ and $x\in \overline{C}_A$; 
 extend $\overline{h}$ to $h:C_A\rightmap{}C_B$ defining $h(s+x)=\overline{h}(x)$. Then, 
 we have that $h$ is a homotopy from $\phi$ to $\psi$. That is 
 $h\in \Hom_{\GM\g S\g S}^{-1}(C_A,C_B)$ satisfies 
$$\mu_Bh=(\phi\otimes h+h\otimes \psi)\mu_A 
\hbox{ \   \ and \  \ }
\phi-\psi=\delta_Bh+h\delta_A.$$
 \end{lemma}

\begin{lemma}\label{L: restricción por una homotopia}
 Under the assumptions of the last lemma, for $d\in \hueca{Z}$, define 
 $$R_h:\Hom_{\GM\g{\cal B}_B}^d(M,N)\rightmap{}\Hom_{\GM\g{\cal B}_A}^{d-1}(M,N)$$
 such that $R_h(u)=(-1)^{\vert u\vert} u(id_M\otimes h)$, for $u\in \Hom_{\GM\g{\cal B}_B}^d(M,N)$.
 Then, the following holds. 
 \begin{enumerate}
  \item Whenever $u:M\rightmap{}N$ and $v:N\rightmap{}L$  are 
  homogeneous morphisms in $\GM\g {\cal B}_B$, we have 
  $$R_h(v*u)=R_h(v)*R_\phi(u)+(-1)^{\vert v\vert} R_\psi(v)*R_h(u).$$
  \item For any homogeneous morphism $u:M\rightmap{}N$ in $\GM\g{\cal B}_B$,  we have 
  $$R_\phi(u)-R_\psi(u)= \hat{\delta}_AR_h(u)+R_h\hat{\delta}_B(u).$$
 \end{enumerate}
\end{lemma}

\begin{proof} (1): Having in mind (\ref{P: restrictions on graded modules}), we get  
$$\begin{matrix}
  R_h(v *u)&=&
  (-1)^{\vert u\vert+\vert v\vert}v(u\otimes id_{C_B})(id_M\otimes \mu_B)(id_M\otimes h)\hfill\\
  &=&
   (-1)^{\vert u\vert+\vert v\vert}v(u\otimes id_{C_B})(id_M\otimes \mu_Bh)\hfill\\
  &=&
    (-1)^{\vert u\vert+\vert v\vert}v(u\otimes id_{C_B})(id_M\otimes (\phi\otimes h+h\otimes\psi)\mu_A)\hfill\\
  &=&
    (-1)^{\vert u\vert+\vert v\vert}v(u\otimes id_{C_B})(id_M\otimes (\phi\otimes h)\mu_A)\hfill\\
  &&\,
    (-1)^{\vert u\vert+\vert v\vert}v(u\otimes id_{C_B})(id_M\otimes (h\otimes\psi)\mu_A)\hfill\\
  &=&
    (-1)^{\vert u\vert +\vert v\vert}v(u(id_M\otimes \phi)\otimes h)(id_M\otimes \mu_A)\hfill\\
  &&\,
    (-1)^{\vert u\vert +\vert v\vert}v(u(id_M\otimes h)\otimes\psi)(id_M\otimes\mu_A)\hfill\\
  &=&
   (-1)^{\vert v\vert}v(id_N\otimes h)(u(id_M\otimes \phi)\otimes id_{C_A})(id_M\otimes \mu_A)\hfill\\
  &&\,
    (-1)^{\vert u\vert +\vert v\vert}v(id_N\otimes \psi)(u(id_M\otimes h)\otimes id_{C_A})(id_M\otimes\mu_A)\hfill\\
  &=&
  R_h(v)*R_\phi(u)+(-1)^{\vert v\vert}R_\psi(v)*R_h(u).\hfill\\
  \end{matrix}$$
(2): Since 
$\hat{\delta}_B(u)=(-1)^{\vert u\vert+1}u(id_M\otimes \delta_{B})$,  we have 
$$\begin{matrix}
  R_h\hat{\delta}_B(u)
  &=&
  (-1)^{\vert u\vert +1}(-1)^{\vert u\vert +1}u(id_M\otimes \delta_B)(id_M\otimes h)\hfill\\
  &=&
  u(id_M\otimes \delta_Bh)= u(id_M\otimes (\phi-\psi-h\delta_A))\hfill\\
  &=&
  u(id_M\otimes \phi)-u(id_M\otimes \psi)-u(id_M\otimes h)(id_M\otimes \delta_A)\hfill\\
  &=&
  R_\phi(u)-R_\psi(u)-\hat{\delta}_AR_h(u).\hfill\\
  \end{matrix}$$
\end{proof}

\begin{proposition}\label{P: equiv de funtores restriccion de morfismos homotopicos}
Adopt the assumptions of 
(\ref{L: homotopias de morfismos de overline C a C}) and assume, 
furthermore, that ${\cal B}_A$ is a triangular $S$-bocs. Then, 
there is an isomorphism of functors $$\underline{\eta}:\underline{R}_\phi\rightmap{}\underline{R}_\psi,$$
 where 
 $\underline{R}_\phi,\underline{R}_\psi:
 \underline{\TGM}^0\g{\cal B}_B\rightmap{}
 \underline{\TGM}^0\g {\cal B}_A$ are the functors induced on the homotopy categories by 
 $R_\phi,R_\psi:\TGM^0\g{\cal B}_B\rightmap{}
 \TGM^0\g {\cal B}_A$, respectively, see (\ref{L: restriction functors and homotopy}). For each $(M,u)\in \TGM\g {\cal B}_B$, the map 
 $\underline{\eta}_{(M,u)}:R_\phi(M,u)\rightmap{}R_\psi(M,u)$ is the class modulo homotopy of the morphism 
 $$\eta_{(M,u)}=\hueca{I}_M+R_h(u):R_\phi(M,u)\rightmap{}R_\psi(M,u).$$
\end{proposition}

\begin{proof} We first show that 
$\hueca{I}_M+R_h(u)\in \Hom^0_{\GM\g {\cal B}_A}(M,M)$
is an isomorphism in the category $\GM\g {\cal B}_A$. In order to see this,
consider  the equivalence 
$F:\GM\g {\cal B}_A\rightmap{}\GModi\g {\cal B}_A$ of (\ref{L: equiv directa entre GMod-B y G hueca(M)od-B}). 
Consider the morphism $F(R_h(u))=(h^0,h^1):M\rightmap{}M$ in $\GModi\g {\cal B}$ and notice that, for $m\in M$, we have   
$$h^0(m)=R_h(u)(m\otimes 1)=-u(id_M\otimes h)(m\otimes 1)=0,$$ because $h(1)=0$. 
Therefore, $h^0=0$ and, by (\ref{C: f iso sii f0 es iso}), we know that  $\hueca{I}_M+R_h(u)$ is an 
isomorphism in $\GM^0\g {\cal B}_A$. 

 Now we show that $\hueca{I}_M+R_h(u)\in \Hom^0_{\TGM\g {\cal B}_A}(R_\phi(M,u),R_\psi(M,u))$. By definition, 
 this is equivalent to show the identity   
 $$(*):\hbox{ \ }\hat{\delta}_A(\hueca{I}_M+R_h(u))+
 R_\psi(u)*(\hueca{I}_M+R_h(u))-(\hueca{I}_M+R_h(u))*R_\phi(u)=0.$$
 We know that $u:M\otimes_SC_B\rightmap{}M$ is a homogeneous morphism of degree $1$ such that 
 $\hat{\delta}_B(u)+u*u=0$. Then, applying $R_h$ and using 
 (\ref{L: restricción por una homotopia}), we get 
 $$\begin{matrix}0&=&R_h(\hat{\delta}_B(u))+R_h(u*u)\hfill\\
 &=&
 R_\phi(u)-R_\psi(u)-\hat{\delta}_AR_h(u)+R_h(u)*R_\phi(u)-R_\psi(u)*R_h(u).\hfill\\
 \end{matrix}$$
 That is
 $\hat{\delta}_AR_h(u)+R_\psi(u)*R_h(u)-R_h(u)*R_\phi(u)+R_\psi(u)-R_\phi(u)=0$, 
 which is equivalent to the identity $(*)$. 
 
 Finally, we show that the family of classes 
 $\underline{\eta}_{(M,u)}:R_\phi(M,u)\rightmap{}R_\psi(M,u)$ modulo homotopy of the morphisms 
 $\eta_{(M,u)}=\hueca{I}_M+R_h(u):R_\phi(M,u)\rightmap{}R_\psi(M,u)$ is a natural transformation. 
 Take any homogeneous morphism $f:(M,u)\rightmap{}(N,v)$ of degree $0$ in $\TGM\g {\cal B}_B$. Hence, we have 
 $\hat{\delta}_B(f)+v*f-f*u=0$. Applying $R_h$ to the preceding equality, we obtain 
 $$\begin{matrix}
   0&=& R_h\hat{\delta}_B(f)+R_h(v*f)-R_h(f*u)\hfill\\
   &=&
   R_\phi(f)-R_\psi(f)-\hat{\delta}_AR_h(f) +
   R_h(v)*R_\phi(f)-R_\psi(v)*R_h(f)\hfill\\
   &&-\,R_h(f)*R_\phi(u)-R_\psi(f)*R_h(u).\hfill\\   
   \end{matrix}$$
   Hence, 
$$\begin{matrix}
   \hat{\delta}_AR_h(f)&=& 
   R_\phi(f)-R_\psi(f)+
   R_h(v)*R_\phi(f)-R_\psi(v)*R_h(f)\hfill\\
   &&-\,R_h(f)*R_\phi(u)-R_\psi(f)*R_h(u).\hfill\\   
   \end{matrix}$$
 Let us compute the difference $D=\eta_{(N,v)}* R_\phi(f)-R_\psi(f)*\eta_{(M,u)}$. 
 
 $$\begin{matrix}
    D&=&(\hueca{I}_N+R_h(v))*R_\phi(f)-R_\psi(f)*(\hueca{I}_M+R_h(u))\hfill\\
  &=&R_\phi(f)-R_\psi(f) +
  R_h(v)*R_\phi(f)-R_\psi(f)*R_h(u)\hfill\\
  &=&
  \hat{\delta}_AR_h(f)+R_h(f)*R_\phi(u)+R_\psi(v)*R_h(f).\hfill\\
  \end{matrix}$$
Hence, $D$ is null-homotopic in $\TGM^0\g {\cal B}_A$. This finishes the proof.
\end{proof}

\begin{corollary}\label{C: equiv homotopica de bocses triangulares induce equiv de cats homotópicas} 
Let $\phi:{\cal B}_A\rightmap{}{\cal B}_B$ be a homotopy equivalence of triangular  $S$-bocses. 
Then, the corresponding restriction functor  determines an equivalence of categories 
$$\underline{R}_\phi:\underline{\TGM}^0\g{\cal B}_B\rightmap{}
 \underline{\TGM}^0\g {\cal B}_A.$$
\end{corollary}

\begin{proof} If $\phi$ is a homotopy equivalence, its class modulo homotopy $\underline{\phi}$ has an inverse $\underline{\psi}$, for some morphism of triangular $S$-bocses $\psi:{\cal B}_B\rightmap{}{\cal B}_A$. Then we have $\underline{\phi\psi}=\underline{\phi}\, \underline{\psi}\sim \underline{id}_{{\cal B}_B}$ and $\underline{\psi\phi}=\underline{\psi}\, \underline{\phi}\sim \underline{id}_{{\cal B}_A}$. Therefore, we get isomorphisms of functors  
$\underline{R}_\phi \underline{R}_\psi= \underline{R}_{\psi\phi}\cong  id_{\underline{\TGM}^0\g{\cal B}_A}$
 and 
$\underline{R}_\psi \underline{R}_\phi= \underline{R}_{\phi\psi}\cong  id_{\underline{\TGM}^0\g{\cal B}_B}$. 
\end{proof}

\section{Splitting idempotents for twisted modules}

In this section we show that idempotents split in the category $\TGM^0\g{\cal B}$, where ${\cal B}$ is any triangular differential graded $S$-bocs. We start with some simple considerations on the additivity of the category $\TGM\g{\cal B}$. 

\begin{lemma}\label{L: el canonical embedding}
Given a graded $S$-bocs ${\cal B}=(C,\mu,\epsilon)$, there is a  functor of graded $k$-categories 
$$L=L_{\cal B}:\GM\g S\rightmap{}\GM\g{\cal B}$$
which acts as the identity on objects and for any homogeneous morphism $f:M\rightmap{}N$ of right graded 
$S$-modules, by definition
$$L(f)=[M\otimes_S C\rightmap{ \ id_M\otimes \epsilon \  }M\otimes_S S\cong M\rightmap{f}N].$$
As a consequence, $\GM\g {\cal B}$ is an additive category. When $\epsilon$ is surjective, the functor $L$ is faithful and we call it \emph{the canonical embedding}. 
\end{lemma}

\begin{remark}\label{R: matrices of morphisms in GM-B}
If $\sigma_j:M_j\rightmap{}\bigoplus_{i=1}^nM_i$ and $\pi_j:\bigoplus_{i=1}^nM_i\rightmap{}M_j$ denote, respectively, the injection and projection maps corresponding to  the direct sum $M=\bigoplus_{i=1}^nM_i$ in $\GM\g S$, we have that  $\underline{\sigma}_j:=L(\sigma_j):M_j\rightmap{}\bigoplus_{i=1}^nM_i$ and $\underline{\pi}_j:=L(\pi_j):\bigoplus_{i=1}^nM_i\rightmap{}M_j$ are, respectively, 
the injection and projection morphisms corresponding to the direct sum $\bigoplus_{i=1}^nM_i$ in $\GM\g{\cal B}$. We will work with the usual matrix notation $(g_{j,i})$ for a morphism $g:\bigoplus_{i=1}^mM_i\rightmap{}\bigoplus_{j=1}^nN_j$ in an additive category, the context will always permit the reader to avoid confusion with this notation: in $\GM\g{\cal B}$, this means that $g_{j,i}=\underline{\pi}_j*g*\underline{\sigma_i}$ and, therefore,  $g=\sum_{i,j}\underline{\sigma}_j*g_{j,i}*\underline{\pi}_i$.

Notice also that, for a differential graded $S$-bocs ${\cal B}$ with differential $\delta$, from the definition of the differential $\widehat{\delta}$  on the category $\GM\g{\cal B}$ and (\ref{L: epsilon d = 0}), we immediately obtain that $\widehat{\delta}(L(h))=0$ for any morphism $h:M\rightmap{}N$ in $\GM\g S$. 

Moreover, if ${\cal B}$ is a differential graded $S$-bocs, the category $\TGM\g{\cal B}$ is additive too: given 
$(M_1,u_1), (M_2,u_2)\in \TGM\g{\cal B}$, their direct sum is 
$$(M_1\oplus M_2,u), \hbox{ where } u=\begin{pmatrix}
                                                                                            u_1&0\\ 0&u_2
                                                                                           \end{pmatrix}\in \Hom^1_{\GM\g{\cal B}}(M_1\oplus M_2,M_1\oplus M_2).$$
The injections are the morphisms $\underline{\sigma}_i:(M_i,u_i)\rightmap{}  (M_1\oplus M_2,u)$,                                                                                          
 while the projections are the morphisms $\underline{\pi}_i:(M_1\oplus M_2,u)\rightmap{}(M_i,u_i)$. 
 Indeed, since $u=\underline{\sigma}_1*u_1*\underline{\pi}_1+\underline{\sigma}_2*u_2*\underline{\pi}_2$, we obtain 
 $u*\underline{\sigma}_i=\underline{\sigma}_i*u_i*\underline{\pi}_i*\underline{\sigma}_i=\underline{\sigma}_i*u_i$. 
 Therefore, 
 $\widehat{\delta}(\underline{\sigma}_i)+u*\underline{\sigma}_i-\underline{\sigma}_i*u_i=0$. Similarly, we have 
 $\widehat{\delta}(\underline{\pi}_i)+u_i*\underline{\pi}_i-\underline{\pi}_i*u=0$. 
 
 Given a morphism $g:\bigoplus_{i=1}^mM_i\rightmap{}\bigoplus_{j=1}^nN_j$ in $\GM\g {\cal B}$, with associated matrix $(g_{j,i})$, from Leibniz rule, we have 
 $\widehat{\delta}(g)= \sum_{i,j}\underline{\sigma}_j*\widehat{\delta}(g_{j,i})*\underline{\pi}_i$. 
 Then, we have  
 $\widehat{\delta}(g)_{t,s}=\underline{\pi}_t*\widehat{\delta}(g)*\underline{\sigma}_s=
 \underline{\pi}_t*\underline{\sigma}_t*\widehat{\delta}(g_{t,s})*\underline{\pi}_s*\underline{\sigma}_s=\widehat{\delta}(g_{t,s})$, which means that the image under $\widehat{\delta}$ of the matrix $(g_{j,i})$ is the matrix of images 
 $(\widehat{\delta}(g_{j,i}))$. 
\end{remark}

\begin{lemma}\label{L: idemps split in GMod-B => idemps split in TGMod-B} 
 Let ${\cal B}=(C,\mu,\epsilon,\delta)$ be a differential graded $S$-bocs. Assume that $e:(M,u)\rightmap{}(M,u)$ is an idempotent morphism in the category $\TGM^0\g{\cal B}$ such that the idempotent $e:M\rightmap{}M$ splits in $\GM^0\g{\cal B}$, 
 then $e$ splits in $\TGM^0\g {\cal B}$. 
\end{lemma}

\begin{proof} By assumption, there is an isomorphism $h:M\rightmap{}M_1\oplus M_2$ in $\GM^0\g{\cal B}$ such that the following diagram commutes in $\GM^0\g {\cal B}$
$$\begin{matrix}
   M&\rightmap{h}&M_1\oplus M_2\\
   \lmapdown{e}&&\rmapdown{\begin{pmatrix}
                            \hueca{I}_{M_1}&0\\ 0&0\\
                           \end{pmatrix}}\\
  M&\rightmap{h}&M_1\oplus M_2\\
  \end{matrix}$$
where $M_1\oplus M_2$ is the direct sum of graded right ${\cal B}$-modules. 

According to (\ref{L: iso en GMod-B de M en N copia la u de M en una de N}), 
there is some   $v\in \Hom^1_{\GM\g{\cal B}}(M_1\oplus M_2,M_1\oplus M_2)$ such that 
 $h:(M,u)\rightmap{}(M_1\oplus M_2,v)$ is an isomorphism in $\TGM^0\g {\cal B}$. 
Then, $h*e*h^{-1}:(M_1\oplus M_2,v)\rightmap{}(M_1\oplus M_2,v)$ is a morphism in $\TGM^0\g {\cal B}$. 
Thus, $\widehat{\delta}(h*g*h^{-1})+v*(h*e*h^{-1})-(h*e*h^{-1})*v=0$. But  
$\widehat{\delta}(h*e*h^{-1})=\widehat{\delta}\begin{pmatrix}
                            \hueca{I}_{M_1}&0\\ 0&0\\
                           \end{pmatrix}=0$, because $\widehat{\delta}(\hueca{I}_{M_1})=0$. Hence, $(h*e*h^{-1})*v=v*(h*e*h^{-1})$. 
  It follows that 
                           $$v=\begin{pmatrix}v_1&0\\ 0&v_2\end{pmatrix},\hbox{ with } v_1\in \Hom^1_{\GM\g{\cal B}}(M_1,M_1) \hbox{ and } v_2\in \Hom^1_{\GM\g{\cal B}}(M_2,M_2).$$
 Then, from the equality         $\widehat{\delta}(v)+v*v=0$, we obtain the equalities                   
$$\widehat{\delta}(v_1)+v_1*v_1=0 \hbox{ \ and \ } \widehat{\delta}(v_2)+v_2*v_2=0.$$
Then, $(M_1,v_1)$ and $(M_2,v_2)$ belong to $\TGM\g{\cal B}$, and 
$$h:(M,u)\rightmap{}(M_1,v_1)\oplus(M_2,v_2)$$ is an isomorphism in $\TGM^0\g{\cal B}$ with 
$h*e*h^{-1}=\begin{pmatrix}
            \hueca{I}_{M_1}&0\\ 0&0
           \end{pmatrix}$, thus $e$ splits in $\TGM^0\g{\cal B}$, as claimed. 
\end{proof}

For the sake of computational simplicity, it is convenient to rewrite the morphisms of the category $\GModi\g {\cal B}$ in a slightly different way.  

\begin{definition}\label{R: canonical embedding GM-S --> GM-B}
Given a normal graded $S$-bocs ${\cal B}=(C,\mu,\epsilon)$, we can consider the following category   
$\GModit\g{\cal B}$. Its objects are the graded right $S$-modules and, for $d\in \hueca{Z}$, a \emph{morphism} $f:M\rightmap{}N$ of degree $d$ is a pair  $f=(f^0,f^1)$, where $f^0:M\rightmap{}N$ and $f^1:\overline{C}\rightmap{}\Hom_{\GM\g k}(M,N)$ are homogeneous morphisms, $f^0$ of right $S$-modules and $f^1$ of $S$-$S$-bimodules, of degree $d$. Thus, the hom spaces are given by 
$$\Hom^d_{\GModitsub\g {\cal B}}(M,N)=\Hom^d_{\GM\g S}(M,N)\times \Hom^d_{\GM\g S\g S}(\overline{C},\Hom_{\GM\g k}(M,N)).$$
The composition in this category is transfered from the composition in the category $\GModi\g {\cal B}$ 
with the help of the natural isomorphism 
$$\eta:\Hom_{\GM\g S}(M\otimes \overline{C},N)\rightmap{}\Hom_{\GM\g S\g S}(\overline{C},\Hom_{\GM\g k}(M,N)).$$
So, the composition of two morphisms  $f:M \rightmap{}N$ and $g:N\rightmap{}L$ in $\GModit \g {\cal B}$ is given  by $g*f=((g*f)^0,(g*f)^1)$, where $(g*f)^0=g^0f^0$ and the morphism of graded $S$-$S$-bimodules $(f*g)^1:\overline{C}\rightmap{}\Hom^*_{\GM\g k}(M,N)$ is   given  by 
$$(g*f)^1(c)= g^1(c)f^0 + g^0f^1(c) +\sum_i g^1 (c_i^2)f^1(c_i^1), 
\hbox{ where  \ } \overline{\mu}(c) =\sum_i  c^1_i \otimes c^2_i.$$
\end{definition}

It is easy to show that the bijection $(f^0,f^1)\mapsto (f^0,\eta(f^1))$ determines an isomorphism of categories  $\GModi\g {\cal B}\rightmap{}\GModit \g {\cal B}$. The formula for the composition of morphisms 
in $\GModit \g {\cal B}$ is the translation of the composition  in $\GModi\g {\cal B}$ using the precedent bijection.

 \begin{lemma}\label{L: idemptes se div en GM-B}
 For any  graded triangular $S$-bocs ${\cal B}$, 
idempotents split in $\GModi^0\g {\cal B}$. 
\end{lemma}

\begin{proof} In this proof, for the sake of notational simplicity, we write $gf$ instead of $g*f$ to indicate the composite morphism in $\GModit^0 \g{\cal B}$. 

It will be enough to show that any idempotent morphism 
 $e=(e^0,e^1):M\rightmap{}M$
splits in $\GModit^0 \g {\cal B}$.  Since idempotents clearly split in $\GM\g S $, it will be enough to show that there is an isomorphism $h:M\rightmap{}M$ such that $(heh^{-1})^1=0$. 
 Adopt the notation of (\ref{D: triangular bocs}), make  $\overline{C}_{-1}:=0$, and let us first show the following. 
 
 \medskip
 \emph{Claim 1: There is a sequence of isomorphisms $M\rightmap{h_1}M\rightmap{h_2}M\rightmap{h_3}\cdots$ in $\GModit^0\g{\cal B}$ such that, for each $i\geq 1$, we have  
 \begin{enumerate}
  \item Each isomorphism has the form $h_i=(id_M,h^1_i)$;
  \item $h^1_{i}(\overline{C}_{i-2})=0$; and 
  \item For $e_i:=h_i\cdots h_2h_1eh_1^{-1}h_2^{-1}\cdots h_i^{-1}$, we have $e_i^1(\overline{C}_{i-1})=0$.
 \end{enumerate}}

\emph{Proof of the Claim 1:} The inductive argument is essentially the same as the one given in \cite{BSZ}(5.12), but we recall it for the sake of the reader. At the base of the induction, we have the isomorphism  $h_1=(id_M,0):M\rightmap{}M$ such that $h_1^1(\overline{C}_{-1})=0$ and the idempotent $e_1:=e$ such that $e_1^1(\overline{C}_0)=0$. 

Assume that we have constructed the isomorphisms $h_1,\ldots,h_i$ satisfying conditions 1--3. 
Notice that $e_i$ is idempotent, and so is $e_i^0$. 
From $e_i^1(\overline{C}_{i-1})=0$  and the triangularity we obtain, for $c\in
\overline{C}_i$ that $e_i^1(c)=(e_i^2)^1(c)=e_i^0e_i^1(c)+e_i^1(c)e_i^0$. Hence,
$e_i^0e_i^1(c)e_i^0=0$. Now, for $c\in \overline{C}$ define
$h_{i+1}^1(c):=e_i^1(c)f_i^0$, where $f_i^0\in\Hom_{\GM\g S}(M,M)$ will be specified
in a moment. Then, clearly $h_{i+1}^1(\overline{C}_{i-1})=0$ and $h_{i+1}^1\in\Hom_{\GM\g S\g S}(\overline{C},\Hom_{\GM\g k}(M,M))$. Then, the pair $h_{i+1}=(1_M,h_{i+1}^1):M\rightmap{} M$ is a morphism in $\GModit^0 \g {\cal B}$. 
From (\ref{C: f iso sii f0 es iso}), we know that $h_{i+1}$ is an isomorphism in $\GModit^0 \g{\cal B}$. 

Let $g_{i+1}:=h_{i+1}^{-1}:M\rightmap{}M$. Then,
$g_{i+1}^0=1_M$ and, since $(g_{i+1}h_{i+1})^1=0$, we obtain for $c\in \overline{C}_i$
that $g_{i+1}^1(c)=-h_{i+1}^1(c)$. Then, by triangularity, we have for $c\in \overline{C}_i$,
$$\begin{matrix}
(h_{i+1}e_ig_{i+1})^1(c)
&=&
h_{i+1}^0(e_ig_{i+1})^1(c)+h_{i+1}^1(c)(e_ig_{i+1})^0\hfill\\ 
&=&
e_i^0g_{i+1}^1(c)+e_i^1(c)+h_{i+1}^1(c)e_i^0\hfill\\ 
&=&
-e_i^0e_i^1(c)f_i^0+e_i^1(c)+e_i^1(c)f_i^0e_i^0\hfill\\ 
&=&
-e_i^0e_i^1(c)f_i^0+e_i^0e_i^1(c)+e_i^1(c)e_i^0+e_i^1(c)f_i^0e_i^0.\hfill\\ \end{matrix}$$
Since $e_i^0e_i^1(c)e_i^0=0$, choosing $f_i^0:=1_M-2e_i^0$ we obtain the equality 
$e^1_{i+1}(c)=(h_{i+1}e_ig_{i+1})^1(c)=0$, as we wanted. \hfill$\square$
\medskip

From (1) and (2), we obtain, for $i\geq 1$ and $c\in \overline{C}_{i-1}$, the equality
$$(h_{i+1}\cdots h_2h_1)^1(c)=(h_i\cdots h_2h_1)^1(c).$$
Indeed, this is clear for $i=1$. For $i\geq 2$, we have $(h_i\cdots h_2h_1)^0=id_M$ and  $h^0_{i+1}=id_M$. By assumption, $\overline{\mu}(c)=\sum_tc^1_t\otimes c^2_t$, with $c^1_t,c^2_t\in \overline{C}_{i-2}$. Then, we obtain
$$\begin{matrix}
   (h_{i+1}\cdots h_2h_1)^1(c)
   &=&
   (h_{i+1}(h_i\cdots h_2h_1))^1(c)\hfill\\
   &=&
   h^1_{i+1}(c)+(h_i\cdots h_2h_1)^1(c)\hfill\\
   &&+\,\sum_th^1_{i+1}(c_t^2)(h_i\cdots h_2h_1)^1(c_t^1)\hfill\\
   &=&
   (h_i\cdots h_2h_1)^1(c).\hfill\\
  \end{matrix}$$
Then, we can consider the map $h^1:\overline{C}\rightmap{}\Hom_{\GM\g k}(M,M)$ defined by 
$$h^1(c)=(h_i\cdots h_2h_1)^1(c), \hbox{ for } c\in \overline{C}_{i-1}.$$
Since $h^1$ is a morphism of graded $S$-$S$-bimodules, we can consider the isomorphism $h=(id_M,h^1):M\rightmap{}M$ in $\GModit^0 \g{\cal B}$.

The statement (\ref{P: morfismos localmente nilpotentes en bocses})  gives us in this context that any morphism of the form $g=(0,g^1):M\rightmap{}M$ satisfies that for any $i\geq 1$ there is some $n_i\in \hueca{N}$ such that for $c\in \overline{C}_i$ and $n\geq n_i$ we have $(g^n)^1(c)=0$. So, we have a well defined morphism 
$\sum_{n=0}^\infty g^n:M\rightmap{}M$. So, consider $g:=(0,-h^1)$ and notice that $h=\hueca{I}_M-g$, so  
$h^{-1}=\sum_{n=0}^\infty g^n=\hueca{I}_M+ \sum_{n=1}^\infty g^n$. 
In particular, $(h^{-1})^1=\sum_{n=1}^\infty [(0,-h^1)^n]^1$.

Consider the iterated comultiplication 
$\overline{\mu}^n:\overline{C}\rightmap{}\overline{C}^{\otimes(n+1)}$, defined recursively, for $n\geq 1$, by $\overline{\mu}^0=id_{\overline{C}}$ and $\overline{\mu}^n=(id_{\overline{C}}\otimes \overline{\mu}^{(n-1)})\overline{\mu}$. 

\medskip
\emph{Claim 2: Given any morphism of the form $g=(0,g^1):M\rightmap{}M$ in $\GModit^0 \g{\cal B}$, the following holds 
\begin{enumerate}
 \item For $n\geq 1$ and $c\in \overline{C}$, we have 
$$(g^n)^1(c)=\sum_tg^1(c^n_t)\cdots g^1(c^1_t), \hbox{ where }
\overline{\mu}^{(n-1)}(c)=\sum_tc_t^1\otimes \cdots \otimes c_t^n.$$ 
\item Moreover, for $n\geq 2$, $i\geq 1$, and $c_i\in \overline{C}_i$, we have $c_t^1,\ldots,c_t^n\in \overline{C}_{i-1}$.
\end{enumerate}}
\medskip

\emph{Proof of the Claim 2:} We proceed by induction. Item (1) holds trivially for $n=1$. 
So assume item (1) holds for $n$. Then, for $c\in \overline{C}$, we have 
$$\begin{matrix}\overline{\mu}^n(c)
&=&
(id_{\overline{C}}\otimes \overline{\mu}^{(n-1)})\overline{\mu}(c)\hfill\\
&=&
(id_{\overline{C}}\otimes \overline{\mu}^{(n-1)})(\sum_sc_s^1\otimes c_s^2)\hfill\\
&=&
\sum_sc_s^1\otimes \overline{\mu}^{(n-1)}(c_s^2)\hfill\\
&=&
\sum_s\sum_{t_s}c_s^1\otimes c_{t_s}^2\otimes\cdots\otimes c_{t_s}^{n+1}\hfill\\
\end{matrix}$$
where $\overline{\mu}(c)=\sum_sc_s^1\otimes c_s^2$ and, for each $s$, 
$\overline{\mu}^{(n-1)}(c_s^2)=\sum_{t_s} c_{t_s}^2\otimes\cdots\otimes c_{t_s}^{n+1}$. 
Thus, 
$$ (g^{n+1})^1(c)=\sum_s(g^n)^1(c_s^2)g^1(c_s^1)=
    \sum_s\sum_{t_s}g^1(c_{t_s}^{n+1})\cdots g^1(c^2_{t_s})g^1(c^1_s).$$
Now, if $n\geq 2$, $i\geq 1$, and $c\in \overline{C}_i$, we have from the triangularity condition that $c_s^1\in \overline{C}_{i-1}$. 
We also have that $c_{t_s}^2,\ldots,c_{t_s}^{n+1}\in \overline{C}_{i-1}$ by the induction hypothesis for item (2). The Claim 2 is proved.\hfill$\square$ 
\medskip

Assume that $n\geq 1$, $c\in \overline{C}_{i-1}$, and $\overline{\mu}^{(n-1)}(c)=\sum_{t} c_t^1\otimes\cdots\otimes c_t^n$, then 
$$\begin{matrix}[(0,-h^1)^n]^1(c)
&=&
(-1)^n\sum_th^1(c_t^n)\cdots h^1(c^1_t)\hfill\\
&=&
(-1)^n\sum_t(h_i\cdots h_1)^1(c_t^n)\cdots (h_i\cdots h_1)^1(c^1_t)\hfill\\
&=&
[(0,-(h_i\cdots h_1)^1)^n]^1(c),\hfill\\
  \end{matrix}$$
and 
$$
  (h^{-1})^1(c)
  =
  \sum_{n=1}^\infty [(0,-h^1)^n]^1(c)
  =
    \sum_{n=1}^\infty [(0,-(h_i\cdots h_1)^1)^n]^1(c)
  =
  [(h_i\cdots h_1)^{-1}]^1(c).$$

  \medskip
\emph{Claim 3: Let $p,q:M\rightmap{}M$ be any morphisms in $\GModit^0 \g{\cal B}$. Then, 
\begin{enumerate}
 \item Given a family of morphisms $\{p_i:M\rightmap{}M\}_{i\geq 1}$ in $\GModit^0 \g{\cal B}$ such that $p_i^0=p^0$ and $p^1_i(c)=p^1(c)$, for $i\geq 1$ and $c\in \overline{C}_{i-1}$, we have 
 $$(qp_i)^1(c)=(qp)^1(c),\hbox{  for } c\in \overline{C}_{i-1}.$$
\item Given a family of morphisms $\{q_i:M\rightmap{}M\}_{i\geq 1}$ in $\GModit^0 \g{\cal B}$ such that $q_i^0=q^0$ and $q^1_i(c)=q^1(c)$, for $i\geq 1$ and $c\in \overline{C}_{i-1}$, we have  
$$(q_ip)^1(c)=(qp)^1(c),\hbox{ for } c\in \overline{C}_{i-1}.$$ 
\end{enumerate}}
\medskip

\emph{Proof of Claim 3:} (1): For $c\in \overline{C}_{i-1}$, we have $\overline{\mu}(c)=\sum_tc_t^1\otimes c_t^2$, with $c_t^1,c_t^2\in \overline{C}_{i-1}$, so
$$\begin{matrix}(qp)^1(c)
&=&
q^0p^1(c)+q^1(c)p^0+\sum_tq^1(c_t^2)p^1(c_t^1)\hfill\\
  &=&
q^0p_i^1(c)+q^1(c)p_i^0+\sum_tq^1(c_t^2)p_i^1(c_t^1)\hfill\\
&=&
(qp_i)^1(c).\hfill\\
  \end{matrix}$$
  The proof of (2) is similar. \hfill$\square$
  \medskip
  
Then, for $c\in \overline{C}$, say with $c\in \overline{C}_{i-1}$, we have 
$$(heh^{-1})^1(c)=(h_i\cdots h_1eh^{-1})^1(c)=(h_i\cdots h_1e(h_i\cdots h_1)^{-1})^1(c)=e^1_i(c)=0.$$
Therefore, $(heh^{-1})^1=0$, as we wanted to show. 
\end{proof}

\begin{definition}\label{D: cats de inducidos}
 Let ${\cal B}=(C,\mu,\epsilon)$ be a graded $S$-bocs. A  graded right ${\cal B}$-comodule $(M, \mu_M)$ is called \emph{induced} iff it is isomorphic to a right ${\cal B}$-comodule of the form $\Ind_{\cal B}(N)=(N\otimes C,id_N\otimes \mu)$ for some $N\in \GM\g S$. We denote by $\GInd\g{\cal B}$ the full subcategory of $\GC\g{\cal B}$ formed by the induced ${\cal B}$-comodules. So $\GInd\g{\cal B}$ is an additive subcategory of $\GC\g{\cal B}$.
 
 If ${\cal B}=(C,\mu,\epsilon,\delta)$ is a  differential graded $S$-bocs, a differential graded  right ${\cal B}$-comodule $(M,\mu_M,\delta_M)$ is called \emph{induced} iff 
 $(M,\mu_M)\in\GInd\g{\cal B}$.
The full subcategory of $\DGC\g{\cal B}$ formed by the differential graded induced ${\cal B}$-comodules 
will be denoted by $\DGInd\g{\cal B}$. So, $\DGInd\g{\cal B}$ is an additive subcategory of $\DGC\g{\cal B}$. 

Here again the superindex $0$ in $\GInd^0\g{\cal B}$ and  $\DGInd^0\g{\cal B}$ indicates the subcategory with all objects and only degree zero morphisms. 
\end{definition}

\begin{proposition}\label{P: split idempotents}
Let ${\cal B}=(C,\mu,\epsilon,\delta)$ be a triangular differential graded  $S$-bocs.  Then, 
in the categories $\GM^0\g{\cal B}$, $\GInd^0\g{\cal B}$, $\TGM^0\g{\cal B}$, and $\DGInd^0\g{\cal B}$  idempotents split.
\end{proposition}

\begin{proof} From (\ref{L: idemptes se div en GM-B}), 
idempotents split in $\GModi^0\g {\cal B}$. But $\GM^0\g{\cal B}$ is equivalent to this category, so idempotents split in $\GM^0\g{\cal B}$. By  (\ref{L: idemps split in GMod-B => idemps split in TGMod-B}), idempotents split in 
$\TGM^0\g{\cal B}$, 

By (\ref{P: comods coalgebras <--F--> mods bocses}), we know that the categories
$\GM^0\g{\cal B}$ and $\GInd^0\g {\cal B}$ are equivalent categories, and that 
$\TGM^0\g{\cal B}$ and $\DGInd^0\g {\cal B}$ are equivalent categories, 
so idempotents split in  $\GInd^0\g {\cal B}$ and in $\DGInd^0\g {\cal B}$. 
\end{proof}

\begin{remark}\label{R: varios canonical embeddings}
Given a triangular graded $S$-bocs ${\cal B}$, we have the equivalent 
categories $\GM\g{\cal B}\simeq\GModi\g{\cal B}\simeq \GModit \g{\cal B}$. Then, we have canonical embeddings $\GM\g S\rightmap{}\GModi\g{\cal B}$ and $\GM\g S\rightmap{}\GModit \g{\cal B}$ which we denote with the same symbol $L$, as in (\ref{L: el canonical embedding}). 

Notice that given $h:M\rightmap{}N$ in $\GM\g S$, we have $L(h)=(h,0)$ in $\GModi\g{\cal B}$ and in $\GModit \g{\cal B}$.
\end{remark}

The following statement can be proved as in \cite{BSZ}(6.2). 

\begin{lemma}\label{L: embedding preserves exactness}
Let ${\cal B}$ be a  triangular graded $S$-bocs. Then, any exact sequence 
$$0\rightmap{}M\rightmap{f^0}E\rightmap{g^0}N\rightmap{}0$$
in $\GM\g S$ determines an exact pair 
$M\rightmap{ \ (f^0,0) \ }E\rightmap{ \ (g^0,0) \ }N$
in $\GModit \g {\cal B}$. 
\end{lemma}

\begin{lemma}\label{L: B es acceptable}
Let ${\cal B}$ be a triangular graded  $S$-bocs. Then, for any morphism $f=(f^0,f^1):M\rightmap{}N$ in $\GModi^0\g{\cal B}$ the following holds.
\begin{enumerate}
 \item If $f^0$ is surjective, there is an automorphism  $h$ of $M$  such that $(f*h)^1=0$.
 \item If $f^0$ is injective, there is an automorphism  $h$ of $N$ such that $(h*f)^1=0$.
\end{enumerate}
\end{lemma}

\begin{proof} We will prove only (1). Assume that $f^0:M\rightmap{}N$ is surjective in $\GM\g S$. Since $S$ is semisimple, $f^0$ is a retraction. Choose a  right inverse $t^0:N\rightmap{}M$ for $f^0$. Adopt the notation of (\ref{D: triangular bocs}), skip the star in the notation for the composition in $\GModit^0 \g{\cal B}$, and make $\overline{C}_{-1}=0$. We first show the following. 
 
 \medskip
 \emph{Claim: There is a sequence of isomorphisms $\cdots \rightmap{h_3}M\rightmap{h_2}M\rightmap{h_1}M$ such that, for each $i\geq 1$, we have  
 \begin{enumerate}
  \item Each isomorphism has the form $h_i=(id_M,h^1_i)$;
  \item $h^1_{i}(\overline{C}_{i-2})=0$; and 
  \item For $f_i:=fh_1h_2\cdots h_i$, we have $f_i^1(\overline{C}_{i-1})=0$.
 \end{enumerate}}

\emph{Proof of the Claim.} The inductive argument is essentially the same as the one given in \cite{BSZ}(5.7), but we rephrase it here for the sake of the reader. At the base of the induction, we have the isomorphism  $h_1=(id_M,0):M\rightmap{}M$ such that $h_1^1(\overline{C}_{-1})=0$ and the morphism $f_1=f$ is such that $f_1^1(\overline{C}_0)=0$. 

Assume that we have constructed the isomorphisms $h_1,\ldots,h_i$ satisfying conditions \emph{1--3} and define 
$h_{i+1}^1(c):=-t^0f_i^1(c)$, for $c\in \overline{C}$.

Clearly $h_{i+1}^1(\overline{C}_{i-1})=0$ and $h_{i+1}^1\in\Hom^0_{\GM\g S\g S}(\overline{C},\Hom_{\GM\g k}(M,M))$. Then, the pair $h_{i+1}=(id_M,h_{i+1}^1):M\rightmap{} M$ is a morphism in $\GModit^0 \g {\cal B}$. 
From (\ref{C: f iso sii f0 es iso}), we know that $h_{i+1}$ is an isomorphism in $\GModit^0 \g{\cal B}$. 
Moreover, for $c\in \overline{C}_{i}$, we have $\overline{\mu}(c)=\sum_tc_t^1\otimes c_t^2$, with $c_t^1,c_t^2\in \overline{C}_{i-1}$. Notice also that $f_i^0=(fh_1\cdots h_i)^0=f^0$ and $h_{i+1}^0=id_M$. Then, we have 
$$\begin{matrix}
  f_{i+1}^1(c)&=&(f_ih_{i+1})^1(c)\hfill\\
  &=&
  f_i^0h_{i+1}^1(c)+f^1_i(c)h_{i+1}^0+\sum_tf_i^1(c_t^2)h_{i+1}^1(c_t^1)\hfill\\
  &=&
  -f^0t^0f_i^1(c)+f_i^1(c)=0,\hfill\\
  \end{matrix}$$ as we wanted. \hfill$\square$
\medskip

From (1) and (2), we obtain, for $i\geq 1$ and $c\in \overline{C}_{i-1}$, the equality
$$(h_1h_2\cdots h_{i+1})^1(c)=(h_1h_2\cdots h_i)^1(c).$$
Indeed, this is clear for $i=1$. For $i\geq 2$, we have $(h_1h_2\cdots h_i)^0=id_M$ and  $h^0_{i+1}=id_M$. By assumption, $\overline{\mu}(c)=\sum_tc^1_t\otimes c^2_t$, with $c^1_t,c^2_t\in \overline{C}_{i-2}$. Then, we obtain
$$\begin{matrix}
   (h_1h_2\cdots h_{i+1})^1(c)
   &=&
   ((h_1h_2\cdots h_i)h_{i+1}))^1(c)\hfill\\
   &=&
   (h_1h_2\cdots h_i)^1(c)+h_{i+1}^1(c)\hfill\\
   &&+\,\sum_t  (h_1h_2\cdots h_i)^1 (c_t^2)h^1_{i+1}(c_t^1)\hfill\\
   &=&
   (h_1h_2\cdots h_i)^1(c).\hfill\\
  \end{matrix}$$
Then, we can consider the map $h^1:\overline{C}\rightmap{}\Hom_{\GM\g k}(M,M)$ defined by 
$$h^1(c)=(h_1h_2\cdots h_i)^1(c), \hbox{ for } c\in \overline{C}_{i-1}.$$
Since $h^1$ is a morphism of graded $S$-$S$-bimodules, we can consider the isomorphism $h=(id_M,h^1):M\rightmap{}M$ in $\GModit^0 \g{\cal B}$. 
For $c\in \overline{C}$, say with $c\in \overline{C}_{i-1}$, we have 
$$(fh)^1(c)=(fh_1h_2\cdots h_i)^1(c)=f^1_i(c)=0.$$
\end{proof}

\begin{proposition}\label{P: estructura exacta de GMod-B}
Let ${\cal B}$ be a triangular differential  $S$-bocs. Denote by ${\cal E}_0$ the class of composable pairs $M\rightmap{f}E\rightmap{g}N$ in $\GModi^0\g {\cal B}$ such that $g*f=0$ and 
$$0\rightmap{}M\rightmap{f^0}E\rightmap{g^0}N\rightmap{}0$$
is an exact sequence in $\GM\g S$. Then we have:
\begin{enumerate}
 \item The pair  $(\GModi^0\g{\cal B},{\cal E}_0)$ is an exact category. The class ${\cal E}_0$ consists of the split exact pairs in $\GModi^0\g{\cal B}$. 
 \item A morphism $f:M\rightmap{}E$ in $\GModi^0\g{\cal B}$ is an ${\cal E}_0$-inflation iff $f^0:M\rightmap{}E$ is injective.
 \item A morphism $g:E\rightmap{}N$ in $\GModi^0\g {\cal B}$ is an  ${\cal E}_0$-deflation iff $g^0:E\rightmap{}N$ is surjective.
\end{enumerate}
\end{proposition}

\begin{proof} By (\ref{L: idemptes se div en GM-B}), we already know in the category $\GModit^0 \g{\cal B}$ idempotents split.  Follow the argument of the proof of \cite{BSZ}(6.6 and 6.7), using 
(\ref{L: embedding preserves exactness}) and (\ref{L: B es acceptable}), to show that 
$(\GModi^0\g{\cal B},{\cal E}_0)$ is an exact category. From (\ref{P: conflations in GMod-B}), 
(\ref{C: f iso sii f0 es iso}), and (\ref{L: embedding preserves exactness}) we get that any composable pair in 
${\cal E}_0$ is a split exact pair in $\GModi^0\g{\cal B}$. 
\end{proof}

\section{The Frobenius category of twisted modules}

Given a triangular differential graded $S$-bocs ${\cal B}=(C,\mu,\epsilon,\delta)$ we will describe a natural structure of a Frobenius category on $\TGModi^0\g{\cal B}$, in the following sense. 

 \begin{definition}\label{D: Frobenius category}
 Let ${\cal A}$ be  an additive $k$-category  where idempotents split, 
  and let ${\cal E}$ be an exact structure on ${\cal A}$.   
  Then $({\cal A},{\cal E})$ is called a \emph{Frobenius category}\index{Frobenius category} iff 
  it has enough ${\cal E}$-projectives and enough ${\cal E}$-injectives and, moreover, 
  the class of ${\cal E}$-projectives coincides with the class of ${\cal E}$-injectives. 
 \end{definition}
 
 \begin{definition}\label{D: la clase de pares componibles cal E}
  Consider the class ${\cal E}$ of composable morphisms in $\TGModi^0\g {\cal B}$ 
 $$\begin{matrix}(M,u)&\rightmap{f}&(E,v)&\rightmap{g}&(N,w)\end{matrix}$$
 such that $g*f=0$ and the short sequence of first components  
 $$\begin{matrix}0&\rightmap{}&M&\rightmap{f^0}&E&\rightmap{g^0}&N&\rightmap{}&0\end{matrix}$$
 is an exact sequence in $\GM\g S$. Equivalently, 
 the composable pair  
 $$M\rightmap{f}E\rightmap{g}N$$
 is a split exact pair in $\GModi^0\g{\cal B}$. 
 \end{definition}

 \begin{lemma}\label{L: split underlying pairs produce exact pairs}
 Any composable pair of morphisms in $\TGModi^0\g{\cal B}$ of the form
 $$\begin{matrix}(M_1,u_1)&\rightmap{ \ s=(\hueca{I}_{M_1},0)^t \ }&(M_1\oplus M_2,v)&
 \rightmap{ \ p=(0,\hueca{I}_{M_2}) \ }&(M_2,u_2)\end{matrix}$$
 belongs to ${\cal E}$ and is an exact pair in $\TGModi^0\g{\cal B}$. 
 \end{lemma}

 \begin{proof} Clearly $M_1\rightmap{s}M_1\oplus M_2\rightmap{p}M_2$ is a split exact pair in $\GModi^0\g{\cal B}$. 
 
 In order to show that $s$ is the kernel of $p$ in $\TGModi^0\g{\cal B}$, take any morphism 
 $t:(N,w)\rightmap{}(M_1\oplus M_2,v)$ in $\TGModi^0\g{\cal B}$ such that $p*t=0$. Then, there is a morphism 
 $t':N\rightmap{}M_1$ in $\GModi^0\g{\cal B}$ such that $t=s*t'$. 
 
 We have $\widehat{\delta}(t)=\widehat{\delta}(s*t')=s*\widehat{\delta}(t')$. Since $s$ is a morphism in $\TGModi^0\g{\cal B}$, we have 
 $\widehat{\delta}(s)+v*s-s*u_1=v*s-s*u_1$, so $v*s=s*u_1$. We also have that 
 $\widehat{\delta}(t)+v*t-t*w=0$. Then, 
  $$\begin{matrix}
    s*[\widehat{\delta}(t')+u_1*t'-t'*w]
    &=&
    s*\widehat{\delta}(t')+s*u_1*t'-s*t'*w\hfill\\
    &=&
    \widehat{\delta}(t)+v*s*t'-s*t'*w\hfill\\
    &=&
 \widehat{\delta}(t)+v*t-t*w=0.\hfill\\
    \end{matrix}$$
Since $s$ is a monomorphism in $\GModi^0\g {\cal B}$, we obtain $\widehat{\delta}(t')+u_1*t'-t'*w=0$. So $t':(N,w)\rightmap{}(M_1,u_1)$ is a morphism in $\TGModi^0\g{\cal B}$. Hence, $s$ is the kernel of $p$ in $\TGModi^0\g{\cal B}$. The proof of the fact that the cokernel of $s$ is $p$ in $\TGModi^0\g{\cal B}$ is dual.
 \end{proof}

 \begin{lemma}\label{L: deflations and inflations in TGMod-B}
 Let ${\cal B}$ be a triangular differential $S$-bocs. Then, we have:
 \begin{enumerate}
  \item  A morphism $g:(M,u)\rightmap{}(M_2,u_2)$ in $\TGModi^0\g {\cal B}$ is an $\cal E$-deflation iff $g:M\rightmap{}M_2$ is a retraction in $\GModi^0\g {\cal B}$. 
  \item A morphism $f:(M_1,u_1)\rightmap{}(M,u)$ in $\TGModi^0\g {\cal B}$ is an $\cal E$-inflation  iff $f:M_1\rightmap{}M$ is a  section in $\GModi^0\g {\cal B}$. 
  \item Every composable pair in ${\cal E}$ is an exact pair in $\TGModi^0\g {\cal B}$.
 \end{enumerate}
 \end{lemma}
 
 \begin{proof} (1): Assume that $g:M\rightmap{}M_2$ is a retraction in $\GModi^0\g {\cal B}$. Since in 
 $\GModi^0\g {\cal B}$ idempotents split, the retraction $g$ has a kernel and, moreover, 
 there is a commutative diagram in $\GModi^0\g {\cal B}$ 
 $$\begin{matrix}
 M_1  &\rightmap{ \ f \ }&M&\rightmap{ \ g \ }&M_2\\
  \parallel&& \rmapdown{h}&&\parallel\\
     M_1&\rightmap{s=(\hueca{I}_{M_1},0)^t}&  M_1\oplus M_2&\rightmap{p=(0,\hueca{I}_{M_2})}&M_2,\\
   \end{matrix}$$
 where 
 $h:M\rightmap{}M_1\oplus M_2$  is an isomorphism. 
 From (\ref{L: iso en GMod-B de M en N copia la u de M en una de N}), we know there is $v\in\Hom^1_{\GM\g {\cal B}}(M_1\oplus M_2,M_1\oplus M_2)$ such that 
 $h:(M,u)\rightmap{}(M_1\oplus M_2,v)$ is an isomorphism in $\TGModi^0\g{\cal B}$. Therefore, $g*h^{-1}:(M_1\oplus M_2,v)\rightmap{}(M_2,u_2)$ is a morphism in $\TGModi^0\g{\cal B}$, so  $\widehat{\delta}(g*h^{-1})+u_2*(g*h^{-1})-(g*h^{-1})*v=0$. 
 Since $\widehat{\delta}(g*h^{-1})=\widehat{\delta}(p)=0$, we obtain  
 $u_2*(g*h^{-1})=(g*h^{-1})*v$. This implies that $v=\begin{pmatrix}u_1&x\\ 0&u_2\end{pmatrix}$, where $u_1\in \Hom^1_{\GModi\g{\cal B}}(M_2,M_2)$ and $x\in \Hom^1_{\GModi\g{\cal B}}(M_2,M_1)$. Moreover, the equation 
 $\widehat{\delta}(v)+v*v=0$ implies that $(M_1,u_1)\in \TGModi\g {\cal B}$. Notice also that 
 $s:(M_1,u_1)\rightmap{}(M_1\oplus M_2,v)$ is a morphism in $\TGModi^0\g {\cal B}$, because 
 $$\widehat{\delta}(s)+v *s-s* u_1=\begin{pmatrix}u_1&x\\ 0&u_2\end{pmatrix}*
 \begin{pmatrix} \hueca{I}_{M_1}\\ 0\end{pmatrix}-\begin{pmatrix}\hueca{I}_{M_1}\\ 0\end{pmatrix}*u_1=0.$$
 Then, the composable pair 
$(M_1,u_1)  \rightmap{ \ f \ }(M,u)\rightmap{ \ g \ }(M_2,u_2)$
in the category $\TGModi^0\g {\cal B}$ belongs to ${\cal E}$ and $g:(M,u)\rightmap{ }(M_2,u_2)$ is an ${\cal E}$-deflation.
  
  The proof of (2) is similar. 
  
  (3): Consider any composable pair  
  $(L,w)\rightmap{f}(M,u)\rightmap{g}(M_2,u_2)$ in ${\cal E}$. 
  Then, $g:M\rightmap{}M_2$ is a retraction in $\GModi^0\g {\cal B}$, so we can apply the preceding argument to construct the following commutative diagram in  $\TGModi^0\g {\cal B}$
  $$\begin{matrix}
 (L,w)  &\rightmap{ \ f \ }&(M,u)&\rightmap{ \ g \ }&(M_2,u_2)\\
  && \rmapdown{h}&&\parallel\\
     (M_1,u_1)&\rightmap{s=(\hueca{I}_{M_1},0)^t}& (M_1\oplus M_2,v)&\rightmap{p=(0,\hueca{I}_{M_2})}&(M_2,u_2),\\
   \end{matrix}$$
   where $h$ is an isomorphism. From (\ref{L: split underlying pairs produce exact pairs}), we know that the second row is an exact pair in $\TGModi^0\g {\cal B}$. Then, since $p*h*f=0$, there is a morphism $t:(L,w)\rightmap{}(M_1,u_1)$ in $\TGModi^0\g{\cal B}$ such that $s*t=h*f$.  If we consider the first components of the underlying diagram in $\GM\g S$, we have the following commutative diagram where the vertical arrows are isomorphisms
   $$\begin{matrix}
 L &\rightmap{ \ f^0 \ }&M&\rightmap{ \ g^0 \ }&M_2\\
  \rmapdown{t^0}&& \rmapdown{h^0}&&\parallel\\
     M_1&\rightmap{s^0}& M_1\oplus M_2&\rightmap{p^0}&M_2.\\
   \end{matrix}$$
   Hence $t:(L,w)\rightmap{}(M_1,u_1)$ is an isomorphism in $\TGModi^0\g{\cal B}$. It follows that the composable pair 
   $(L,w)\rightmap{f}(M,u)\rightmap{g}(M_2,u_2)$ is exact. 
 \end{proof}

 \begin{proposition}\label{P: estructura exacta de TGMod-B}
 Let ${\cal B}$ be a triangular $S$-bocs, then we have the following.
\begin{enumerate}
 \item The pair  $(\TGModi^0\g{\cal B},{\cal E})$ is an exact category.
 \item A morphism $f:(M,u_M)\rightmap{}(E,u_E)$ in $\TGModi^0\g{\cal B}$ is an ${\cal E}$-inflation iff $f^0:M\rightmap{}E$ is injective.
 \item A morphism $g:(E,u_E)\rightmap{}(N,u_N)$ in $\TGModi^0\g {\cal B}$ is an ${\cal E}$-deflation iff $g^0:E\rightmap{}N$ is surjective.
\end{enumerate} 
 \end{proposition}

 \begin{proof} From the last section, we know that idempotents split in 
 $\TGModi^0\g {\cal B}$. 
 The description of the ${\cal E}$-inflations and the ${\cal E}$-deflations follows from (\ref{L: deflations and inflations in TGMod-B}) and (\ref{P: estructura exacta de GMod-B}). The fact that ${\cal E}$ is a class of exact pairs closed under isomorphisms, follows from (\ref{L: deflations and inflations in TGMod-B}). 
 
 Consider a  
 morphism $f:\underline{Z}'\rightmap{}\underline{Z}$ and a deflation $d:\underline{Y}\rightmap{}\underline{Z}$ in $\TGModi^0\g{\cal B}$.
 Consider the morphism $(f,d):\underline{Z}'\bigoplus \underline{Y}\rightmap{}\underline{Z}$ in $\TGModi^0\g {\cal B}$.
 Since $d:\underline{Y}\rightmap{}\underline{Z}$ is a deflation, the morphism $d:Y\rightmap{}Z$ is a retraction in $\GModi^0\g{\cal B}$ and, therefore, so is $(f,d):Z'\oplus Y\rightmap{}Z$. 
 Then, $(f,d):\underline{Z}'\bigoplus \underline{Y}\rightmap{}\underline{Z}$ is a deflation in $\TGModi^0\g{\cal B}$ and it 
 appears in an exact pair of  ${\cal E}$
 $$\begin{matrix} \underline{Y}'&\rightmap{ \ (\widehat{d},  -\widehat{f})^t \ }& \underline{Z}'\bigoplus \underline{Y}&\rightmap{ \ (f,d) \ }&\underline{Z}.\end{matrix}$$
 Therefore, we have a pull-back diagram in the category $\TGModi^0\g{\cal B}$
$$\begin{matrix}
 \underline{Y}'& \rightmap{\widehat{d}}& \underline{Z}'\\ 
 \lmapdown{\widehat{f}}&&\rmapdown{f}\\ 
\underline{Y}& \rightmap{d}& \underline{Z},\\ 
\end{matrix}$$
and an exact sequence in $\GM\g S$ given by the first components 
$$\begin{matrix} Y'&\rightmap{ \ (\widehat{d}^0,  -\widehat{f}^0)^t \ }&Z'\bigoplus Y&\rightmap{ \ (f^0,d^0) \ }&Z.\end{matrix}$$
So we have the pullback-diagram in the category $\GM\g S$ 
$$\begin{matrix}
 Y'& \rightmap{\widehat{d}^0}& Z'\\ 
 \lmapdown{\widehat{f}^0}&&\rmapdown{f^0}\\ 
Y& \rightmap{d^0}& Z.\\ 
\end{matrix}$$
Since $d^0$ is surjective, then  $\widehat{d}^0$ is surjective, thus $\widehat{d}:Y'\rightmap{}Z'$ is a retraction, and 
$\hat{d}:\underline{Y}'\rightmap{}\underline{Z}'$ is an ${\cal E}$-deflation. 

The other requirements in the definition of an exact structure are easy to verify for ${\cal E}$ using (\ref{L: deflations and inflations in TGMod-B}). 
 \end{proof}

 We shall see that $\TGM^0\g {\cal B}$  is a Frobenius category with  the
 following additional structure. 
 
 \begin{definition}\label{D: Especial Frobenius category}
 A Frobenius category  $({\cal A},{\cal E})$ is called \emph{special}\index{special Frobenius category} iff 
 there are an exact automorphism $T:{\cal A}\rightmap{}{\cal A}$ and 
 an endofunctor $J:{\cal A}\rightmap{}{\cal A}$ such that 
 \begin{enumerate}
  \item $J(M)$ is ${\cal E}$-projective, for any $M\in {\cal A}$.
  \item There are natural transformations $\alpha:id_{\cal A}\rightmap{}J$ 
  and $\beta: J\rightmap{}T$ such that, for each $M\in {\cal A}$ the pair 
  $$\begin{matrix}\xi_M&:&M&\rightmap{\alpha_M}&J(M)&\rightmap{\beta_M}&T(M)\\ \end{matrix},$$
  is an exact pair in ${\cal E}$. 
 \end{enumerate}
 \end{definition}

\begin{remark}\label{R: remark sobre sigmaM para GMod-B}
 In the following lemma, recall that we have already defined, for each graded $S$-module $M$ its 
  shifting $M[1]$. 
 Thus $M[1]=\bigoplus_{i\in \hueca{Z}}M[1]_i$, with  $M[1]_i=M_{i+1}$, for $i\in \hueca{Z}$. 
 We use also the canonical morphism $\sigma_M:M\rightmap{} M[1]$ of degree $-1$  determined by 
 the identity map on $M$. The isomorphism $\underline{\sigma}_M:=L(\sigma_M)\in \Hom^{-1}_{\GM\g{\cal B}}(M,M[1])$ plays an important role in the following.  
 \end{remark}

 \begin{lemma}\label{D: traslacion para GM-B y TGM-B} 
 Let ${\cal B}$ be a triangular differential $S$-bocs.
 Then, we have:
 \begin{enumerate}
 \item For each $M\in \GM\g {\cal B}$, make $T(M)=M[1]$ and, 
 for any morphism $f:M\rightmap{}N$ in $\GM\g {\cal B}$, make 
 $$T(f)=f[1]=\underline{\sigma}_N*f*\underline{\sigma}_M^{-1}:M[1]\rightmap{}N[1] \hbox{ in } \GM\g{\cal B}.$$
 Then, we have a degree preserving  autofunctor 
 $T:\GM\g{\cal B}\rightmap{}\GM\g{\cal B}$ with inverse $T^{-1}$ given by $T^{-1}(M)=M[-1]$ and 
 $$T^{-1}(f)=f[-1]=\underline{\sigma}_{N[-1]}^{-1}*f*\underline{\sigma}_{M[-1]}:M[-1]\rightmap{}N[-1] \hbox{ in } \GM\g{\cal B}.$$ 
 It is called \emph{the shifting autofunctor of $\GM\g{\cal B}$.} 
 \item  For each $(M,u)\in \TGM\g {\cal B}$, make $T(M,u)=(M[1],-u[1])$ and, 
 for any morphism $f:(M,u)\rightmap{}(N,v)$ in $\TGM\g {\cal B}$, make 
 $$T(f)=f[1]=\underline{\sigma}_N*f*\underline{\sigma}_M^{-1}:(M[1],-u[1])\rightmap{}(N[1],-v[1]) \hbox{ in } \TGM\g{\cal B}.$$
 Then, we have a degree preserving  autofunctor 
 $T:\TGM\g{\cal B}\rightmap{}\TGM\g{\cal B}$ with inverse $T^{-1}$ given by $T^{-1}(M,u)=(M[-1],-u[-1])$ and 
 $$T^{-1}(f)=f[-1]=\underline{\sigma}_{N[-1]}^{-1}*f*\underline{\sigma}_{M[-1]}:(M[-1],-u[-1])\rightmap{}(N[-1],-v[-1]) .$$ 
 The functor $T$ is called \emph{the shifting autofunctor of $\TGM\g {\cal B}$.} 
\item Furthermore, the functor $T:(\TGM^0\g {\cal B},{\cal E})\rightmap{}(\TGM^0\g {\cal B},{\cal E})$ is an  
 automorphism of exact categories, that is   
 $$T(M)\rightmap{ \ T(f) \ }T(E)\rightmap{ \ T(g) \ }T(N)\in {\cal E} \hbox{ \ if and only if \  } 
M\rightmap{f}E\rightmap{g}N\in {\cal E}.$$ 
As a consequence, the functors $T$ and $T^{-1}$ preserve the clases of ${\cal E}$-injectives and ${\cal E}$-projectives. 
\end{enumerate}
 \end{lemma}

\begin{proof} The proof of (1) is easy. In order to prove (2), notice that  
given $h:M\rightmap{}N$ in $\GM\g{\cal B}$, using that $\widehat{\delta}(\underline{\sigma}_M)=0$ and Leibniz rule, we have 
$$\widehat{\delta}(h[1])=\widehat{\delta}[ \underline{\sigma}_N*h*\underline{\sigma}_M^{-1}]=
 -\underline{\sigma}_N*\widehat{\delta}(h)*\underline{\sigma}_M^{-1}=-\widehat{\delta}(h)[1].$$
 Then, given $(M,u)\in \TGM\g{\cal B}$ we have $\widehat{\delta}(u)+u*u=0$. Thus, 
 $$\widehat{\delta}(-u[1])+(-u[1])*(-u[1])=\widehat{\delta}(u)[1]+u[1]*u[1]=(\widehat{\delta}(u)+u*u)[1]=0,$$
 and $(M[1],-u[1])\in \TGM\g {\cal B}$. Moreover, given a morphism 
 $f:(M,u) \rightmap{}(N,v)$ in $\TGM\g {\cal B}$ we have 
 $$\begin{matrix}
\widehat{\delta}(f[1])+(-v[1])*f[1]-f[1]*(-u[1])
&=&
-\widehat{\delta}(f)[1]-v[1]*f[1]+f[1]*u[1]\hfill\\
&=&
-(\widehat{\delta}(f)+v*f-f*u)[1]=0,\hfill\\
 \end{matrix}$$
 so $f[1]:(M[1],-u[1])\rightmap{}(N[1],-v[1])$ is a morphism in $\TGM\g{\cal B}$. Now, it is clear that $T$ is an endofunctor of $\TGM\g{\cal B}$. It is easy to see that $T$ is an autofunctor with inverse functor $T^{-1}$. 
 
 (3): Consider a composable pair $(M,u)\rightmap{f}(E,v)\rightmap{g}(N,w)$ in the category $\TGM^0\g{\cal B}$ and its image  
 $$(M[1],-u[1])\rightmap{f[1]}(E[1],-v[1])\rightmap{g[1]}(N[1],-w[1])$$
 under the functor $T$. Clearly $g*f=0$ iff $g[1]*f[1]=0$. Notice that for any morphism $h:M\rightmap{}N$ in 
 $\GM\g S$, the first component $L(h)^0$ of $L(h)$ is precisely $h$. So, $\underline{\sigma}_M^0=\sigma_M$, for any $M\in \GM\g S$. Hence, we have 
 $(f[1])^0
 =(\underline{\sigma}_E*f*\underline{\sigma}_M^{-1})^0
 =\underline{\sigma}^0_Ef^0(\underline{\sigma}^0_M)^{-1}
 =\sigma_Ef^0\sigma_M^{-1}$ and similarly for $g$. So we have the commutative diagram in $\GM\g S$
 $$\begin{matrix}
   0&\rightmap{}&M&\rightmap{f^0}&E&\rightmap{g^0}&N&\rightmap{}&0\\
   &&\rmapdown{\sigma_M}&&\rmapdown{\sigma_E}&&\rmapdown{\sigma_N}&&\\
    0&\rightmap{}&M[1]&\rightmap{f[1]^0}&E[1]&\rightmap{g[1]^0}&N[1]&\rightmap{}&0,\\
   \end{matrix}$$
   where the vertical morphisms are isomorphisms. Hence, the first row is exact iff the second one is so.
\end{proof}

 \begin{proposition}\label{P: construction of J for TGMod-B} 
 For $(M,u_M)\in \TGM\g {\cal B}$, we write $u_{M[1]}:=-u_M[1]$, so we have $T(M,u_M)=(M[1],u_{M[1]})$. 
 There is a degree preserving endofunctor  $$J: \TGM\g {\cal B}\rightmap{}\TGM\g {\cal B}$$ 
 such that, for any $(M,u_M)\in \TGM\g{\cal B}$ we have 
 $$J(M,u_M)=(M\oplus M[1],\begin{pmatrix}
                            u_M&\underline{\sigma}^{-1}_M\\ 0&u_{M[1]}\\
                           \end{pmatrix})$$
 and,  given a morphism $f:(M,u_M)\rightmap{}(N,u_N)$ in $\TGM\g {\cal B}$, the morphism 
 $J(f)$ is   given by the matrix 
 $$J(f)=\begin{pmatrix}
               f& 0\\ 0& f[1]
              \end{pmatrix}: J(M,u_M)\rightmap{}J(N,u_N).$$
 \end{proposition}

\begin{proof} Given $(M,u)\in \TGM\g {\cal B}$, we have 
$$\widehat{\delta}\begin{pmatrix}
  u_M&\underline{\sigma}^{-1}_M\\ 0&u_{M[1]}\\
\end{pmatrix}
+
\begin{pmatrix}
  u_M&\underline{\sigma}^{-1}_M\\ 0&u_{M[1]}\\
\end{pmatrix}*
\begin{pmatrix}
  u_M&\underline{\sigma}^{-1}_M\\ 0&u_{M[1]}\\
  \end{pmatrix}$$
$$=\begin{pmatrix}
 \widehat{\delta}(u_M)+ u_M*u_M&u_M*\underline{\sigma}_M^{-1}+\underline{\sigma}_M^{-1}*u_{M[1]}\\
 0&\widehat{\delta}(u_{M[1]})+u_{M[1]}*u_{M[1]}\\
 \end{pmatrix}=\begin{pmatrix}
 0&0\\ 0&0\end{pmatrix}.
$$
So, indeed, we have  $J(M,u)\in \TGM\g{\cal B}$. 

Given a morphism $f:(M,u_M)\rightmap{}(N,u_N)$ in $\TGM\g{\cal B}$, we have 
$$\widehat{\delta}\begin{pmatrix}
  f&0\\ 0&f[1]\\
\end{pmatrix}
+
\begin{pmatrix}
  u_N&\underline{\sigma}^{-1}_N\\ 0&u_{N[1]}\\
\end{pmatrix}*\begin{pmatrix}
  f&0\\ 0&f[1]\\
\end{pmatrix}-\begin{pmatrix}
  f&0\\ 0&f[1]\\
\end{pmatrix}*
\begin{pmatrix}
  u_M&\underline{\sigma}^{-1}_M\\ 0&u_{M[1]}\\
  \end{pmatrix}$$
$$=\begin{pmatrix}
 \widehat{\delta}(f)+ u_N*f-f*u_M&
 \underline{\sigma}_N^{-1}*f[1]-f*\underline{\sigma}_M^{-1}\\
 0&\widehat{\delta}(f[1])+u_{N[1]}*f[1]-f[1]*u_{M[1]}\\
 \end{pmatrix}=\begin{pmatrix}
 0&0\\ 0&0\end{pmatrix}.
$$
So, indeed, we have that $J(f):J(M,u_M)\rightmap{}J(N,u_N)$ is a morphism in $\TGM\g{\cal B}$. Since $T: \GM\g{\cal B}\rightmap{}\GM\g{\cal B}$ is a $k$-functor, so is $J$. 
\end{proof}

 \begin{proposition}\label{P: isos naturales de morfismos en TGMod-B en GMod-B}
 For $(M,u_M),(N,u_N)\in \TGM\g {\cal B}$, we have natural isomorphisms  
 $$\eta_1:\Hom^0_{\TGM\g {\cal B}}((M,u_M),J(N,u_N))\rightmap{}\Hom^0_{\GM\g {\cal B}}(M,N)$$
 and 
 $$\eta_2:\Hom^0_{\TGM\g {\cal B}}(J(M,u_M),(N,u_N))\rightmap{}\Hom^0_{\GM\g {\cal B}}(M[1],N).$$
 Any morphism of twisted graded ${\cal B}$-modules $f=(f_1,f_2)^t:(M,u_M)\rightmap{}J(N,u_N)$, with underlying codomain 
 $N\oplus N[1]$ in $\GM\g {\cal B}$  
 is mapped on  $\eta_1(f)=f_1$, and any morphism $g=(g_1,g_2):J(M,u_M)\rightmap{}(N,u_N)$ of twisted graded ${\cal B}$-modules, 
 with undelying domain 
 $M\oplus M[1]$  in $\GM\g {\cal B}$  is mapped on  $\eta_2(g)=g_2$.
 \end{proposition}

\begin{proof} (1): Let $f:(M,u_M)\rightmap{}J(N,u_N)$ be a homogeneous morphism of twisted graded ${\cal B}$-modules with degree $0$. Thus,   $f=(f_1,f_2)^t:M\rightmap{}N\oplus N[1]$ is a morphism in $\GM\g{\cal B}$ which satisfies the equation
$$\widehat{\delta}\begin{pmatrix}f_1\\ f_2\end{pmatrix}+\begin{pmatrix}u_N&\underline{\sigma}_N^{-1}\\
                                                        0&u_{N[1]}\\ \end{pmatrix}*\begin{pmatrix}
                                                        f_1\\ f_2\end{pmatrix}  -\begin{pmatrix}
                                                        f_1\\ f_2\end{pmatrix}*u_M=\begin{pmatrix}0\\ 0\end{pmatrix}.$$
Equivalently, it satisfies the equations:
$$\begin{matrix}
   (a_1) &\ & 0&=&\widehat{\delta}(f_1)+u_N*f_1+\underline{\sigma}_N^{-1}*f_2-f_1*u_M\hfill \\
   (b_1) & \ & 0&=& \widehat{\delta}(f_2)+u_{N[1]}*f_2-f_2*u_M\hfill\\
  \end{matrix}$$
From the equality $(a_1)$, we obtain the following expression of $f_2$ in terms of $f_1$
$$f_2=\underline{\sigma}_N*f_1*u_M-\underline{\sigma}_N*\widehat{\delta}(f_1)-\underline{\sigma}_N*u_N*f_1.$$
Thus, the linear map $\eta_1$ is injective.

Let us see that $\eta_1$ is surjective. For this, take $f_1\in \Hom^0_{\GM\g {\cal B}}(M,N)$ and define $f_2$ by the equality $(a_1)$, so  $f:=(f_1,f_2)^t\in \Hom^0_{\GM\g{\cal B}}(M,N\oplus N[1])$. Then, 
$$f\in \Hom^0_{\TGM\g {\cal B}}((M,u_M),J(N,u_N)) \hbox{ iff equality } (b_1) \hbox{ holds. }$$ 
We have the following  
$$\begin{matrix}
          f_2*u_M&=&
          \underline{\sigma}_N*f_1*(u_M*u_M)-\underline{\sigma}_N*\widehat{\delta}(f_1)*u_M-\underline{\sigma}_N*u_N*f_1*u_M\hfill\\          
          &=&
           -\underline{\sigma}_N*[f_1*\widehat{\delta}(u_M)+\widehat{\delta}(f_1)*u_M+u_N*f_1*u_M]\hfill\\      
          \end{matrix}$$
$$\begin{matrix}
  u_{N[1]}*f_2&=&
 u_{N[1]}*\underline{\sigma}_N*f_1*u_M-u_{N[1]}*\underline{\sigma}_N*\widehat{\delta}(f_1)-u_{N[1]}*\underline{\sigma}_N*u_N*f_1\hfill\\
 &=&
 -\underline{\sigma}_N*u_N*f_1*u_M
 +\underline{\sigma}_N*u_N*\widehat{\delta}(f_1)+\underline{\sigma}_N*u_{N}*u_N*f_1\hfill\\
  &=&
 \underline{\sigma}_N*[-u_N*f_1*u_M
 +u_N*\widehat{\delta}(f_1)-\widehat{\delta}(u_N)*f_1]\hfill\\
 
  \end{matrix}$$
and 
$$\begin{matrix}\widehat{\delta}(f_2)
&=&
   -\underline{\sigma}_N*\widehat{\delta}(f_1)*u_M-\underline{\sigma}_N*f_1*\widehat{\delta}(u_M)
   +\underline{\sigma}_N*\widehat{\delta}(u_N)*f_1-\underline{\sigma}_N*u_N*\widehat{\delta}(f_1)\hfill\\
&=&
   \underline{\sigma}_N*[-\widehat{\delta}(f_1)*u_M-f_1*\widehat{\delta}(u_M)
   +\widehat{\delta}(u_N)*f_1-u_N*\widehat{\delta}(f_1)].\hfill\\
  \end{matrix}$$
Then, we have that equation $(b_1)$ holds and $\eta_1(f)=f_1$. So $\eta_1$ is an isomorphism, which is clearly a natural transformation. 
\medskip

(2) Let $g:J(M,u_M)\rightmap{}(N,u_N)$ be a homogeneous morphism of twisted graded ${\cal B}$-modules with degree $0$. Thus,   $g=(g_1,g_2):M\oplus M[1]\rightmap{}N$ is a morphism in $\GM\g{\cal B}$ which satisfies the equation
$$\widehat{\delta}\begin{pmatrix}g_1& g_2\end{pmatrix}+u_N*\begin{pmatrix}
                                                        g_1 & g_2\end{pmatrix}-\begin{pmatrix}
                                                        g_1 & g_2\end{pmatrix}*\begin{pmatrix}u_M&\underline{\sigma}_M^{-1}\\
                                                        0&u_{M[1]}\\ \end{pmatrix}=\begin{pmatrix}0 & 0\end{pmatrix}  .$$
Equivalently, it satisfies the equations:
$$\begin{matrix}
   (a_2) &\ & 0&=&\widehat{\delta}(g_1)+u_N*g_1-g_1*u_M\hfill \\
   (b_2) & \ & 0&=& \widehat{\delta}(g_2)+u_N*g_2-g_1*\underline{\sigma}_M^{-1}-g_2*u_{M[1]}\hfill\\
  \end{matrix}$$
From the equality $(b_2)$, we obtain the following expression of $g_1$ in terms of $g_2$
$$g_1=\widehat{\delta}(g_2)*\underline{\sigma}_M+u_N*g_2*\underline{\sigma}_M-g_2*u_{M[1]}*\underline{\sigma}_M.$$
Thus, the linear map $\eta_2$ is injective.

Let us see that $\eta_2$ is surjective. For this, take $g_2\in \Hom^0_{\GM\g {\cal B}}(M[1],N)$ and define $g_1$ by the equality $(b_2)$, so  $g:=(g_1,g_2)\in \Hom^0_{\GM\g{\cal B}}(M\oplus M[1],N)$. Then, $$g\in \Hom^0_{\TGM\g {\cal B}}(J(M,u_M),(N,u_N)) \hbox{ iff equality } (a_2) \hbox{ holds. }$$ 
We have the equalities 
$$\begin{matrix}
          g_1*u_M
          &=&
          \widehat{\delta}(g_2)*\underline{\sigma}_M*u_M
          +
          u_N*g_2*\underline{\sigma}_M*u_M
          -
          g_2*u_{M[1]}*\underline{\sigma}_M*u_M\hfill\\     
          &=&
             -\widehat{\delta}(g_2)*u_{M[1]}*\underline{\sigma}_M
             -
             u_N*g_2*u_{M[1]}*\underline{\sigma}_M
             +
             g_2*u_{M[1]}*u_{M[1]}*\underline{\sigma}_M\hfill\\     
           &=&
             -[\widehat{\delta}(g_2)*u_{M[1]}
             +
             u_N*g_2*u_{M[1]}
             +
             g_2*\widehat{\delta}(u_{M[1]})]*\underline{\sigma}_M\hfill\\            
          \end{matrix}$$
$$\begin{matrix}
  u_N*g_1
  &=&
  u_N*\widehat{\delta}(g_2)*\underline{\sigma}_M
  +
  u_N*u_N*g_2*\underline{\sigma}_M
  -
  u_N*g_2*u_{M[1]}*\underline{\sigma}_M\hfill\\
   &=&
  [u_N*\widehat{\delta}(g_2)
  -
  \widehat{\delta}(u_N)*g_2
  -
  u_N*g_2*u_{M[1]}]*\underline{\sigma}_M\hfill\\
  \end{matrix}$$
and 
$$\begin{matrix}\widehat{\delta}(g_1)
&=&
\widehat{\delta}(u_N)*g_2*\underline{\sigma}_M
-u_N*\widehat{\delta}(g_2)*\underline{\sigma}_M
-\widehat{\delta}(g_2)*u_{M[1]}*\underline{\sigma}_M
-g_2*\widehat{\delta}(u_{M[1]})*\underline{\sigma}_M\hfill\\
&=& 
[\widehat{\delta}(u_N)*g_2
-u_N*\widehat{\delta}(g_2)
-\widehat{\delta}(g_2)*u_{M[1]}
-g_2*\widehat{\delta}(u_{M[1]})]*\underline{\sigma}_M.\hfill\\
  \end{matrix}$$
Then, we have that equation $(a_2)$ holds and $\eta_2(g)=g_2$. So $\eta_2$ is an isomorphism, which is clearly a natural transformation. 
\end{proof}

\begin{corollary}\label{P: J(M,u) es E-proy y es E-iny}
 For each $(M,u)\in \TGM\g {\cal B}$, the twisted graded ${\cal B}$-module 
 $J(M,u)$ is ${\cal E}$-projective and ${\cal E}$-injective in the exact category 
 $(\TGM^0\g {\cal B},{\cal E})$. 
\end{corollary}
 
 \begin{proof} Let us write $\underline{M}=(M,u_M)$ for the objects of $\TGM\g{\cal B}$. 
 For each exact pair
 $\underline{M}\rightmap{f}\underline{E}\rightmap{g}\underline{N}$ in ${\cal E}$, we have a commutative diagram 
 $$\begin{matrix}
  \Hom^0_{\TGM\g {\cal B}}(J(\underline{L}),\underline{M})&\rightmap{f_*}  &\Hom^0_{\TGM\g {\cal B}}(J(\underline{L}),\underline{E})&\rightmap{g_*} &\Hom^0_{\TGM\g {\cal B}}(J(\underline{L}),\underline{N})\hfill\\
  \shortlmapdown{\eta_2}&&\shortlmapdown{\eta_2}&&\shortlmapdown{\eta_2}\\
  \Hom^0_{\GM\g {\cal B}}(L[1],M)&\rightmap{f_*}  &\Hom^0_{\GM\g {\cal B}}(L[1],E)&\rightmap{g_*} &\Hom^0_{\GM\g {\cal B}}(L[1],N),\hfill\\
  \end{matrix}$$
where the second row is exact with $g_*$ surjective and $f_*$ injective. 
Since $\eta_2$ is a natural isomorphism, the first row
is exact with $g_*$ surjective and $f_*$ injective. 
So $J(\underline{L})$ is ${\cal E}$-projective.  

The argument showing that $J(\underline{L})$ is ${\cal E}$-injective is similar,
now using $\eta_1$.   
 \end{proof}

For the rest of this section, we consider the restriction functors  
$$T:\TGM^0\g{\cal B}\rightmap{} \TGM^0\g{\cal B} \hbox{ \  and \ } J:\TGM^0\g{\cal B}\rightmap{} \TGM^0\g{\cal B}$$ denoted with the same symbols used for the functors $T$ and $J$ considered before.
 
\begin{proposition}\label{P: pares exactos con termino medio J(M)}
There are natural transformations $\alpha: id_{\TGM^0\g{\cal B} }\rightmap{}J$ and $\beta:J\rightmap{}T$ such that for each $\underline{M}\in \TGM\g{\cal B}$ we have an ${\cal E}$-conflation 
$$\begin{matrix}\underline{M}&
\rightmap{ \ \alpha_{\underline{M}} \ }&J(\underline{M})&
\rightmap{ \ \beta_{\underline{N}} \ }&
T(\underline{M}).\end{matrix}$$
For $\underline{M}\in  \GModi\g{\cal B}$, these morphisms are given by
$$\alpha_{\underline{M}}=(\hueca{I}_{M},0)^t: \underline{M}\rightmap{}J(\underline{M}) \hbox{  \ and  \ }
 \beta_{\underline{M}}=(0,\hueca{I}_{M[1]}):J(\underline{M})\rightmap{}\underline{N}.$$
\end{proposition}

\begin{proof} Consider the $0$-degree homogeneous morphisms of twisted ${\cal B}$-modules $$\alpha_{\underline{M}}:=\eta^{-1}_1(\hueca{I}_M){ \   \ and \  \  }\beta_{\underline{M}}:=\eta_2^{-1}(\hueca{I}_{M[1]}),$$
which are described explicitely by the formulas for the inverses of $\eta_1$ and $\eta_2$ given in  the proof of (\ref{P: isos naturales de morfismos en TGMod-B en GMod-B}).
\end{proof}

From (\ref{P: J(M,u) es E-proy y es E-iny}) and (\ref{P: pares exactos con termino medio J(M)}), we immediately obtain the following. 

\begin{corollary}\label{C: proy e iny en TGMod-B son sumandos dir de J(M)'s}
The ${\cal E}$-projective and the ${\cal E}$-injective objects in the exact category $(\TGModi^0\g{\cal B},{\cal E})$ are the direct summands of the objects $J(\underline{M})$, for $\underline{M}\in \TGModi\g{\cal B}$. 
\end{corollary}

 \begin{proposition}\label{P: homotopically trivial=factors through E-proj}
 Let $f:\underline{M}\rightmap{}\underline{N}$ be a morphism in $\TGModi^0\g{\cal B}$. Then, $f$ is homotopically trivial iff it factors through an ${\cal E}$-projective twisted ${\cal B}$-module. 
 \end{proposition}

 \begin{proof} Since the ${\cal E}$-projective twisted ${\cal B}$-modules coincide with the  ${\cal E}$-injective twisted ${\cal B}$-modules, a morphism $f:\underline{M}\rightmap{}\underline{N}$ factors through an ${\cal E}$-projective iff it factors through $\alpha_{\underline{M}}$. So we have to show that $f$ factors through $\alpha_{\underline{M}}$ iff it is homotopically trivial. 
 
 Suppose first that $f$ factors through $\alpha_{\underline{M}}$ Then there is 
 $h=(h_1,h_2)\in \Hom^0_{\TGModi\g{\cal B}}(J(\underline{M}),\underline{N})$ such that $f=h*\alpha_{\underline{M}}$. 
 Then, $f=h_1$ and the first component of $h$ is given by 
 $$h_1=\widehat{\delta}(h_2)*\underline{\sigma}_M+u_N*h_2*\underline{\sigma}_M
 -h_2*u_{M[1]}*\underline{\sigma}_M.$$
 Hence,
 $$f=\widehat{\delta}(h_2*\underline{\sigma}_M)+u_N*(h_2*\underline{\sigma}_M)+(h_2*\underline{\sigma}_M)*u_M,$$
 with $h_2*\underline{\sigma}_M\in \Hom^{-1}_{\GModi\g{\cal B}}(M,N)$. So, $f$ is homotopically trivial. 
 
 Assume now that $f:\underline{M}\rightmap{}\underline{N}$ is homotopically trivial and take a morphism 
 $g\in \Hom^{-1}_{\GModi\g{\cal B}}(M,N)$ such that $f=\widehat{\delta}(g)+u_N*g+g*u_M$. Now, consider the morphism $h_2:=g*\underline{\sigma}_M^{-1}\in \Hom^0_{\GModi\g{\cal B}}(M[1],N)$. Then, we have 
 $$\begin{matrix}
   f&=& \widehat{\delta}(h_2*\underline{\sigma}_M)+u_N*(h_2*\underline{\sigma}_M)+(h_2*\underline{\sigma}_M)*u_M\hfill\\
   &=&
   \widehat{\delta}(h_2)*\underline{\sigma}_M+u_N*h_2*\underline{\sigma}_M
 -h_2*u_{M[1]}*\underline{\sigma}_M.\hfill\\
   \end{matrix}$$
   Then, $h=(f,h_2)\in \Hom^0_{\TGModi\g{\cal B}}(J(\underline{M}),\underline{N})$ satisfies $f=h*\alpha_{\underline{M}}$. 
 \end{proof}

\begin{definition}\label{D: estructura triangular de TGMod-B}
  Let ${\cal B}$ be a triangular differential graded $S$-bocs. 
  Then, a sequence of morphisms 
  $$(M,u_M)\rightmap{\underline{f}}(E,u_E)\rightmap{\underline{g}}(N,u_N)\rightmap{\underline{h}}T(M,u_M)$$ 
  in the homotopic category $\underline{\TGM}^0\g {\cal B}$ 
 such that 
 $$\xi: (M,u_M)\rightmap{f}(E,u_E)\rightmap{g}(N,u_N)$$
 is an exact pair in ${\cal E}$ and we have a commutative diagram of the form 
 $$\begin{matrix}
   \xi&:& (M,u_M)&\rightmap{f}&(E,u_E)&\rightmap{g}&(N,u_N)\\
    &&\parallel&&\lmapdown{}&&\lmapdown{h}\\
     \xi_M&:&(M,u_M)&
     \rightmap{\alpha_{\underline{M}}}&
     J(M,u_M)&
     \rightmap{\beta_{\underline{M}}}&T(M,u_M)\\
   \end{matrix}$$
   in $\TGM^0\g {\cal B}$ is called a 
   \emph{canonical triangle} in  $\underline{\TGM}^0\g {\cal B}$.
   Notice that, in this case, we have 
   $\xi_Mh=\xi$. Now, consider the class ${\cal T}$ of  sequences  of morphisms 
   $$X\rightmap{u}Y\rightmap{v}Z\rightmap{w}TX$$ 
   in $\underline{\TGM}^0\g {\cal B}$ which are isomorphic to   canonical triangles. That is, 
   there is a canonical triangle 
  $M\rightmap{\underline{f}}E\rightmap{\underline{g}}N\rightmap{\underline{h}}TM$
   and isomorphisms $a,b,c$ such that the following diagram commutes 
    $$\begin{matrix}X&\rightmap{u}&Y&\rightmap{v}&Z&\rightmap{w}&TX\\
    \lmapdown{a}&&\lmapdown{b}&&\lmapdown{c}&&\lmapdown{T(a)}\\
    M&\rightmap{\underline{f}}&E&\rightmap{\underline{g}}&N&
    \rightmap{\underline{h}}&TM.\\
    \end{matrix}
    $$
    The elements of ${\cal T}$ are called the 
    \emph{triangles of $\underline{\TGM}^0\g {\cal B}$}.
 \end{definition}
 
 The following result follows from the general theory for special Frobenius categories, see \cite{H} and \cite{BS}. 

 \begin{theorem}\label{T: TGMod-A estable es triangulada} 
 Let ${\cal B}$ be a triangular differential graded $S$-bocs. Then the category $\TGM^0\g {\cal B}$ is a special Frobenius category with the automorphism $T$ of (\ref{D: traslacion para GM-B y TGM-B}) and the endofunctor $J$ described in (\ref{P: construction of J for TGMod-B}), both restricted to $\TGM^0\g {\cal B}$.  
 
 The stable category $\underline{\TGM}^0\g {\cal B}$ with the automorphism $T$ and the class of triangles ${\cal T}$ 
 defined above is a triangulated category. 
 \end{theorem}

\section{Quasi-isomorphisms of twisted modules}

Before, the proof of our theorem (\ref{P: (M,u0) aciclico sii (M,u) homotopicamente trivial}), we fix some useful notation.

\begin{remark}\label{R: transferencia de dif en GMod-B a GMod'-B}
Given a graded differential $S$-bocs ${\cal B}=(C,\mu,\epsilon,\delta)$, 
 the equivalence of categories $F:\GM\g{\cal B}\rightmap{}\GModi\g {\cal B}$ described in 
 (\ref{L: equiv directa entre GMod-B y G hueca(M)od-B}) permits to transfer the differential $\widehat{\delta}$ of the category 
 $\GM\g{\cal B}$ to a differential $\widehat{\delta}$ on the graded category $\GModit\g {\cal B}$. 
 
 Thus, given a homogeneous morphism $f=(f^0,f^1):M\rightmap{}N$ in $\GModit\g {\cal B}$, we have 
 $\widehat{\delta}(f)=(0,\widehat{\delta}(f)^1)$, where for any homogeneous elements $m\in M$ and $c\in \overline{C}$, we have $$\widehat{\delta}(f)^1(c)[m]=(-1)^{\vert f\vert +\vert m\vert+1}f^1(\delta(c))[m].$$
 If $g:M\rightmap{}N$ is a morphism in $\GM\g{\cal B}$, then $\widehat{\delta}(F(g))=F(\widehat{\delta}(g))$. 
 So $F$ determines an equivalence of the differential graded categories $\GM\g{\cal B}$ and $\GModit\g{\cal B}$. 
 
 Then, the category of twisted ${\cal B}$-modules $\TGModit\g{\cal B}$ is formed by the pairs $\underline{M}=(M,u_M)$, with $u_M\in \Hom^1_{\GModitsub\g{\cal B}}(M,N)$ such that $\widehat{\delta}(u_M)+u_M*u_M=0$. 
 
 A homogeneous morphism $f:\underline{M}\rightmap{}\underline{N}$ of degree $d$ in $\TGModit\g{\cal B}$
 is a homogeneous 
 morphism 
 $f:M\rightmap{}N$ of degree $d$ in $\GModit\g{\cal B}$ such that the equality  $\widehat{\delta}(f)+u_N*f-(-1)^df*u_M=0$ holds. 
 
 A homogeneous morphism $f:\underline{M}\rightmap{}\underline{N}$ in $\TGModit^0\g{\cal B}$ is called \emph{homotopically trivial} iff there is a homogeneous morphism $h:M\rightmap{}N$ in $\GModit\g{\cal B}$ of degree $-1$ such that $f=\widehat{\delta}(h)+u_N*h+h*u_M$.
 \medskip
 
 For the sake of notational simplicity, from now on, given $f=(f^0,f^1):M\rightmap{}N$ and $g=(g^0,g^1):N\rightmap{}L$ in $\GModit\g{\cal B}$, we write
 $$g^0f^1:=[\overline{C}\rightmap{f^1}\Hom_{\GM\g k}(M,N)\rightmap{g^0_*}\Hom_{\GM\g k}(M,L)],$$
 $$g^1f^0:=[\overline{C}\rightmap{g^1}\Hom_{\GM\g k}(N,L)\rightmap{f^0_*}\Hom_{\GM\g k}(M,L)],$$
 and
 $g^1\cdot f^1:= (g*f)^1-g^0f^1-g^1f^0:\overline{C}\rightmap{}\Hom_{\GM\g k}(M,L)$. Thus, the maps $g^0f^1$, $g^1f^0$ and $g^1\cdot f^1$ are morphisms of graded $S$-$S$-bimodules. Then, as mentioned in 
 (\ref{R: canonical embedding GM-S --> GM-B}), there is a formula to compute $g^1\cdot f^1$. Namely, given $c\in \overline{C}$ with $\overline{\mu}(c)=\sum_sc^1_s\otimes c_s^2$, with $c^1_s,c^2_s\in \overline{C}$, we have 
 $$(g^1\cdot f^1)(c)=\sum_sg^1(c_s^2)f^1(c_s^1).$$
 Therefore, $(g*f)^0=g^0f^0$ and $(g*f)^1= g^0f^1+g^1f^0+g^1\cdot f^1$.
 \medskip
 
 Then, the equality $\widehat{\delta}(u_M)+u_M*u_M=0$ is equivalent to 
 $$(A):\hbox{\hskip1.5cm}\begin{cases}
      0=(u_M^0) ^2\\
      0=\widehat{\delta}(u_M)^1+u^0_Mu^1_M+u^1_Mu^0_M+u^1_M\cdot u^1_M.\\
      \end{cases}$$
      
 A homogeneous morphism $f:\underline{M}\rightmap{}\underline{N}$ of degree $0$ in $\TGModit\g{\cal B}$ is a morphism 
 $f=(f^0,f^1):M\rightmap{}N$ of degree $0$ in $\GModit\g{\cal B}$ such that 
 $$(D):\hbox{\hskip1.5cm}\begin{cases}
      0=u_N^0f^0-f^0u_M^0\\
      0=\widehat{\delta}(f)^1+u^0_Nf^1+ u_N^1f^0+u_N^1\cdot f^1- f^0u^1_M-f^1u^0_M-f^1\cdot u^1_M.\\
      \end{cases}$$
  In particular, $f^0:(M,u_M^0)\rightmap{}(N,u^0_N)$ is a morphism of complexes of right $S$-modules.     
\end{remark}

For the proof of the following result, it is convenient to extend the preceding notations as follows.

\begin{remark}\label{R: notacion producto de h1 con u1}
Let ${\cal B}=(C,\mu,\epsilon,\delta)$ be a triangular differential graded $S$-bocs and consider its triangular filtration
$0=\overline{C}_0\subseteq \overline{C}_1\subseteq \cdots\subseteq \overline{C}_i\subseteq \overline{C}_{i+1}\subseteq \cdots$

 For any $M,N, L\in \GM\g S$, $f_i^1\in \Hom_{\GM \g S}(M\otimes_S\overline{C}_i,N)$, and 
$g^1_i\in \Hom_{\GM \g S}(N\otimes_S\overline{C}_i,L)$ we can consider the morphism of graded right $S$-modules
$g_i^1*f_i^1\in \Hom_{\GM \g S}(M\otimes_S\overline{C}_{i+1},L)$
 defined as the composition 
$$M\otimes_S\overline{C}_{i+1}\rightmap{ \ id_M\otimes\overline{\mu} \ }
M\otimes_S \overline{C}_i\otimes_S \overline{C}_i\rightmap{ \  f_i^1 \otimes id_{\overline{C}_i}\ }N\otimes_S \overline{C}_i\rightmap{g^1_i}L.$$
For each $i\geq 0$, we have the canonical isomorphism
$$\eta_i:\Hom_{\GM \g S}(M\otimes_S\overline{C}_i,L)\rightmap{}\Hom_{\GM\g S\g S}(\overline{C}_i,\Hom_{\GM\g k}(M,L)).$$
We can define the product morphism:  $\eta_i(g^1_i)\cdot \eta_i(f^1_i):=\eta_{i+1}(g^1_i*f^1_i)$.

Then, given any two morphisms $f_i^1\in \Hom_{\GM\g S\g S}(\overline{C}_i,\Hom_{\GM\g k}(M,N))$ and $g_i^1\in \Hom_{\GM\g S\g S}(\overline{C}_i,\Hom_{\GM\g k}(N,L))$, we have their product  morphism $g_i^1\cdot f_i^1\in
\Hom_{\GM\g S\g S}(\overline{C}_{i+1},\Hom_{\GM\g k}(M,L))$ which is computed on each element  
$c\in \overline{C}_{i+1}$ as  
$$(g_i^1\cdot f_i^1)(c)=\sum_jg^1_i(c^2_j)f^1_i(c^1_j),
\hbox{ where } \overline{\mu}(c)=\sum_jc_j^1\otimes c_j^2 \hbox{ with } c^1_j,c_j^2\in \overline{C}_i.$$
Given $g^0\in \Hom_{\GM\g S}(N,L)$ and $f_i^1\in \Hom_{\GM\g S\g S}(\overline{C}_i,\Hom_{\GM\g k}(M,N))$, 
the morphism  
$g^0f_i^1\in \Hom_{\GM\g S\g S}(\overline{C}_i,\Hom_k(M,L))$ is defined by $(g^0f^1_i)(c)=g^0f^1_i(c)$. 
Similarly, given morphisms $g_i^1\in \Hom_{\GM\g S\g S}(\overline{C}_i,\Hom_{\GM\g k}(N,L))$ and 
$f^0\in \Hom_{\GM\g S}(M,N)$, 
$g_i^1f^0\in \Hom_{\GM\g S\g S}(\overline{C}_i,\Hom_{\GM\g k}(M,L))$ is defined by $(g^1_if^0)(c)=g_i^1(c)f^0$.

The following formulas hold:
$$ h^1_i\cdot g^0f_i^1=h_i^1g^0\cdot f_i^1, \hbox{ \ \ \ }
   h^0(f^1_i\cdot g^1_i)=h^0f^1_i\cdot g^1_i,\hbox{ and  }
   (h^1_i\cdot g^1_i)f^0= h^1_i\cdot g^1_if^0.$$   
If the morphism $h^1_{i+1}$ restricts to $h^1_i$ and the morphism $f^1_{i+1}$ restricts to $f^1_i$, from the coassociativity of the comultiplication $\overline{\mu}$, we have the \emph{associativity formula} 
 $$\begin{matrix}h^1_{i+1}\cdot(g^1_i\cdot f^1_i)&=&(h^1_i\cdot g^1_i)\cdot f^1_{i+1}.\hfill\\
  \end{matrix}$$

Since the bocs ${\cal B}$ is triangular, we have $\delta(\overline{C}_{i+1})\subseteq Z(\overline{C}_{i+1})$. Then, for a homogeneous element  
$\widehat{f}^1_{i}\in \Hom_{\GM\g S\g S}(\overline{C}_i+Z(\overline{C}_{i+1}),\Hom_{\GM\g k}(M,N))$, 
we can define the morphism 
$\widehat{d}(\widehat{f}^1_i)\in 
\Hom_{\GM\g S\g S}(\overline{C}_{i+1},\Hom_{\GM\g k}(M,N))$ by  
$$\widehat{d}(\widehat{f}^1_{i})(c)[m]=(-1)^{\vert m\vert+\vert \widehat{f}_{i}^1\vert +1}\widehat{f}^1_i(\delta(c))[m],$$
for any $c\in \overline{C}_{i+1}$  and  any homogeneous element $m\in M$. 

Given a homogeneous  $\widehat{g}^1_{i}\in \Hom_{\GM\g S\g S}(\overline{C}_i+Z(\overline{C}_{i+1}),\Hom_{\GM\g k}(N,L))$,  the following \emph{Leibniz formula} for morphisms $\overline{C}_{i+1}\rightmap{} \Hom_{\GM\g k}(M,L)$ holds
$$\widehat{d}(g^1_i\cdot f^1_i)=\widehat{d}(\widehat{g}^1_i)\cdot f^1_i
+
(-1)^{\vert g_i^1\vert}g_i^1\cdot \widehat{d}(\widehat{f}^1_i),$$
where $f^1_i$ and $g^1_i$ denote the restrictions of $\widehat{f}^1_i$ and $\widehat{g}^1_i$ 
from their common domain 
$\overline{C}_i+Z(\overline{C}_{i+1})$ to  $\overline{C}_i$; thus the morphism $g_i^1\cdot f_i^1$ is defined on 
$\overline{C}_{i+1}$, hence it is defined on $\overline{C}_i+Z(\overline{C}_{i+1})$; the morphisms $\widehat{d}(\widehat{g}^1_i)$ and $\widehat{d}(\widehat{f}^1_i)$ are defined on $\overline{C}_{i+1}$, hence on $\overline{C}_i$.
\end{remark}

 \begin{theorem}\label{P: (M,u0) aciclico sii (M,u) homotopicamente trivial} Assume that $S$ is a finite product of copies of the field $k$. 
 Let $(M,u)\in \TGModi\g{\cal B}$. Then, the twisted ${\cal B}$-module $(M,u)$ is homotopically trivial in $\TGModit^0\g{\cal B}$ iff the complex of right $S$-modules $(M,u^0)$ is \emph{acyclic}, that is $H^i(M,u^0)=0$, for all $i\in \hueca{Z}$.
 \end{theorem}

 \begin{proof} If $(M,u)$ is homotopically trivial, then $\hueca{I}_M=(id_M,0)$ is homotopic to the zero map in $\TGModit^0\g{\cal B}$. Then, there is a morphism $h\in \Hom^{-1}_{\GModitsub\g{\cal B}}(M,M)$ such that $\hueca{I}_M=\widehat{\delta}(h)+u*h+h*u$. Looking at the first components, we obtain 
 $id_M=u^0h^0+h^0u^0$. So, $(M,u^0)$ is a homotopically trivial complex of right $S$-modules, which implies that $H^i(M,u^0)=0$, for all $i\in \hueca{Z}$. 
 
 Now, assume that $(M,u^0)$ is an acyclic complex of right $S$-modules. Then, for instance from \cite{Brown}(0.3), we know that $(M,u^0)$ is homotopically trivial. So, there is a homogeneous morphism $h^0:M\rightmap{}M$ of right $S$-modules of degree $-1$ such that the following equality holds
  $$(\Delta^0):\hbox{\hskip1.cm}id_M=u^0h^0+h^0u^0.$$
  We want to show that $(M,u)$ is homotopically trivial, so we are looking for a homogeneous morphism of  $S$-$S$-bimodules $h^1:\overline{C}\rightmap{}\Hom_{\GM\g k}(M,M)$ of degree $-1$ such that $h=(h^0,h^1)$ satisfies $\widehat{\delta}(h)+u*h+h*u=\hueca{I}_M$. That is such that the following equality holds
 $$\Delta^1:=\widehat{\delta}(h)^1+u^0h^1+u^1h^0+u^1\cdot h^1+ 
 h^0u^1 +h^1u^0+h^1\cdot u^1=0.$$
 
 Write $E_M:=\Hom_{\GM\g k}(M,M)$ and make $h^1_0=0\in \Hom_{\GM\g S\g S}^{-1}(\overline{C}_0,E_M)$. 
 We will construct, for each $i\geq 0$, a morphism 
 $h^1_{i+1}\in \Hom_{\GM\g S\g S}^{-1}(\overline{C}_{i+1},E_M)$ such that 
 $$\Delta^1_i:=\widehat{d}(h^1_{i+1})+u^0h_{i+1}^1+u_{i+1}^1h^0+u_i^1\cdot h_i^1+ 
 h^0u_{i+1}^1 +h_{i+1}^1u^0+h_i^1\cdot u_i^1=0$$
 and $h^1_{i+1}$ restricts to $h^1_i$, for all $i$. Here, $u^1_i$ denotes the restriction of $u^1$ to $\overline{C}_i$. Once we evaluate $u_i^1$, we can skip the subindex:  $u^1_i(c)=u^1(c)$, for $c\in \overline{C}_i$. 
 
 Once we have done this, we can define $h^1\in \Hom_{\GM\g S\g S}^{-1}(\overline{C},E_M)$, by $h^1(c)=h^1_i(c)$, whenever $c \in \overline{C}_i$. So $\Delta^1(c)=\Delta^1_i(c)=0$, for $c\in \overline{C}_{i+1}$. Hence $\Delta^1=0$ and we are done.

We require for this construction  a special vector space basis $\hueca{B}_{i+1}$ of $\overline{C}_{i+1}$ consisting of homogeneous elements. 
 In order to describe this basis, we will consider, for each $i\geq 0$, decompositions of graded $S$-$S$-bimodules of the form 
 $$\overline{C}_{i+1}=\overline{C}_{i}\oplus V_{i+1}\oplus W_{i+1},$$
 where
 $\overline{C}_{i}+Z(\overline{C}_{i+1})=\overline{C}_{i}\oplus V_{i+1}$ and 
 $V_{i+1}\subseteq Z(\overline{C}_{i+1})$. 
 By assumption, we have a decomposition $1=\sum_{s=1}^ne_s$ of the unit element of the algebra $S$ as a sum of central primitive orthogonal idempotents.  
 The special basis we are interested in has the form  $\hueca{B}_{i+1}=\bigcup_{s,t\in [1,n]}\hueca{B}_{i+1}(s,t)$, where each subset $\hueca{B}_{i+1}(s,t)$ is the basis of $e_t\overline{C}_{i+1}e_s$ defined recursively as follows. At the base $i=0$, we have  $\hueca{B}_1(s,t):=\hueca{B}^v_1(s,t)\cup\hueca{B}^w_1(s,t)$, where 
 $\hueca{B}^v_1(s,t)$ and $\hueca{B}^w_1(s,t)$ are basis formed by homogeneous elements of $e_tV_1e_s$ and $e_tW_1e_s$, respectively. Once we have defined 
 a basis $\hueca{B}_i$ of $\overline{C}_i$, for $i \geq 1$,  we define 
 $$\hueca{B}_{i+1}(s,t):=\hueca{B}_i(s,t)\cup\hueca{B}^v_{i+1}(s,t)\cup\hueca{B}^w_{i+1}(s,t),$$
 where $\hueca{B}^v_{i+1}(s,t)$ and $\hueca{B}^w_{i+1}(s,t)$ are basis consisting of homogeneous elements 
 of $e_tV_{i+1}e_s$ and $e_tW_{i+1}e_s$, respectively  
 
 \medskip
 \emph{Step 1:  The construction of $h_1^1:\overline{C}_1\rightmap{}E_M$.}
 \medskip
 
 In order to define the homogeneous morphism  of $S$-$S$-bimodules $h_1^1$ we want,  it will be enough to give, for each homogeneous basic element $c\in \hueca{B}_1(s,t)$, a homogeneous element $h_1^1(c)\in 
 \Hom_{\GM\g k}(Me_t,Me_s)$ of degree $\vert c\vert-1$ satisfying the equation $\Delta^1_0(c)=0$.

 Start with a basic element $c\in \hueca{B}_1^v(s,t)$, so $c\in Z(\overline{C}_1)$. 
 Consider the  homogeneous morphism $f_0(c):Me_t\rightmap{}Me_s$ of degree $\vert c\vert$ given by $f_0(c):=u_1^1(c)h^0+h^0u_1^1(c)$. 
 Since $id_M= u^0h^0 + h^0u^0 $, we have 
 $$\begin{matrix}
 u^0f_0(c)-f_0(c)u^0
 &=&
 u^0u^1(c)h^0+u^0h^0u^1(c)-u^1(c)h^0u^0-h^0u^1(c)u^0\hfill\\
   &=&
 u^0u^1(c)h^0 + u^1(c) - h^0u^0u^1(c)\hfill\\
 &&\, 
 -u^1(c) +u^1(c)u^0h^0-h^0u^1(c)u^0\hfill\\
   &=&  
  [u^0u_1^1(c)+u_1^1(c)u^0]h^0-h^0[u^0u_1^1(c)+u_1^1(c)u^0].\hfill\\
   \end{matrix}$$
Since $\overline{\mu}(c)=0$ and $\overline{\delta}(c)=0$, from 
(\ref{R: transferencia de dif en GMod-B a GMod'-B})(A), we get 
$u^0u^1(c)+u^1(c)u^0=0$. Therefore, 
$$u^0f_0(c)=f_0(c)u^0.$$

 In the following, given $N\in \GM\g k$, we denote by $\sigma^{[i]}:N\rightmap{}N[i]$ the homogeneous morphism of degree $-i$ which acts as the identity on the underlying non-graded spaces. 
 
 Consider the morphism $\tau:=\sigma^{[\vert c\vert]}:Me_s\rightmap{}Me_s[\vert c\vert]$. Then, 
the homogeneous  morphism 
$\tau f_0(c):Me_t\rightmap{}Me_s[\vert c\vert]$ of degree $0$ satisfies  
$$\tau f_0(c)u^0-\tau u^0\tau^{-1}\tau f_0(c)=\tau(f_0(c)u^0-u^0f_0(c))=0.$$
Thus $\tau f_0(c)$ is a morphism of complexes $(Me_t,u^0_{\vert Me_t})\rightmap{}(Me_s[\vert c\vert],\tau(u^0_{\vert Me_s})\tau^{-1})$. 

   The complex $(Me_t,u^0_{\vert Me_t})$ is acyclic, so it is homotopically trivial. Then, so is the morphism $-\tau f_0(c)$.  Hence,  there is a homogeneous morphism 
   $\underline{h}_0^1(c)\in \Hom^{-1}_{\GM\g k}(Me_t,Me_s[\vert c\vert])$ such that 
   $$-\tau f_0(c)=\tau u^0\tau^{-1}\underline{h}_0^1(c)+
   \underline{h}^1_0(c)u^0.$$
   Then, we have the homogeneous morphism $\widehat{h}^1_0(c):=\tau^{-1}\underline{h}_0^1(c):Me_t\rightmap{}Me_s$, with  degree $\vert c\vert-1$, such that 
   $-f_0(c)=u^0\widehat{h}_0^1(c)+\widehat{h}^1_0(c)u^0$. Hence, 
   $$\widehat{\Delta}_0^1(c):=u_1^1(c)h^0+h^0u_1^1(c)+u^0\widehat{h}^1_0(c)+\widehat{h}^1_0(c)u^0=0.$$
 
 We have defined $\widehat{h}^1_0(c)$ for any $c\in \hueca{B}_1^v$ such that the preceding equality holds. Then, we can consider the homogeneous  morphism  
 $$\widehat{h}^1_0:Z(\overline{C}_1)=V_1\rightmap{}\Hom_{\GM\g k}(M,M)$$
 of degree $-1$ determined by the given values $\widehat{h}^1_0(c)$ on the basic elements $c\in \hueca{B}^v_1$. Thus $\widehat{h}^1_0\in \Hom_{\GM\g S\g S}^{-1}(Z(\overline{C}_1),E_M)$ extends $h^1_0:\overline{C}_0\rightmap{}E_M$ and satisfies $\widehat{\Delta}^1_0(c)=0$, for all $c\in Z(\overline{C}_1)$. 
 \medskip
 
   Now, take an element $c\in \hueca{B}_1^w(s,t)$, so $\overline{\mu}(c)=0$ and $\delta(c)\in Z(\overline{C}_1)$. 
   Consider the homogeneous morphism $g_0(c):Me_t\rightmap{}Me_s$ of degree $\vert c\vert$ given by 
   $$g_0(c):= \widehat{d}(\widehat{h}_0^1)(c)+u_1^1(c)h^0+h^0u_1^1(c).$$
Write $\lambda(c):= u_1^1(c)h^0+h^0u_1^1(c)$ and $\gamma(c):=\widehat{d}(\widehat{h}_0^1)(c)$.  
Since $id_M=u^0h^0+h^0u^0$, we have 
 $$\begin{matrix}
 u^0\lambda(c)-\lambda(c)u^0
 &=&
u^0u^1(c)h^0+u^0h^0u^1(c)
 -u^1(c)h^0u^0-h^0u^1(c)u^0\hfill\\
   &=&
  u^0u^1(c)h^0+u^1(c)-h^0u^0u^1(c)\hfill\\
 && -\,
 u^1(c)+ u^1(c)u^0h^0 -h^0u^1(c)u^0\hfill\\
   &=&  
   [u^0u_1^1(c)+u_1^1(c)u^0]h^0-h^0[u^0u_1^1(c)+u_1^1(c)u^0].\hfill\\
   \end{matrix}$$
  
  Since $\delta(c)\in Z(\overline{C}_1)$, we have $\widehat{\Delta}_0^1(\delta(c))=0$. 
   Then, for any homogeneous element  $m\in Me_t$,  we have 
   $$\begin{matrix}
    [u^0\gamma(c)-\gamma(c)u^0](m)
    &=&
    [u^0\widehat{d}(\widehat{h}_0^1)(c)-\widehat{d}(\widehat{h}_0^1)(c)u^0](m)\hfill\\
    &=&
    (u^0\widehat{d}(\widehat{h}_0^1)(c))[m]-(\widehat{d}(\widehat{h}_0^1)(c))(u^0[m])\hfill\\
    &=&
    (-1)^{\vert m\vert}u^0\widehat{h}_0^1(\delta(c))[m]-
    (-1)^{\vert m\vert+1}\widehat{h}_0^1(\delta(c))(u^0[m])\hfill\\
    &=&
     (-1)^{\vert m\vert}[u^0\widehat{h}_0^1(\delta(c))
    +\widehat{h}_0^1(\delta(c))u^0](m)\hfill\\
    &=&
    (-1)^{\vert m\vert +1}[u^1(\delta(c))h^0+h^0u^1(\delta(c))](m)\hfill\\
    &=&
    [\widehat{d}(u_1^1)(c)h^0-h^0\widehat{d}(u_1^1)(c)](m).\hfill\\
     \end{matrix}$$
Thus, $u^0\gamma(c)-\gamma(c)u^0= [\widehat{d}(u_1^1)h^0-h^0\widehat{d}(u_1^1)](c)$. 
Then, from (\ref{R: transferencia de dif en GMod-B a GMod'-B})(A), we have 
$$\begin{matrix}u^0g_0(c)-g_0(c)u^0
&=&
u^0\lambda(c)-\lambda(c)u^0+u^0\gamma(c)-\gamma(c)u^0\hfill\\
&=&
[\widehat{d}(u_1^1)(c)+u^0u_1^1(c)+u_1^1(c)u^0]h^0\hfill\\
&&-\, 
h_0[\widehat{d}(u_1^1)(c)+u^0u_1^1(c)+u_1^1(c)u^0]=0.\hfill\\
   \end{matrix}$$
Thus $\tau g_0(c)$ is a morphism of complexes $(Me_t,u^0_{\vert Me_t})\rightmap{}(Me_s[\vert c\vert]),\tau(u^0_{\vert Me_s})\tau^{-1}).$ 
Proceeding as before, we get a homogeneous morphism $\widetilde{h}_0^1(c):Me_t\rightmap{}Me_s$ with degree $\vert c\vert-1$ such that 
$$g_0(c)=-u^0\widetilde{h}_0^1(c)-\widetilde{h}_0^1(c)u^0.$$
Then, we have the equality
$$\widetilde{\Delta}^1_0(c):=\widehat{d}(\widehat{h}^1_0)(c)+u_1^1(c)h^0+h^0u_1^1(c)+u^0\widetilde{h}^1_0(c)+\widetilde{h}^1_0(c)u^0=0.$$
We have the homogeneous morphism $\widetilde{h}^1_0:W_1\rightmap{}\Hom_{\GM\g k}(M,M)$ of degree $-1$ determined by the given values $\widetilde{h}_0^1(c)$ on the basic elements $c\in \hueca{B}^w_1$. 
Then, the morphisms $\widehat{h}_0^1$ and $\widetilde{h}_0^1$ determine a homogeneous morphism 
$h^1_1:\overline{C}_1=Z(\overline{C}_1)\oplus W_1\rightmap{}E_M$ of degree $-1$ such that the equation $\Delta^1_0(c)=0$ is satisfied for all $c\in \overline{C}_1$, because either $\widehat{\Delta}^1_0(c)=0$ or  $\widetilde{\Delta}_0^1(c)$ hold on basic elements $c\in \hueca{B}_1$. Clearly, $h^1_1$ extends $h^1_0$. 
 
 \medskip
 \emph{Step 2:  The construction of $h_{i+1}^1:\overline{C}_{i+1}\rightmap{}E_M$, from $h_i^1:\overline{C}_i\rightmap{}E_M$.}
 \medskip
 
 Assume we have already defined the homogeneous morphism $h_i^1:\overline{C}_i\rightmap{}E_M$ of degree $-1$ such that $\Delta^1_{i-1}(c)=0$ holds for all $c\in \overline{C}_i$. 
 
 Given a basic element $c\in \hueca{B}^v_{i+1}(s,t)$, we consider the homogeneous morphism 
 $f_i(c):Me_t\rightmap{}Me_s$ of degree $\vert c\vert$ given by 
 $$f_i(c):= u^1_{i+1}(c)h^0+h^0u_{i+1}^1(c)+(u^1_i\cdot h^1_i)(c)+(h^1_i\cdot u_{i}^1)(c).$$
 Write 
 $\lambda(c):=u^1_{i+1}(c)h^0+h^0u_{i+1}^1(c)$ and $\rho(c)=(u^1_i\cdot h^1_i)(c)+(h^1_i\cdot u_{i}^1)(c)$.  Then, from (\ref{R: transferencia de dif en GMod-B a GMod'-B})(A),  we have 
 $$\begin{matrix}
    u^0\lambda(c)-\lambda(c)u^0
    &=&
    u^0u^1(c)h^0+u^0h^0u^1(c)-u^1(c)h^0u^0-h^0u^1(c)u^0\hfill\\
    &=&
    u^0u^1(c)h^0-h^0u^0u^1(c)+u^1(c)+u^1(c)u^0h^0-u^1(c)-h^0u^1(c)u^0\hfill\\
    &=&
     [u^0u^1(c)+u^1(c)u^0]h^0 -h^0[u^0u^1(c)+u^1(c)u^0]\hfill\\
    &=&
    -(\widehat{d}(u_{i+1}^1)h^0)(c)-(u_i^1\cdot u_i^1)(c)h^0+(h^0\widehat{d}(u_{i+1}^1))(c)+h^0(u_i^1\cdot u_i^1)(c)\hfill\\
    &=&
    (h^0\widehat{d}(u_{i+1}^1))(c)+(h^0u^1_i\cdot u_i^1)(c)
    -(\widehat{d}(u_{i+1}^1)h^0)(c)-(u^1_i\cdot u^1_ih^0)(c)\hfill\\
   \end{matrix}$$
  and, since $c\in Z(\overline{C}_{i+1})$, we obtain 
    $$u^0\lambda(c)-\lambda(c)u^0=(h^0u^1_i\cdot u_i^1)(c)-(u^1_i\cdot u^1_ih^0)(c).$$
  Moreover, we have 
   $$\begin{matrix}
     u^0\rho(c)-\rho(c) u^0&=&
     u^0(u^1_i\cdot h^1_i)(c)+u^0(h^1_i\cdot u_{i}^1)(c)
     -(u^1_i\cdot h^1_i)(c)u^0-(h^1_i\cdot u_{i}^1)(c)u^0. 
     \hfill\\
     &=&
    (u^0u^1_i\cdot h^1_i)(c)+(u^0h^1_i\cdot u_{i}^1)(c)
     -(u^1_i\cdot h^1_iu^0)(c)-(h^1_i\cdot u_{i}^1u^0)(c). 
     \hfill\\  
     \end{matrix}$$
Now, we have $\Delta^1_{i-1}=0$. That is the following equality of morphisms from $\overline{C}_i$ to $E_M$ holds
$$\widehat{d}(h^1_i)+u^0h_i^1+u_i^1h^0+u_{i-1}^1\cdot h_{i-1}^1+ 
 h^0u_i^1 +h_i^1u^0+h_{i-1}^1\cdot u_{i-1}^1=0.$$
Mutiplying the equation $\Delta^1_{i-1}=0$ on the right by $u_i^1$, we have the following  equality of morphisms from $\overline{C}_{i+1}$ to $E_M$ 
$$u^0h_i^1\cdot u^1_i
=-\widehat{d}
(h^1_i)\cdot u^1_i-u_i^1h^0\cdot u^1_i-[u_{i-1}^1\cdot h_{i-1}^1]\cdot u^1_i -h^0u^1_i\cdot u^1_i 
 -h_i^1u^0\cdot u^1_i-[h_{i-1}^1\cdot u_{i-1}^1]\cdot u^1_i.$$
Evaluating at our fixed element $c$, we obtain 
$$\begin{matrix}
 (u^0h^1_i\cdot u^1_i)(c)
 &=&
 -(\widehat{d}(h^1_i)\cdot u^1_i)(c)
 - (u^1_ih^0\cdot u_i^1)(c)
  -([u_{i-1}^1\cdot h_{i-1}^1]\cdot u^1_i)(c)
 \hfill\\
 &&-\,
 (h^0u_i^1\cdot u^1_i)(c) 
 -(h_i^1u^0\cdot u^1_i)(c)
 -([h_{i-1}^1\cdot u_{i-1}^1]\cdot u_i^1)(c).\hfill\\
 \end{matrix}
 $$
 Similarly, multiplying the equation $\Delta_{i-1}^1=0$ on the left by $u^1_i$, we have  
 $$-u_i^1\cdot h_i^1u^0 
 =
 u_i^1\cdot \widehat{d}(h^1_i)
 +
 u_i^1\cdot u^0h_i^1
 +u_i^1\cdot u_i^1h^0+
 u_i^1\cdot [u_{i-1}^1\cdot h_{i-1}^1]
 +
 u_i^1\cdot h^0u_i^1
 +u_i^1\cdot [h_{i-1}^1\cdot u_{i-1}^1].$$ 
 
 Evaluating at the element $c$, we get
 $$\begin{matrix}
 -(u^1_i\cdot h_i^1u^0)(c)
 &=&
 (u^1_i\cdot\widehat{d}(h^1_i))(c)
 +
 (u^1_i\cdot u^0h_i^1)(c)+ (u^1_i\cdot u_i^1h^0)(c)\hfill\\
 &&+\,
 (u^1_i\cdot[u_{i-1}^1\cdot h_{i-1}^1])(c)+
 (u^1_i\cdot h^0u_i^1)(c) +(u^1_i\cdot[h_{i-1}^1\cdot u_{i-1}^1])(c).\hfill\\
 \end{matrix}$$

 Then,  since $u^1_ih^0\cdot u^1_i=u^1_i\cdot h^0u^1_i$, we have 
 $$\begin{matrix}
 u^0\rho(c)-\rho(c)u^0
 &=&
 (u^0u^1_i\cdot h^1_i)(c)-(h^1_i\cdot u_{i}^1u^0)(c)\hfill\\
 &&-\,
 (\widehat{d}(h^1_i)\cdot u^1_i)(c)
  -([u_{i-1}^1\cdot h_{i-1}^1]\cdot u^1_i)(c)
 \hfill\\
 &&-\,
 (h^0u_i^1\cdot u^1_i)(c) 
 -(h_i^1u^0\cdot u^1_i)(c)
 -([h_{i-1}^1\cdot u_{i-1}^1]\cdot u_i^1)(c)\hfill\\
 &&+\,
 (u^1_i\cdot\widehat{d}(h^1_i))(c)
 +
 (u^1_i\cdot u^0h_i^1)(c)+ (u^1_i\cdot u_i^1h^0)(c)\hfill\\
 &&+\,
 (u^1_i\cdot[u_{i-1}^1\cdot h_{i-1}^1])(c) +(u^1_i\cdot[h_{i-1}^1\cdot u_{i-1}^1])(c).\hfill\\
   \end{matrix}$$

  Moreover, from the associativity of (\ref{R: notacion producto de h1 con u1}), we have 
 $$\begin{matrix}
 u^0\rho(c)-\rho(c)u^0
 &=&
 -(  \widehat{d}(h^1_i)\cdot u_i^1)(c)-(h^1_i\cdot[u^0u^1_i+u^1_iu^0+u^1_{i-1}\cdot u^1_{i-1}])(c)\hfill\\
 &&+\,
 (u^1_i\cdot \widehat{d}(h^1_i))(c)+([u^0u^1_i+u^1_iu^0+u^1_{i-1}\cdot u^1_{i-1}]\cdot h^1_i)(c)\hfill\\
 &&+\,
 (u^1_i\cdot u^1_ih^0)(c)-(h^0u^1_i\cdot u^1_i)(c).\hfill \\
 \end{matrix}$$
 
 Since $h_i^1,u_i^1$ are defined on $\overline{C}_i$, they are defined on $\overline{C}_{i-1}+Z(\overline{C}_i)$, then the morphisms  $\widehat{d}(h^1_i)$ and $\widehat{d}(u^1_i)$ are defined on $\overline{C}_i$. Moreover, the morphisms $h^1_i$ and $u^1_i$ are defined on $\overline{C}_i$, hence their product $h_i^1\cdot u^1_i$ is defined on $\overline{C}_{i+1}$, hence it is defined on $\overline{C}_i+Z(\overline{C}_{i+1})$. Therefore, by the Leibniz formula of 
  (\ref{R: notacion producto de h1 con u1}), we have 
  $$0=\widehat{d}(h_i^1\cdot u^1_i)(c)=[\widehat{d}(h^1_i)\cdot u^1_i-h^1_i\cdot \widehat{d}(u^1_i)](c),$$
  where the left equality is due to the fact that $c\in \hueca{B}_{i+1}^v(s,t)\subseteq V_{i+1}\subseteq Z(\overline{C}_{i+1})$. Thus, $(\widehat{d}(h^1_i)\cdot u^1_i)(c)=(h^1_i\cdot \widehat{d}(u^1_i))(c)$ and, similarly, we have  $(\widehat{d}(u^1_i)\cdot h_i^1)(c)=(u^1_i\cdot \widehat{d}(h^1_i))(c)$. 
  Then, we have 
 $$\begin{matrix}
 u^0\rho(c)-\rho(c)u^0
 &=&
 -(h^1_i\cdot[\widehat{d}(u^1_i)+u^0u^1_i+u^1_iu^0+u^1_{i-1}\cdot u^1_{i-1}])(c)\hfill\\
 &&+\,
 ([\widehat{d}(u^1_i)+u^0u^1_i+u^1_iu^0+u^1_{i-1}\cdot u^1_{i-1}]\cdot h^1_i)(c)\hfill\\
 &&+\,
 (u^1_i\cdot u^1_ih^0)(c)-(h^0u^1_i\cdot u^1_i)(c).\hfill \\
 \end{matrix}$$
 
  From (\ref{R: transferencia de dif en GMod-B a GMod'-B})(A), we know that $\widehat{d}(u^1_i)+u^0u^1_i+u^1_iu^0+u^1_{i-1}\cdot u^1_{i-1}=0$ on   $\overline{C}_i$. It follows that
 $u^0f_i(c)-f_i(c)u^0=u^0\lambda(c)-\lambda(c) u^0+u^0\rho(c)-\rho(c) u^0=0$. 
 As before, we have a morphism of complexes $-\tau f_i(c):(Me_t,u^0_{\vert Me_t})\rightmap{}(Me_s[\vert c\vert],\tau u^0_{\vert Me_s}\tau^{-1})$ which is homotopically trivial. As a consequence, there is a homogeneous morphism $\underline{h}_i^1(c)\in\Hom_{\GM\g k}^{-1}(Me_t,Me_s[\vert c\vert])$ such that 
 $$-\tau f_i(c)=\tau u^0\tau^{-1}\underline{h}_i^1(c)+\underline{h}_i^1(c)u^0.$$
 The morphism $\widehat{h}^1_i(c):=\tau^{-1}\underline{h}_i^1(c):
 Me_t\rightmap{}Me_s$ is homogeneous of degree $\vert c\vert-1$ such that $-f_i(c)=u^0\widehat{h}_i^1(c)+\widehat{h}_i^1(c)u^0$. Hence, 
 $$\widehat{\Delta}^1_i(c):=u_{i+1}^1(c)h^0+h^0u_{i+1}^1(c)+u^0\widehat{h}_i^1(c)+\widehat{h}_i^1(c)u^0+
(u^1_i\cdot h^1_i)(c)+(h^1_i\cdot u_i^1)(c)=0.$$
 
 We have defined $\widehat{h}^1_i(c)$ for any $c\in \hueca{B}_{i+1}^v$ such that the preceding equality holds. Then, we can consider the homogeneous morphism 
 $$\widehat{h}^1_i:\overline{C}_i\oplus V_{i+1}\rightmap{}\Hom_{\GM\g k}(M,M)$$
 of degree $-1$ determined by the given values $\widehat{h}^1_i(c)$ on the basic elements $c\in \hueca{B}^v_{i+1}$ and by $\widehat{h}^1_i(c):=h^1_i(c)$ on the basic elements  $c\in \hueca{B}_i$. Therefore, the morphism $\widehat{h}^1_i\in \Hom_{\GM\g S\g S}^{-1}(\overline{C}_i+Z(\overline{C}_{i+1}),E_M)$ extends $h^1_i:\overline{C}_i\rightmap{}E_M$ and satisfies $\widehat{\Delta}^1_i(c)=0$, for all $c\in Z(\overline{C}_{i+1})$. 
 \medskip

  Now, take an element $c\in \hueca{B}_{i+1}^w(s,t)$. 
   Consider the homogeneous morphism $g_i(c):Me_t\rightmap{}Me_s$ of degree $\vert c\vert$ given by 
   $$g_i(c):=
   \widehat{d}(\widehat{h}_i^1)(c)+u_{i+1}^1(c)h^0+h^0u_{i+1}^1(c)
   +(u^1_i\cdot h^1_i)(c)+(h^1_i\cdot u^1_i)(c).$$
   Write $\lambda(c):=u^1_{i+1}(c)h^0+h^0u_{i+1}^1(c)$, $\rho(c)=(u^1_i\cdot h^1_i)(c)+(h^1_i\cdot u_{i}^1)(c)$ and $\gamma(c)=\widehat{d}(\widehat{h}^1_i)(c)$ . Then,
   the same calculations used at the beginining of the preceding case  show that the following two equalities hold
 $$\begin{matrix}
 u^0\lambda(c)-\lambda(c)u^0
 &=& (h^0\widehat{d}(u^1_{i+1}))(c)+(h^0u^1_i\cdot u_i^1)(c)- (\widehat{d}(u^1_{i+1})h^0)(c)
    -(u^1_i\cdot u^1_ih^0)(c)\hfill\\
   \end{matrix}$$
 and 
 $$\begin{matrix}
 u^0\rho(c)-\rho(c)u^0
 &=&
 -(  \widehat{d}(h^1_i)\cdot u_i^1)(c)-(h^1_i\cdot[u^0u^1_i+u^1_iu^0+u^1_{i-1}\cdot u^1_{i-1}])(c)\hfill\\
 &&+\,
 (u^1_i\cdot \widehat{d}(h^1_i))(c)+([u^0u^1_i+u^1_iu^0+u^1_{i-1}\cdot u^1_{i-1}]\cdot h^1_i)(c)\hfill\\
 &&+\,
 (u^1_i\cdot u^1_ih^0)(c)-(h^0u^1_i\cdot u^1_i)(c).\hfill \\
 \end{matrix}$$
 Applying again Leibniz formula of  (\ref{R: notacion producto de h1 con u1}), we have 
  $$(\widehat{d}(h^1_i)\cdot u^1_i)(c)=
  (h^1_i\cdot \widehat{d}(u^1_i))(c)+
  \widehat{d}(h^1_i\cdot u^1_i)(c)$$
  and 
  $$(u^1_i\cdot\widehat{d}(h^1_i))(c)=
  (\widehat{d}(u^1_i)\cdot h^1_i)(c)-
  \widehat{d}(u^1_i\cdot h^1_i)(c).$$
 Thus,      
  $$\begin{matrix}
 u^0\rho(c)-\rho(c)u^0
 &=&
 -(h^1_i\cdot[\widehat{d}(u^1_i)+u^0u^1_i+u^1_iu^0+u^1_{i-1}\cdot u^1_{i-1}])(c)\hfill\\
 &&+\,
 ([\widehat{d}(u^1_i)+u^0u^1_i+u^1_iu^0+u^1_{i-1}\cdot u^1_{i-1}]\cdot h^1_i)(c)\hfill\\
 &&+\,
 (u^1_i\cdot u^1_ih^0)(c)-(h^0u^1_i\cdot u^1_i)(c)\hfill\\
 &&-\,\widehat{d}[h^1_i\cdot u^1_i+u^1_i\cdot h^1_i](c).\hfill\\
  \end{matrix}$$
From  (\ref{R: transferencia de dif en GMod-B a GMod'-B})(A), we obtain  
 $$u^0\rho(c)-\rho(c)u^0=
 (u^1_i\cdot u^1_ih^0)(c)-(h^0u^1_i\cdot u^1_i)(c)
 -\widehat{d}[h^1_i\cdot u^1_i+u^1_i\cdot h^1_i](c).$$
 
Since  $u^0\gamma(c)-\gamma(c) u^0=u^0\widehat{d}(\widehat{h}^1_i)(c)-\widehat{d}(\widehat{h}^1_i)(c)u^0$,
we have 
 $$\begin{matrix}
 u^0g_i(c)-g_i(c)u^0
 &=&
 u^0\lambda(c)-\lambda(c)u^0+ u^0\rho(c)-\rho(c)u^0+u^0\gamma(c)-\gamma(c) u^0
 \hfill\\
 &=&
  [u^0\widehat{d}(\widehat{h}^1_i)-\widehat{d}(\widehat{h}^1_i)u^0](c)-\widehat{d}[h^1_i\cdot u^1_i+u^1_i\cdot h^1_i](c)\hfill\\
  &&+\,
  h^0\widehat{d}(u_{i+1}^1)(c)-\widehat{d}(u_{i+1}^1)h^0(c).\hfill\\
 \end{matrix}$$

Moreover, $\delta(c)\in Z(\overline{C}_{i+1})$, so  $\widehat{\Delta}_i^1(\delta(c))=0$. Hence, 
$$\begin{matrix}
\widehat{d}[u^1_i\cdot h^1_i+h^1_i\cdot u_i^1](c)[m]
&=&
(-1)^{\vert m\vert +1}[u^1_i\cdot h^1_i+h^1_i\cdot u_i^1](\delta(c))[m]\hfill\\
&=&
(-1)^{\vert m\vert}[u^1_{i+1}h^0+h^0u^1_{i+1}+u^0\widehat{h}^1_i+\widehat{h}^1_iu^0](\delta(c))[m]\hfill\\
&=&
(-1)^{\vert m\vert} (u_{i+1}^1h^0)(\delta(c))[m]+(-1)^{\vert m\vert}(h^0u_{i+1}^1)(\delta(c))[m]\hfill\\
&&+\,
(-1)^{\vert m\vert}(u^0\widehat{h}_i^1)(\delta(c))[m]+(-1)^{\vert m \vert}(\widehat{h}_i^1u^0)(\delta(c))[m]\hfill\\
&=&
[-\widehat{d}(u^1_{i+1}h^0)-\widehat{d}(h^0u^1_{i+1})-\widehat{d}(u^0\widehat{h}_i^1) -
\widehat{d}(\widehat{h}^1_iu^0)](c)[m].\hfill\\
\end{matrix}$$
Thus, 
$$\begin{matrix}
\widehat{d}[u^1_i\cdot h^1_i+h^1_i\cdot u_i^1](c)
&=&
[-\widehat{d}(u^1_{i+1})h^0+h^0\widehat{d}(u^1_{i+1})+u^0\widehat{d}(\widehat{h}_i^1) -
\widehat{d}(\widehat{h}^1_i)u^0](c).\hfill\\
\end{matrix}$$

Therefore, 
$u^0g_i(c)-g_i(c)u^0
=0$ and  $\tau g_i(c)$ is a morphism of complexes $(Me_t,u^0_{\vert Me_t})\rightmap{}(Me_s[\vert c\vert]),\tau(u^0_{\vert Me_s})\tau^{-1}).$ 
Proceeding as before, we get a homogeneous morphism $\widetilde{h}_i^1(c):Me_t\rightmap{}Me_s$ with degree $\vert c\vert-1$ such that 
$$g_i(c)=-u^0\widetilde{h}_i^1(c)-\widetilde{h}_i^1(c)u^0.$$
Then, we have the equality
$$\widetilde{\Delta}^1_i(c):=\widehat{d}(\widehat{h}^1_i)(c)+u_{i+1}^1(c)h^0+h^0u_{i+1}^1(c)+(u_i^1\cdot h_i^1)(c)+(h_i^1\cdot u_i^1)(c)
+u^0\widetilde{h}^1_i(c)+\widetilde{h}^1_i(c)u^0=0.$$
Consider the homogeneous morphism $\widetilde{h}^1_i:W_{i+1}\rightmap{}\Hom_{\GM\g k}(M,M)$ of degree $-1$  determined by the given values $\widetilde{h}_i^1(c)$ on the basic elements $c\in \hueca{B}^w_{i+1}$.
Then, the morphisms $h^1_i$,  $\widehat{h}_i^1$ and $\widetilde{h}_i^1$ determine a homogeneous morphism 
$$h^1_{i+1}:\overline{C}_{i+1}=\overline{C}_i\oplus V_{i+1}\oplus W_{i+1}\rightmap{}E_M$$
of degree $-1$ such that the equation $\Delta^1_i(c)=0$ is satisfied for all $c\in \overline{C}_{i+1}$, because either $\Delta^1_{i-1}(c)=0$, $\widehat{\Delta}^1_{i}(c)=0$ or  $\widetilde{\Delta}_{i}^1(c)=0$ hold on basic elements $c\in \hueca{B}_{i+1}$. Clearly, $h^1_{i+1}$ extends $h^1_i$. 
 \end{proof}

 \begin{theorem}\label{T: f0 quasi iso <-> f equiv homotopica}
 Assume that $S$ is a finite product of copies of the field $k$. 
 Let  ${\cal B}$ be a triangular differential graded $S$-bocs and $f=(f^0,f^1):(M,u_M)\rightmap{}(N,u_N)$ a morphism in $\TGModi^0\g{\cal B}$. Then $f$ is a homotopy equivalence iff the morphism of complexes of right $S$-modules $f^0:(M,u^0_M)\rightmap{}(N,u^0_N)$ is a quasi-isomorphism.
 \end{theorem}

 \begin{proof} Consider the homotopy category $\underline{\TGModi}^0\g{\cal B}$ with its triangulated structure as in (\ref{T: TGMod-A estable es triangulada}). Then, there is a triangle of the form 
 $$(M,u_M)\rightmap{\underline{f}}(N,u_N)\rightmap{\underline{g}}(L,u_L)\rightmap{}(M,u_M)[1]$$
 of $\underline{\TGModi}^0\g{\cal B}$. 
 By definition of this triangular structure, the preceding triangle is 
 isomorphic to a triangle of the form 
 $$(K,u_K)\rightmap{\underline{\phi}}(E,u_E)\rightmap{\underline{\gamma}}(Q,u_Q)\rightmap{}(K,u_K)[1],$$
 where 
 $$\xi\hbox{ \ }:\hbox{ \ }(K,u_K)\rightmap{\phi}(E,u_E)\rightmap{\gamma}(Q,u_Q)$$ is an exact pair in the exact structure ${\cal E}$ of $\TGModi^0\g{\cal B}$. Consider  an   isomorphism of triangles 
 $$\begin{matrix}
    (M,u_M)&\rightmap{\underline{f}}&(N,u_N)&\rightmap{\underline{g}}&(L,u_L)&\rightmap{}&(M,u_M)[1]\\
    \lmapdown{\underline{s}}&&\lmapdown{\underline{t}}&&\lmapdown{\underline{r}}&&\lmapdown{\underline{s}[1]}\\
    (K,u_K)&\rightmap{\underline{\phi}}&(E,u_E)&
    \rightmap{\underline{\gamma}}&(Q,u_Q)&\rightmap{}&(K,u_K)[1].\\
   \end{matrix}$$
Then, the morphisms $s=(s^0,s^1):(M,u_M)\rightmap{}(K,u_K)$ and $t=(t^0,t^1):(N,u_N)\rightmap{}(E,u_E)$ are homotopy equivalences such that $\phi s$ is homotopic to $tf$. Therefore, the morphisms of complexes $\phi^0s^0$ and $t^0f^0$ from $(M,u^0_M)$ to $(E,u^0_E)$ are homotopic. As a consequence 
$$H^i(\phi^0)H^i(s^0)=H^i(t^0)H^i(f^0), \hbox{ for all } i\in \hueca{Z}.$$
Since $H^i(s^0)$ and $H^i(t^0)$ are isomorphisms for all $i\in \hueca{Z}$, then $H^i(f^0)$ is an isomorphism if and only if $H^i(\phi^0)$ is an isomorphism. Thus $f^0$ is a  quasi-isomorphism  iff $\phi^0$ is a quasi-isomorphism. Clearly, $f$ is a homotopy equivalence iff $\phi$ is so. 

Since the exact pair $\xi$ belongs to ${\cal E}$, we have the exact sequence of graded right  $S$-modules
$$0\rightmap{}K\rightmap{\phi^0}E\rightmap{\gamma^0}Q\rightmap{}0.$$
Then, we have the exact sequence of complexes of right $S$-modules 
$$0\rightmap{}(K,u^0_K)\rightmap{\phi^0}(E,u^0_E)\rightmap{\gamma^0}(Q,u^0_Q)\rightmap{}0.$$ 

In the homology long exact sequence associated to the preceding exact sequence of complexes, we see that $H^i(\phi^0)$ is an isomorphism for all $i\in \hueca{Z}$ iff  $H^i(Q,u^0_Q)=0$ for all $i\in \hueca{Z}$. So $\phi^0$ is a quasi-isomorphism iff $(Q,u^0_Q)$ is acyclic. Then, by  (\ref{P: (M,u0) aciclico sii (M,u) homotopicamente trivial}), we obtain that $\phi^0$ is a quasi-isomorphism iff $(Q,u_Q)$ is homotopically trivial in $\TGModi^0\g{\cal B}$, which is equivalent to the fact that $\underline{\phi}:(K,u_K)\rightmap{}(E,u_E)$ is an isomorphism in $\underline{\TGModi}^0\g{\cal B}$. 
 \end{proof}

 \section{The Frobenius category of $A_\infty$-modules}
 
 Given a fixed $A_\infty$-algebra $A$, denote by ${\cal B}_A$  the differential tensor $S$-coalgebra (or differential tensor $S$-bocs) ${\cal B}_A=(T_S(A[1]),\mu,\epsilon,\delta)$ given by the bar construction. 
 In order to describe precisely  the connection of  the category $\Mod_\infty\g A$ of right $A_\infty$-modules over $A$ with the category $\TGModi\g{\cal B}_A$ of twisted  ${\cal B}_A$-modules, it is convenient to introduce the following categories. 
 
 \begin{definition}\label{D: GMod-A con A A-infinito-algebra}
  We will denote with $\GM\g A$ the following $k$-category. Its class of objects coincides with the class of objects of $\GM\g S$. Given two graded right $S$-modules $M$ and $N$, and $d\in \hueca{Z}$,  a \emph{homogeneous morphism $f:M\rightmap{}N$ of degree $d$ in $\GM\g A$} is a collection of morphisms  $f=\{f_n\}_{n\in \hueca{N}}$, where each 
  $$f_n:M\otimes_S A^{\otimes(n-1)}\rightmap{}N$$
  is a homogeneous morphism of graded right $S$-modules of 
  degree $\vert f_n\vert=d+1-n$. We denote by $\Hom^d_{\GM\g A}(M,N)$ the space of homogeneous morphisms from $M$ to $N$ in $\GM\g A$, and we make 
  $$\Hom_{\GM\g A}(M,N)=\bigoplus_{d\in \hueca{Z}} \Hom^d_{\GM\g A}(M,N).$$
  
  If $f\in \Hom_{\GM\g A}(M,N)$  and 
 $g\in \Hom_{\GM\g A}(N,L)$ are homogeneous morphisms,
 their composition 
 $$g\circ f=\{(g\circ f)_n\}_{n\in \hueca{N}}:M\rightmap{}L$$ 
 is defined, for each  $n\in \hueca{N}$, by 
 $$(g\circ f)_n=\sum_{\scriptsize\begin{matrix}r+s=n\\ r\geq 1;
 s\geq 0\end{matrix}}(-1)^{(\vert f\vert +r+1)s}
 g_{1+s}( f_r\otimes id^{\otimes s}).$$
 Given $M\in \GM\g A$, the identity morphism 
 $\I_M=\{h_n\}:M\rightmap{}M$ is given by $h_1=id_M$ and $h_n=0$,
 for all $n\geq 2$.

 The category $\GM\g A$ is a graded category with differential 
 $\delta_\infty$ defined for any homogeneous $f\in \Hom_{\GM\g A}(M,N)$ by 
 $$\delta_\infty(f)_n=\sum_{\scriptsize\begin{matrix}r+s+t=n\\ 
 t\geq 0;
 r,s\geq 1\end{matrix}}(-1)^{\vert f\vert +r+st+1}
 f_{n-s+1}(id^{\otimes r}\otimes m_s\otimes id^{\otimes t}),$$
 for $n\geq 1$, thus $\delta_{\infty}(f)_1=0$. 
 
 A morphism $f=\{f_n\}_{n\in \hueca{N}}:M\rightmap{}N$ in $\GM\g A$ is called \emph{strict} iff $f_n=0$, for all $n\geq 2$. 
 \end{definition}

 The fact that the preceding notions give rise indeed to a differential graded category is a consequence of the following.
 
 \begin{proposition}\label{P: equiv entre GMod-A y GMod-B(A)}
  Let ${\cal B}_A=(T_S(A[1]),\mu,\epsilon,\delta)$ denote the differential tensor $S$-bocs
 associated to the $A_\infty$-algebra $A$.
 Then there is an equivalence of differential graded $k$-categories
 $$\hueca{G}:\GM\g A\rightmap{}\GM\g {\cal B}_A.$$
 Given $M\in \GM\g A$, by definition $\hueca{G}(M)=M[1]$, so it  acts as the usual translation on objects. 
  Given two graded right $S$-modules $M$, $N$ and  a homogeneous morphism $f=\{f_n\}_{n\in \hueca{N}}:M\rightmap{}N$ in $\GM\g A$ with degree $\vert f\vert=d$, 
 we have the family of morphisms of right $S$-modules 
  $$\{\hat{f}_n:M[1]\otimes_S A[1]^{\otimes n}\rightmap{}N[1]\}_{n\geq 0}$$ 
  determined 
  by the commutativity of the following diagrams
   $$\begin{matrix}
  M\otimes_S A^{\otimes n}&
  \rightmap{ \ \sigma_M\otimes \sigma^{\otimes n} \ }&
  M[1]\otimes_S A[1]^{\otimes n}\\
     \lmapdown{f_{n+1}}&&\lmapdown{\hat{f}_{n}}\\
  N&\rightmap{ \  \ \ \sigma_N \ \  \ }&N[1]\\
  \end{matrix} \hbox{ and  \  \ \  } 
  \begin{matrix}
  M&\rightmap{ \ \zeta\sigma_M \ \ }&M[1]\otimes_S S\\
     \rmapdown{f_1}&&\rmapdown{\hat{f}_0}\\
  N&\rightmap{ \ \ \  \sigma_N \   \ \ }&N[1],\\
  \end{matrix}$$
  where $n$ runs in $\hueca{N}$ and $\zeta:M[1]\rightmap{}M[1]\otimes_SS$ is the canonical isomorphism. Each morphism 
   $\hat{f}_n$ is homogeneous with degree $\vert\hat{f}_n\vert=d$. The family of maps 
  $\{\hat{f}_n\}_{n\geq 0}$ extends to a homogeneous morphism of right $S$-modules  of degree $d$
  $$\hat{f}:M[1]\otimes T_S(A[1])\rightmap{}N[1].$$
  By definition, we have $\hueca{G}(f)=\hat{f}\in \Hom^d_{\GM\g{\cal B}_A}(M[1],N[1])$. 
 \end{proposition}

 \begin{proof} \emph{Step 1:  We have indeed a graded $k$-category $\GM\g A$ and $\hueca{G}$ is an equivalence of graded $k$-categories.}
 \medskip
 
Given composable morphisms $f:M\rightmap{}N$ and $g:N\rightmap{}L$
in $\GM\g A$, we know that the composition 
$\hueca{G}(g)*\hueca{G}(f)=\hat{g}*\hat{f}:M[1]\rightmap{}L[1]$ 
in $\GM\g{\cal B}_A$ is given by the composition of maps 
$$\hat{g}*\hat{f}=
\left(M[1]\otimes_SC\rightmap{ \  \ id_{M[1]}\otimes \mu \  \ }M[1]\otimes_SC\otimes_SC\rightmap{ \  \ \hat{f}\otimes id_C \  \ }N[1]\otimes_SC\rightmap{\hat{g}}L[1]\right),$$
where $C=T_S(A[1])$. 
Consider the restrictions $(\hat{g}*\hat{f})_n:M[1]\otimes A[1]^{\otimes n}\rightmap{}L[1]$, for $n\geq 0$, and let us verify that the following diagrams commute 
  $$\begin{matrix}
  M\otimes_S A^{\otimes n}&
  \rightmap{ \ \sigma_M\otimes \sigma^{\otimes n} \ }&
  M[1]\otimes_S A[1]^{\otimes n}\\
     \lmapdown{(g\circ f)_{n+1}}&&\lmapdown{(\hat{g}*\hat{f})_{n}}\\
  L&\rightmap{ \  \ \ \sigma_L \ \  \ }&L[1]\\
  \end{matrix} \hbox{ and  \  \ \  } 
  \begin{matrix}
  M&\rightmap{ \ \zeta\sigma_M \ \ }&M[1]\otimes_S S\\
     \rmapdown{(g\circ f)_1}&&\rmapdown{(\hat{g}*\hat{f})_0}\\
  L&\rightmap{ \ \ \  \sigma_L \   \ \ }&L[1],\\
  \end{matrix}$$
 where $n$ runs in $\hueca{N}$. 
 For this, it is convenient to write the restriction $\mu_n:A[1]^{\otimes n}\rightmap{}C\otimes C$ of the bar comultiplication $\mu$ on $T_S(A[1])$ as follows 
 $$\mu_n=\sum_{\scriptsize\begin{matrix} r+s=n\\ r,s\geq 1;
  \end{matrix}}id^{\otimes r}\otimes id^{\otimes s} + 1\otimes id^{\otimes n}+id^{\otimes n}\otimes 1.$$
 Then, for $n\geq 1$, we have 
 $$\begin{matrix}
(\hat{g}*\hat{f})_n
&=&
\sum_{\scriptsize\begin{matrix} r+s=n\\ r,s\geq 1;
  \end{matrix}}\hat{g}_s(\hat{f}_r(id_{M[1]}\otimes id^{\otimes r})\otimes id^{\otimes s})\hfill\\
  &&+\,
  \hat{g}_n(\hat{f}_0(id_{M[1]}\otimes 1)\otimes id^{\otimes n})
  +\hat{g}_0(\hat{f}_n(id_{M[1]}\otimes id^{\otimes n})\otimes 1).\hfill\\
  &=&
 \sum_{\scriptsize\begin{matrix} r+s=n\\ r,s\geq 1;
  \end{matrix}}\hat{g}_s(\hat{f}_r\otimes id^{\otimes s})
  +
  \hat{g}_n(\hat{f}_0\zeta\otimes id^{\otimes n})
  +\hat{g}_0(\zeta\hat{f}_n).\hfill\\ 
  \end{matrix}$$
  Write $(\hat{g}*\hat{f})_n=X_n^{(1)}+X_n^{(2)}+X_n^{(3)}$ with 
  $X_n^{(1)}:=
 \sum_{\scriptsize\begin{matrix} r+s=n\\ r,s\geq 1;
  \end{matrix}}\hat{g}_s(\hat{f}_r\otimes id^{\otimes s})$, 
  $X_n^{(2)}=\hat{g}_0(\zeta\hat{f}_n)$ and  $X_n^{(3)}=\hat{g}_n(\hat{f}_0\zeta\otimes id^{\otimes n})$.
  We want to  show that, for $n\geq 1$, we have   
  $$(\hat{g}*\hat{f})_n(\sigma_M\otimes \sigma^{\otimes n})=\sigma_L (g\circ f)_{n+1}.$$
  By definition of $g\circ f$,  for $n\geq 1$, we have $(g\circ f)_n=Y_n^{(1)}+Y_n^{(2)}+Y_n^{(3)}$, where 
  $$Y_n^{(1)}=\sum_{\scriptsize\begin{matrix}r+s=n\\ r\geq 2;
 s\geq 1\end{matrix}}
 (-1)^{(\vert f\vert +r+1)s}
 g_{1+s}( f_r\otimes id^{\otimes s}),$$
 $Y_n^{(2)}=g_1(f_n)$, $Y_1^{(3)}=0$ and, for $n\geq 2$, $Y_n^{(3)}=(-1)^{\vert f\vert (n-1)}g_n(f_1\otimes id^{\otimes(n-1)})$. 
We have 
  $$\begin{matrix}
X_n^{(1)}(\sigma_M\otimes \sigma^{\otimes n})
&=&
    \sum_{\scriptsize\begin{matrix}r+s=n\\ r,s\geq 1\end{matrix}}
    \hat{g}_s(\hat{f}_{r}\otimes id^{\otimes s})
 (\sigma_M\otimes \sigma^{\otimes n})\hfill\\
   &=&
   \sum_{\scriptsize\begin{matrix}r+s=n\\ r,s\geq 1\end{matrix}}
   \hat{g}_s(\hat{f}_{r}\otimes id^{\otimes s})
 (\sigma_M\otimes \sigma^{\otimes r}\otimes \sigma^{\otimes s})\hfill\\
  
 &=&
   \sum_{\scriptsize\begin{matrix}r+s=n\\ r,s\geq 1\end{matrix}}
 \hat{g}_s(\hat{f}_r
 (\sigma_M\otimes \sigma^{\otimes r})\otimes \sigma^{\otimes s})\hfill\\
  &=&
    \sum_{\scriptsize\begin{matrix}r+s=n\\ r,s\geq 1\end{matrix}}
\hat{g}_s
 (\sigma_Nf_{r+1}\otimes \sigma^{\otimes s})\hfill\\
  &=&
    \sum_{\scriptsize\begin{matrix}r+s=n\\ r,s\geq 1\end{matrix}}
    (-1)^{(\vert f\vert -r)s}
 \hat{g}_s(\sigma_N\otimes \sigma^{\otimes s})
 (f_{r+1}\otimes id^{\otimes s})\hfill\\
 &=&
 \sum_{\scriptsize\begin{matrix}r+s=n\\ r,s\geq 1\end{matrix}}
 (-1)^{(\vert f\vert -r)s}\sigma_Lg_{1+s}(f_{r+1}\otimes id^{\otimes s})\hfill\\
  &=&
 \sum_{\scriptsize\begin{matrix}r'+s=n+1\\ r'\geq 2; s\geq 1\end{matrix}}
 (-1)^{(\vert f\vert +r'-1)s}\sigma_Lg_{1+s}(f_{r'}\otimes id^{\otimes s})\hfill\\
  &=&
 \sigma_L Y_{n+1}^{(1)}.\hfill\\
   \end{matrix}$$
We also  have 
$$\begin{matrix}
   X_n^{(2)}(\sigma_M\otimes \sigma^{\otimes n})
    &=& 
  \hat{g}_0\zeta\hat{f}_n(\sigma_M\otimes \sigma^{\otimes n})\hfill\\
  &=&
  \hat{g}_0\zeta \sigma_Nf_{n+1}=\sigma_L g_1(f_{n+1})=\sigma_LY_{n+1}^{(2)}\hfill\\
   \end{matrix}$$
and 
$$\begin{matrix}
   X_n^{(3)}(\sigma_M\otimes \sigma^{\otimes n})
  &=&
  \hat{g}_n(\hat{f}_0\zeta\otimes id^{\otimes n})(\sigma_M\otimes\sigma^{\otimes n})\hfill\\
    &=&
   \hat{g}_n(\hat{f}_0\zeta\sigma_M\otimes \sigma^{\otimes n})\hfill\\
    &=&
  \hat{g}_n(\sigma_N f_1\otimes \sigma^{\otimes n})\hfill\\
    &=&
 (-1)^{\vert f\vert n}\hat{g}_n(\sigma_N \otimes \sigma^{\otimes n})(f_1\otimes id^{\otimes n})\hfill\\
    &=&
(-1)^{\vert f\vert n}\sigma_Lg_{n+1}(f_1\otimes id^{\otimes n})=\sigma_LY_{n+1}^{(3)}.\hfill\\
   \end{matrix}$$

Moreover, we have 
$$\begin{matrix}
   (\hat{g}*\hat{f})_0\zeta\sigma_M
   &=&
    (\hat{g}*\hat{f})_0(\sigma_M\otimes 1)\hfill
    &=&
    \hat{g}(\hat{f}\otimes id_C)(id_{M[1]}\otimes \mu)(\sigma_M\otimes 1)\hfill\\
    &=&
    \hat{g}(\hat{f}\otimes id_C)(\sigma_M\otimes 1\otimes 1)\hfill
    &=&
    \hat{g}(\hat{f}(\sigma_M\otimes 1)\otimes 1)\hfill\\
    &=&
    \hat{g}_0\zeta\hat{f}_0\zeta\sigma_M\hfill
    &=&
    \hat{g}_0\zeta\sigma_Nf_1\hfill\\
    &=&
    \sigma_Lg_1f_1\hfill
    &=&
    \sigma_L(g\circ f)_1.\hfill\\
  \end{matrix}$$

 From the preceding calculations, we obtain that  
 $\hueca{G}(g\circ f)=\hueca{G}(g)*\hueca{G}(f)$. Take any $M\in \GM\g A$ and consider the morphism 
 $h=\{h_n\}_{n\in \hueca{N}}:M\rightmap{}M$ in $\GM\g A$ given by $h_1=id_M$ and $h_n=0$, for $n\geq 2$. 
 Let us show that $\hueca{G}(h)=\hueca{I}_{M[1]}$, the identity morphism on the object $\hueca{G}(M)\in \GM\g {\cal B}_A$. 
  
Clearly, we have  $\hat{h}_0=\sigma_M h_1(\zeta\sigma_M)^{-1}=
\sigma_M \sigma_M^{-1}\zeta^{-1}=\zeta^{-1}$ and, for $n\geq 1$, we have that 
$\hat{h}_n=\sigma_M h_{n+1}(\sigma_M\otimes \sigma^{\otimes n})^{-1}=0$. This means that the components $(\hat{h}^0,\hat{h}^1)$ of the morphism $\hueca{G}(h)=\hat{h}$  are $(id_{M[1]},0)$, so $\hueca{G}(h)=\hueca{I}_{M[1]}$. 

It is clear that the association $f\longmapsto\hat{f}$ is an isomorphism of graded vector spaces 
$\Hom_{\GM\g A}(M,N)\rightmap{}\Hom_{\GM\g{\cal B}_A}(M[1],N[1])$. It follows that $\GM\g A$ is a graded $k$-category and that $\hueca{G}:\GM\g A\rightmap{}\GM\g {\cal B}_A$ is an equivalence of categories.

 \medskip
 \emph{ Step 2: We have indeed a differential $\delta_\infty$ on the graded category $\GM\g A$ and $\hueca{G}$ is an equivalence of differential graded $k$-categories.}
 \medskip
 
 It will be enough to show that for any homogeneous  morphism $f:M\rightmap{}N$ in $\GM\g A$, we have 
$$\hueca{G}(\delta_\infty(f))=\widehat{\delta}(\hueca{G}(f)).$$ 

Consider the restrictions $\widehat{\delta}(\hat{f})_n:M[1]\otimes A[1]^{\otimes n}\rightmap{}N[1]$, for $n\geq 0$.  Let us verify that the following diagrams commute 
  $$\begin{matrix}
  M\otimes_S A^{\otimes n}&
  \rightmap{ \ \sigma_M\otimes \sigma^{\otimes n} \ }&
  M[1]\otimes_S A[1]^{\otimes n}\\
     \lmapdown{\delta_\infty(f)_{n+1}}&&\lmapdown{\widehat{\delta}(\hat{f})_n}\\
  N&\rightmap{ \  \ \ \sigma_N \ \  \ }&N[1]\\
  \end{matrix} \hbox{ and  \  \ \  } 
  \begin{matrix}
  M&\rightmap{ \ \zeta\sigma_M \ \ }&M[1]\otimes_S S\\
     \rmapdown{\delta_\infty(f)_1}&&\rmapdown{\widehat{\delta}(\hat{f})_0}\\
  N&\rightmap{ \ \ \  \sigma_N \   \ \ }&N[1],\\
  \end{matrix}$$
  where $n$ runs in $\hueca{N}$. 
  
By definition,  the differential $\widehat{\delta}$ of the category $\GM\g {\cal B}_A$ applied to the morphism  $\hueca{G}(f)=\hat{f}$ is  
$$\widehat{\delta}(\hat{f})=(-1)^{\vert \hat{f}\vert+1}\hat{f}(id_{M[1]}\otimes \delta).$$

 Recall that the differential $\delta$ on the tensor $S$-coalgebra $T_S(A[1])$ satisfies that $\delta(S)=0$ and  is determined by the morphisms of right $S$-modules 
 $$\delta_n:A[1]^{\otimes n}\rightmap{}
 \bigoplus_{1\leq s\leq n}A[1]^{\otimes (n-s+1)}$$
 given, for $n\geq 1$, by the following formula  
 $\delta_n=
\sum_{\scriptsize\begin{matrix} r+s+t=n\\ r,t\geq 0; s\geq 1 \end{matrix}} 
id^{\otimes r}\otimes \hat{m}_s\otimes id^{\otimes t},$
where $\hat{m}_s$ is defined, for $s\geq 1$, by the following commutative square
$$\begin{matrix}
  A^{\otimes s}&\rightmap{ \ \ \sigma^{\otimes s} \ \ }&A[1]^{\otimes s}\\
     \rmapdown{m_s}&&\rmapdown{\hat{m}_s}\\
  A&\rightmap{ \ \ \  \sigma \   \ \ }&A[1].\\
  \end{matrix}$$
Here, if $r=0$ (or $t=0$) then $id^{\otimes r}$ (resp. $id^{\otimes t}$) is omitted.  

Then, for $n\geq 1$, making   
$\Delta_n:=\widehat{\delta}(\hat{f})_n(\sigma_M\otimes \sigma^{\otimes n})$ we have  
$$\begin{matrix}
\Delta_n 
  &=&
  \sum_{\scriptsize\begin{matrix} r+s+t=n\\ r,t\geq 0; s\geq 1 \end{matrix}} 
(-1)^{\vert \hat{f}\vert+1}\hat{f}(id_{M[1]}\otimes id^{\otimes r}\otimes \hat{m}_s\otimes id^{\otimes t})(\sigma_M\otimes \sigma^{\otimes n})\hfill\\
&=&
 \sum_{\scriptsize\begin{matrix} r+s+t=n\\ r,t\geq 0; s\geq 1 \end{matrix}} 
(-1)^{\vert f\vert+r}\hat{f}_{n-s+1}(\sigma_M\otimes \sigma^{\otimes r}\otimes \hat{m}_s\sigma^{\otimes s}\otimes \sigma^{\otimes t})\hfill\\
&=&
\sum_{\scriptsize\begin{matrix} r+s+t=n\\ r,t\geq 0; s\geq 1 \end{matrix}} 
(-1)^{\vert f\vert+r}\hat{f}_{n-s+1}(\sigma_M\otimes \sigma^{\otimes r}\otimes \sigma m_s\otimes \sigma^{\otimes t})\hfill\\
&=&
\sum_{\scriptsize\begin{matrix} r+s+t=n\\ r,t\geq 0; s\geq 1 \end{matrix}} 
(-1)^{\vert f\vert+r+st}\hat{f}_{n-s+1}(\sigma_M\otimes \sigma^{\otimes n})(id_{M}\otimes id^{\otimes r}\otimes  m_s\otimes id^{\otimes t})\hfill\\
&=&
\sum_{\scriptsize\begin{matrix} r+s+t=n\\ r,t\geq 0; s\geq 1 \end{matrix}} 
(-1)^{\vert f\vert+r+st}\sigma_Nf_{n-s+2}(id_{M}\otimes id^{\otimes r}\otimes  m_s\otimes id^{\otimes t})\hfill\\
&=&
\sum_{\scriptsize\begin{matrix} r+s+t=n+1\\ t\geq 0; r,s\geq 1 \end{matrix}} 
(-1)^{\vert f\vert+r+st+1}\sigma_Nf_{(n+1)-s+1}(id^{\otimes r}\otimes  m_s\otimes id^{\otimes t})\hfill\\
&=&
\sigma_N\delta_\infty(f)_{n+1}.\hfill\\
  \end{matrix}$$

Finally, we have 
$$\begin{matrix}
  \widehat{\delta}(\hat{f})_0\zeta\sigma_M
  &=& 
   \widehat{\delta}(\hat{f})_0(\sigma_M\otimes 1)\hfill\\
   &=&
   (-1)^{\vert \hat{f}\vert +1}\hat{f}(id_{M[1]}\otimes \delta)(\sigma_M\otimes 1)\hfill\\
   &=&0=\sigma_N\delta_{\infty}(f)_1.\hfill\\
  \end{matrix}$$
 This finishes the proof. 
\end{proof}

\begin{definition}\label{D: TGMod-A con A A-infinito-algebra}
 Given an $A_\infty$-algebra $A=(A,\{m_n\})$,  we will denote by $\TGM\g A$ the following graded  $k$-category. 
 Its objects $(M,m^M)$ are graded right $S$-modules $M$ equipped with a morphism $m^M:M\rightmap{}M$ in $\GM\g A$ with degree $\vert m^M\vert=1$ such that the following Maurer Cartan equation holds
 $$\delta_\infty(m^M)+m^M\circ m^M=0.$$
 A  \emph{homogeneous morphism $f:(M,m^M)\rightmap{}(N,m^N)$ in $\TGM\g A$ with degree $\vert f\vert=d$} is a homogeneous morphism 
 $f:M\rightmap{}N$ is $\GM\g A$ with degree $d$ such that the following  equation 
 holds
 $$\delta_\infty(f)+m^N\circ f-(-1)^{\vert f\vert}f\circ m^M=0.$$
 The composition in $\TGM\g A$ is the same composition of $\GM\g A$.
 \end{definition}
 
  Again, the fact that the preceding notions give rise indeed to a  graded $k$-category follows from the following statement, which is quite clear. 
 
 \begin{proposition}\label{P: equiv entre TGMod-A y TGMod-B(A)}
  Let ${\cal B}_A$ denote the differential tensor $S$-bocs
 associated to the $A_\infty$-algebra $A$, then there is an 
 equivalence of graded $k$-categories
 $$\hueca{G}:\TGM\g A\rightmap{}\TGM\g {\cal B}_A.$$
 Given $(M,m^M)\in \GM\g A$, by definition $\hueca{G}(M,m^M)=
 (M[1],\hueca{G}(m^M))$.  
 
  Given   a homogeneous  morphism $f:(M,m^M)\rightmap{}(N,m^N)$ in $\TGM\g A$  with degree $d$,  
  we look at the underlying morphism $f:M\rightmap{}N$ in $\GM\g A$ and apply $\hueca{G}$ to obtain a morphism $\hueca{G}(f):\hueca{G}(M)\rightmap{}\hueca{G}(N)$ in $\GM\g {\cal B}_A$, such that 
  $\hueca{G}(f):\hueca{G}(M,m^M)\rightmap{}\hueca{G}(N,m^N)$ is a morphism in $\TGM\g {\cal B}_A$. 
 \end{proposition}
 
 \begin{proposition}\label{P: translation to A-infinte modules}
  Let ${\cal B}_A$ denote the differential tensor $S$-bocs
 associated to an $A_\infty$-algebra $A$. Then, we have:
 \begin{enumerate}
  \item  The preceding requirements for a pair $(M,m^M)\in \TGM\g A$  mean that $(M,m^M)$  is a right $A_\infty$-module over $A$ in the usual sense of (\ref{D: A-infinto modulos derechos}).

 \item The preceding requirements for a family of morphisms 
 $f=\{f_n\}_{n\in \hueca{N}}$ to be a homogeneous morphism of degree $d$ from 
 $(M,m^M)$ to $(N,m^N)$ in $\TGM\g A$  translate into the following. They mean that $f$ is a family 
  $f=\{f_n\}_{n\in \hueca{N}}$, where each 
  $$f_n:M\otimes_S A^{\otimes (n-1)}\rightmap{}N$$
  is a homogeneous morphism of graded right $S$-modules of degree $\vert f_n\vert=d+1-n$ such that, for each $n\in \hueca{N}$, 
  the equality $\Sigma_n^{f +}+\Sigma_n^{f 0}+\Sigma_n^{f -}=0$ holds, where 
  $$\Sigma_n^{f +}=\sum_{\scriptsize\begin{matrix}r+s+t=n\\ t\geq 0; r,s\geq 1\end{matrix}}
  (-1)^{\vert f\vert+ r+st}
 f_{r+1+t}(id^{\otimes r}\otimes m_s\otimes id^{\otimes t}),$$
$$\Sigma_n^{f 0}=\sum_{\scriptsize\begin{matrix}s+t=n\\ t\geq 0; s\geq 1\end{matrix}}
(-1)^{\vert f\vert +st}
 f_{1+t}( m^M_s\otimes id^{\otimes t}),$$
and
$$\Sigma_n^{f-}=\sum_{\scriptsize\begin{matrix}r+s=n\\ r\geq 1; s\geq 0\end{matrix}}
(-1)^{(\vert f\vert+r+1)s+1}
 m^N_{1+s}(f_r\otimes id^{\otimes s}).$$
 The condition in case $n=1$ is equivalent to $m_1^Nf_1=f_1m_1^M$, 
 that is, to the requirement that the map $f_1:M\rightmap{}N$ is a morphism 
 of complexes of right $S$-modules.

 Then,  if we restrict to the case where $f$ has degree $\vert f\vert =0$, we have that $f:(M,m^M)\rightmap{}(N,m^N)$ is a morphism 
 of right $A_\infty$-modules in the usual sense of (\ref{D: A-infinto modulos derechos}). 
 \end{enumerate}
 Therefore, the usual category $\Mod_\infty\g A$ of right $A_\infty$-modules over the $A_\infty$-algebra $A$ is the subcategory $\TGM^0\g A$ of 
 $\TGM\g A$ with the same objects and the homogeneous morphisms of degree $0$ between them. Then, we can restrict the functor $\hueca{G}$ to an equivalence functor $\hueca{G}:\Mod_\infty\g A\rightmap{}\TGM^0\g{\cal B}_A$. 
\end{proposition}

\begin{proof} Item (1) follows from the observation that, for $n\geq 1$, we have 
$$\delta_\infty(m^M)_n=\Sigma_n^+\hbox{ \  and \ }(m^M\circ m^M)_n=\Sigma_n^0.$$
Item (2) follows from the observation that, for $n\geq 1$, we have
$$\begin{matrix}\delta_\infty(f)_n=-\sum_n^{f +}, \hbox{ \ } (m^N\circ f)_n=-\sum_n^{f -}, \hbox{ and } (-1)^{\vert f\vert}(f\circ m^M)_n=\sum_n^{f 0}.\end{matrix}$$
\end{proof}

 \begin{proposition}\label{P: estructura exacta de GMod-A}
Let $A$ be an $A_\infty$-algebra. Then, the category $\GM^0\g A$ is additive with split idempotents. 
Denote by ${\cal E}_0$ the class of composable morphisms $M\rightmap{f}E\rightmap{g}N$ in $\GM^0\g A$ such that $g\circ f=0$ and 
$$0\rightmap{}M\rightmap{f_1}E\rightmap{g_1}N\rightmap{}0$$
is an exact sequence in $\GM\g S$. Then we have:
\begin{enumerate}
 \item The pair  $(\GM^0\g A,{\cal E}_0)$ is an exact category. The class ${\cal E}_0$ consists of the split exact pairs in 
 $\GM^0\g A$. 
 \item A morphism $f:M\rightmap{}E$ in $\GM^0\g A$ is an ${\cal E}_0$-inflation iff $f_1:M\rightmap{}E$ is injective.
 \item A morphism $g:E\rightmap{}N$ in $\GM^0\g A$ is an  ${\cal E}_0$-deflation iff $g_1:E\rightmap{}N$ is surjective.
\end{enumerate}
\end{proposition}
 
  \begin{proof}
  Recall that the coalgebra  ${\cal B}_A=(T_S(A[1]),\mu,\epsilon,\delta)$ associated  by the bar construction to the $A_\infty$-algebra $A$ is a triangular differential graded $S$-bocs. Then, we have (\ref{P: split idempotents}) and  
  (\ref{P: estructura exacta de GMod-B}), and we can translate them to $\GM^0\g A$ with the equivalence $\hueca{G}:\GM^0\g A\rightmap{}\GM^0\g {\cal B}_A$. 
  \end{proof}

 \begin{proposition}\label{D: estructura exacta de TGMod-A}
 Let $A$ be an $A_\infty$-algebra.  Then, the category $\Mod_\infty\g A$ is additive with split idempotents. Denote by ${\cal E}_\infty$ the class of composable morphisms in $\Mod_\infty\g A$ 
 $$\begin{matrix}(M,m^M)&\rightmap{f}&(E,m^E)&\rightmap{g}&(N,m^N)\end{matrix}$$
 such that $g\circ f=0$ and the short sequence of first components  
 $$\begin{matrix}0&\rightmap{}&M&\rightmap{f_1}&E&\rightmap{g_1}&N&\rightmap{}&0\end{matrix}$$
 is an exact sequence in $\GM\g S$. Then we have the following.
\begin{enumerate}
 \item The pair  $(\Mod_\infty\g A,{\cal E}_\infty)$ is an exact category. The class ${\cal E}_\infty$  coincides with the class of composable morphisms such that when we
 forget the second components of the objects, we have a split exact pair in $\GM^0\g A$. 
 \item A morphism $f:(M,m^M)\rightmap{}(E,m^E)$ in $\Mod_\infty\g A$ is an ${\cal E}_\infty$-inflation iff $f_1:M\rightmap{}E$ is injective.
 \item A morphism $g:(E,m^E)\rightmap{}(N,m^N)$ in $\Mod_\infty\g A$ is an ${\cal E}_\infty$-deflation iff $g_1:E\rightmap{}N$ is surjective.
 \item The functor $\hueca{G}:(\Mod_\infty\g A,{\cal E}_\infty)\rightmap{}(\TGM^0\g {\cal B}_A,{\cal E})$ is an equivalence of exact categories. 
\end{enumerate} 
 \end{proposition}

\begin{proof}   As before, the $S$-coalgebra  ${\cal B}_A$  associated  by the bar construction to the $A_\infty$-algebra $A$ is a triangular differential graded $S$-bocs. Then, we can apply 
(\ref{P: split idempotents}) to $\TGM^0\g{\cal B}_A$. Moreover, we can translate  
  (\ref{P: estructura exacta de TGMod-B}) to $\Mod_\infty\g A=\TGM^0\g A$ with the equivalence functor $\hueca{G}:\Mod_\infty\g A\rightmap{}\TGM^0\g {\cal B}_A$. 
  \end{proof}
  
 \begin{remark}\label{R: obs sobre sigmaM en GMod-A}
 Given an $A_\infty$-algebra $A$, as we did in (\ref{R: remark sobre sigmaM para GMod-B}), we can associate to any 
 given $M\in \GM\g A$ the homogeneous isomorphism  $\underline{\sigma}_M:M\rightmap{}M[1]$ in $\GM\g A$ of degree $-1$ such that $\underline{\sigma}_M=\{\sigma_n\}_{n\in \hueca{N}}$, where $\sigma_1:=\sigma_M:M\rightmap{}M[1]$ and, for $n\geq 2$, 
 $\sigma_n:=0:M\otimes A^{\otimes(n-1)}\rightmap{}M[1]$. 
 \end{remark}

\begin{lemma}\label{D: traslacion para GM-A y TGM-A} Let $A$ be an $A_\infty$-algebra. Then, we have:
 \begin{enumerate}
 \item For each $M\in \GM\g A$, make $T(M)=M[1]$ and, 
 for any morphism $f:M\rightmap{}N$ in $\GM\g A$, make 
 $$T(f)=f[1]=\underline{\sigma}_N\circ f\circ \underline{\sigma}_M^{-1}:M[1]\rightmap{}N[1] \hbox{ in } 
 \GM\g A.$$
 Then, we have an autofunctor 
 $T:\GM\g A\rightmap{}\GM\g A$ with inverse $T^{-1}$ given by $T^{-1}(M)=M[-1]$ and 
 $$T^{-1}(f)=f[-1]=\underline{\sigma}_{N[-1]}^{-1}\circ f\circ \underline{\sigma}_{M[-1]}:M[-1]\rightmap{}N[-1] \hbox{ in } \GM\g A.$$ 
 It is called the shifting autofunctor of $\GM\g A$. 
 \item  For each $(M,m^M)\in \TGM\g A$, make $T(M,m^M)=(M[1],-m^M[1])$ and, 
 for any morphism $f:(M,m^M)\rightmap{}(N,m^N)$ in $\GM\g A$, make 
 $$T(f)=f[1]=\underline{\sigma}_N\circ f\circ \underline{\sigma}_M^{-1}:(M[1],-m^M[1])\rightmap{}(N[1],-m^N[1]) \hbox{ in } \TGM\g A.$$
 Then, we have an autofunctor 
 $T:\TGM\g A\rightmap{}\TGM\g A$ with inverse $T^{-1}$ given by $T^{-1}(M,m^M)=(M[-1],-m^M[-1])$ and 
 $$T^{-1}(f)=f[-1]=\underline{\sigma}_{N[-1]}^{-1}\circ f\circ \underline{\sigma}_{M[-1]}:(M[-1],-m^M[-1])\rightmap{}(N[-1],-m^N[-1]) .$$ 
 It is called the shifting autofunctor of $\TGM\g A$. 
\item The functor $\hueca{G}:\GM\g A\rightmap{}\GM\g {\cal B}_A$ commutes with the corresponding shifting autofunctors, and the functor $\hueca{G}:\TGM\g A\rightmap{}\TGM\g {\cal B}_A$ 
commutes with the corresponding shifting autofunctors.
\item Furthermore, the shifting functor $T:(\Mod_\infty\g A,{\cal E}_\infty)\rightmap{}(\Mod_\infty\g A,{\cal E}_\infty)$ is an  
 automorphism of exact categories. 
As a consequence, $T$ and $T^{-1}$ preserve the classes of ${\cal E}_\infty$-injectives and ${\cal E}_\infty$-projectives. 
\end{enumerate}
 \end{lemma}
 
\begin{proof} It follows from 
(\ref{D: traslacion para GM-B y TGM-B}).  
\end{proof}

 \begin{proposition}\label{P: construction of J for TGMod-A} 
 For $(M,m^M)\in \TGM\g A$, we write $m^{M[1]}:=-m^M[1]$, so we have $T(M,m^M)=(M[1],m^{M[1]})$. 
 There is an endofunctor  $$J: \TGM\g A\rightmap{}\TGM\g A$$ 
 such that, for any $(M,m^M)\in \TGM\g A$ we have 
 $$J(M,m^M)=(M\oplus M[1],\begin{pmatrix}
                            m^M&\underline{\sigma}^{-1}_M\\ 0&m^{M[1]}\\
                           \end{pmatrix})$$
 and,  given a morphism $f:(M,m^M)\rightmap{}(N,m^N)$ in $\TGM\g A$, the morphism 
 $J(f)$ is   given by the matrix 
 $$J(f)=\begin{pmatrix}
               f& 0\\ 0& f[1]
              \end{pmatrix}: J(M,m^M)\rightmap{}J(N,m^N).$$
The equivalence $\hueca{G}:\TGM\g A\rightmap{}\TGM\g {\cal B}_A$ satisfies $\hueca{G}J=J\hueca{G}$. 
 
 Moreover, there are natural transformations 
 $\alpha: id_{\Mod_\infty\g A }\rightmap{}J$ and $\beta:J\rightmap{}T$ of endofunctors of $\Mod_\infty\g A$ such that, for each $\underline{M}=(M,m^M)\in \Mod_\infty\g A$, we have an ${\cal E}_\infty$-conflation 
$$\begin{matrix}\underline{M}&\rightmap{ \ \alpha_{\underline{M}} \ }&J(\underline{M})&\rightmap{ \ \beta_{\underline{N}} \ }&
T(\underline{M}).\end{matrix}$$
For $\underline{M}\in  \Mod_\infty\g A$, the morphisms $\alpha_{\underline{M}}$ and $\beta_{\underline{M}}$ are strict morphisms with 
$$(\alpha_{\underline{M}})_1=(id_M,0)^t: M\rightmap{}M\oplus M[1]$$
and  
 $$(\beta_{\underline{M}})_1=(0,id_{M[1]}):M\oplus M[1]\rightmap{}M[1].$$   
 Then, the category $\Mod_\infty\g A$ has enough ${\cal E}_\infty$-projectives and enough ${\cal E}_\infty$-injectives and, moreover, 
  the class of ${\cal E}_\infty$-projectives coincides with the class of ${\cal E}_\infty$-injectives. So, we obtain that $\Mod_\infty\g A$ is a special Frobenius category.
 \end{proposition}
 
\begin{proof} It follows from (\ref{P: construction of J for TGMod-B}), and 
(\ref{P: pares exactos con termino medio J(M)}), which are translated into $\Mod_\infty\g A$ with the help of the functor $\hueca{G}$. Then, we use (\ref{D: estructura exacta de TGMod-A}) and (\ref{D: traslacion para GM-A y TGM-A}). 
\end{proof}
 
  \begin{lemma}\label{L: T y J de GMod-A con familias}
  Let $A$ be an $A_\infty$-algebra. Then, we have the following explicit descriptions for the shifting functor $T$ and the functor $J$.
  \begin{enumerate}
   \item Given $(M, m^M)\in \TGM\g A$, its  translate $(M,m^M)[1]=(M[1],m^{M[1]})$ is given for each $n\in \hueca{N}$, by 
$$m^{M[1]}_n=
(-1)^n \sigma_ M m_n^M(\sigma_M^{-1}\otimes id^{\otimes(n-1)}).$$  
\item Given a morphism $f:(M, m^ M)\rightmap{}(N, m^N)$ in $\TGM\g A$, its translate $f[1]:(M[1],m^{M[1]})\rightmap{}(N[1],m^{N[1]})$ is given, for each $n\in \hueca{N}$, by 
$$f[1]_n=(-1)^{n-1} \sigma_ N f_n(\sigma_M^{-1}\otimes id^{\otimes (n-1)}).$$ 
\item Given $(M,m^M)\in \TGM\g A$, we have $J(M,m^M)=(M\oplus M[1],m^{J(M)})$, where 
$$m_n^{J(M)}:M\otimes A^{\otimes (n-1)}\oplus M[1]\otimes A^{\otimes(n-1)}\rightmap{}M\oplus M[1]$$
with 
$$m_1^{J(M)}=\begin{pmatrix}
             m_1^M&\sigma_M^{-1}\\
             0&-m_1^M[1]\\
             \end{pmatrix} \hbox{ \  and \  }
        m_n^{J(M)}=\begin{pmatrix}
             m_n^M&0\\
             0&-m_n^M[1]\\
             \end{pmatrix}   \hbox{ for } n\geq 2.
             $$
\item Given a morphism $f:(M, m^ M)\rightmap{}(N, m^N)$ in 
$\TGM\g A$, its translate $J(f):(M\oplus M[1],m^{J(M)})\rightmap{}(N\oplus N[1],m^{J(N)})$ is given, for each $n\in \hueca{N}$, by 
$$J(f)_n:M\otimes A^{\otimes (n-1)}\oplus M[1]\otimes A^{\otimes(n-1)}\rightmap{}M\oplus M[1]$$
with matrix 
$$J(f)_n=\begin{pmatrix}
          f_n&0\\ 0&(-1)^{n-1} \sigma_ N f_n(\sigma_M^{-1}\otimes id^{(n-1)})\\
         \end{pmatrix}.$$
  \end{enumerate}
 \end{lemma}
 
 \begin{proof} (1) and (2): It will be enough to show that for any  
 morphism $f:M\rightmap{}N$ in $\GM\g A$, its translate $f[1]:M[1]\rightmap{}N[1]$ in $\GM\g A$ is given, for each $n\in \hueca{N}$, by 
$$f[1]_n=(-1)^{n-1} \sigma_ N f_n(\sigma_M^{-1}\otimes id^{\otimes(n-1)}).$$
 Indeed, from the definition of $f[1]$ and of the composition in $\GM\g A$, we have 
 $$\begin{matrix}
   f[1]_n
   &=&
   (\underline{\sigma}_N\circ f\circ \underline{\sigma}_M^{-1})_n\hfill\\
   &=&
  \sum_{\scriptsize\begin{matrix}r+s=n\\ r\geq 1;
 s\geq 0\end{matrix}}(-1)^{(\vert f\circ \underline{\sigma}_M^{-1}\vert +r+1)s}
 (\underline{\sigma}_N)_{1+s}( (f\circ\underline{\sigma}^{-1}_M)_r\otimes id^{\otimes s}) \hfill\\
 &=&
 \sigma_N(f\circ\underline{\sigma}^{-1}_M)_n\hfill\\
 &=&
 \sigma_N
 (\sum_{\scriptsize\begin{matrix}r+s=n\\ r\geq 1;
 s\geq 0\end{matrix}}(-1)^{(\vert \underline{\sigma}_M^{-1}\vert +r+1)s}
 f_{1+s}( (\underline{\sigma}_M^{-1})_r\otimes id^{\otimes s})\hfill\\
 &=&
 (-1)^{n-1}\sigma_Nf_n(\sigma_M^{-1}\otimes id^{\otimes (n-1)}).\hfill\\
   \end{matrix}$$

 Items (3) and (4) clearly follow from (1) and (2). 
 \end{proof}

 \begin{remark}\label{D: homotopia para GMod0-A es la original}
  Let $A$ be an $A_\infty$-algebra and consider a pair of morphisms  $f,g\in \Hom_{\Mod_\infty\g A}((M,m^M),(N,m^N))$. Then,  there is a homotopy $h=\{h_n\}_{n\in \hueca{N}}$ from $f$ to $g$, as in definition (\ref{D: homotopy for Mod-inf-A}), 
  iff there is a morphism $h\in \Hom^{-1}_{\GM\g A}(M,N)$ satisfying 
  $$f-g=\delta_\infty(h)+m^N\circ h+h\circ m^M.$$ 
  Indeed, the preceding formula translates into the usual formulas given in   (\ref{D: homotopy for Mod-inf-A}):   
  For $n\geq 1$, we have $\delta_\infty(h)_n=H_n^{(3)}$, $(m^N\circ h)_n=H^{(1)}_n$, and $(h\circ m^M)_n=H_n^{(2)}$. 
 \end{remark}
 
 \begin{proposition}\label{P: homotopias en TGM-A}
  Let $A$ be an $A_\infty$-algebra. Then, the functor 
  $$\hueca{G}:\Mod_\infty\g A\rightmap{}\TGM^0\g {\cal B}_A$$ preserves and reflects homotopies. Therefore, it induces an equivalence of these categories modulo homotopy
  $$\hueca{G}:\underline{\Mod}_\infty\g A\rightmap{}\underline{\TGM}^0\g {\cal B}_A$$
 \end{proposition}

 \begin{definition}\label{D: estructura triangular de TGMod-A}
  Let $A$ be an $A_\infty$-algebra. Then, 
  a sequence of morphisms 
  $$(M,m^M)\rightmap{\underline{f}}(E,m^E)\rightmap{\underline{g}}(N,m^N)\rightmap{\underline{h}}T(M,m^M)$$ 
  in the homotopy category $\underline{\Mod}_\infty\g A$ 
 such that 
 $$\xi: (M,m^M)\rightmap{f}(E,m^E)\rightmap{g}(N,m^N)$$
 is an exact pair in ${\cal E}_\infty$ and we have a commutative diagram of the form 
 $$\begin{matrix}
   \xi&:& (M,m^M)&\rightmap{f}&(E,m^E)&\rightmap{g}&(N,m^N)\\
    &&\parallel&&\lmapdown{}&&\lmapdown{h}\\
     \xi_M&:&(M,m^M)&\rightmap{\alpha_M}&J(M,m^M)&
     \rightmap{\beta_M}&T(M,m^M)\\
   \end{matrix}$$
   in $\Mod_\infty\g A$ is called a 
   \emph{canonical triangle} in  $\underline{\Mod}_\infty\g A$.
    Now, consider the class ${\cal T}_\infty$ of  sequences  of morphisms 
   $$X\rightmap{u}Y\rightmap{v}Z\rightmap{w}TX$$ 
   in $\underline{\Mod}_\infty\g A$ which are isomorphic to   canonical triangles. That is, 
   there is a canonical triangle 
  $M\rightmap{\underline{f}}E\rightmap{\underline{g}}N\rightmap{\underline{h}}TM$
   and isomorphisms $a,b,c$ such that the following diagram commutes 
    $$\begin{matrix}X&\rightmap{u}&Y&\rightmap{v}&Z&\rightmap{w}&TX\\
    \lmapdown{a}&&\lmapdown{b}&&\lmapdown{c}&&\lmapdown{T(a)}\\
    M&\rightmap{\underline{f}}&E&\rightmap{\underline{g}}&N&
    \rightmap{\underline{h}}&TM.\\
    \end{matrix}
    $$
    The elements of ${\cal T}_\infty$ are called the 
    \emph{triangles}\index{triangles of $A_\infty$-modules} of $\underline{\Mod}_\infty\g A$.
 \end{definition}
 
 As in (\ref{T: TGMod-A estable es triangulada}), the next statement follows from the general result for special Frobenius categories. 

 \begin{theorem}\label{T: TGMod-A estable es triangulada equiv a la DIGC-C} 
 Let $A$ be an $A_\infty$-algebra  and consider its associated differential tensor $S$-bocs 
 ${\cal B}_A$.
 The stable category $\underline{\Mod}_\infty\g A$ with the automorphism $T$ and the class of triangles ${\cal T}_\infty$ 
 defined above is a triangulated category. The functor $\hueca{G}:\Mod_\infty\g A\rightmap{}\TGM^0\g {\cal B}_A$ induces an equivalence of triangulated categories 
 $$\underline{\hueca{G}}:\underline{\Mod}_\infty\g A\rightmap{}\underline{\TGM}^0\g {\cal B}_A.$$
 \end{theorem}
 
 From the preceding theorem and (\ref{T: f0 quasi iso <-> f equiv homotopica}), we immediately obtain the following important statement, see \cite{K1}(5.2). 
 
 \begin{theorem}\label{T: quasi-iso de A-inf-modules es homotopy equivalence} 
  Assume that $S$ is a finite product of copies of the field $k$ and let $A$ be an $A_\infty$-algebra.
 Then, every quasi-isomorphism of $A_\infty$-modules is a homotopy equivalence, hence an invertible morphism in $\underline{\Mod}_\infty\g A$.  
 \end{theorem}

  \section{Restriction functors for $A_\infty$-modules}

 In this section we show how the result (\ref{C: equiv homotopica de bocses triangulares induce equiv de cats homotópicas}) on restriction functors for triangular differential graded $S$-bocses 
 translates into the corresponding result for restriction functors for $A_\infty$-algebras. We start by recalling the following link between homotopies in $\Alg_\infty$ and $\DGCoalg$, as defined in (\ref{D: homotopy for Alg-inf}) and (\ref{D: para morfismos de bocses homotopia}). 
 
 \begin{lemma}\label{L: Psi:Alg-infty--> DGCoalg preserva homotopias}
  The functor $\Psi:\Alg_\infty\rightmap{}\DGCoalg$ preserves and reflects null-homotopic morphisms.
 \end{lemma}
 
 For the sake of completeness, we notice that this follows, for instance, from the following statement.

 \begin{lemma}\label{D: conjunto de premorfismos}
  Given morphisms of $A_\infty$-algebras $f,g:A\rightmap{}B$ and $d\in \hueca{Z}$, we consider the set ${\cal H}^d_\infty(A,B)$ formed by the families 
  $h=\{h_n\}_{n\in \hueca{N}}$, such that, for each $n\in \hueca{N}$, the map 
  $h_n:A^{\otimes n}\rightmap{}B$ is a homogeneous morphism of $S$-$S$-bimodules of degree $\vert h_n\vert=d+1-n$. Make $\phi:=\Psi(f)$ and $\psi=\Psi(g)$. 
  Then, 
  \begin{enumerate}
   \item We have a linear isomorphism  
  $$\Delta=\Delta_{f,g}:
  {\cal H}^d_\infty(A,B)\rightmap{}
  \Coder^d_{\phi,\psi}({\cal B}_A,{\cal B}_B)$$ 
  constructed as follows. Given $h\in {\cal H}_\infty^{d}(A,B)$, we consider the morphism of $S$-$S$-bimodules $\widehat{h}:\overline{T}_S(A[1])\rightmap{}B[1]$ of degree $\vert \widehat{h}\vert=d$ determined by the commutativity of the squares 
  $$\begin{matrix} 
     A^{\otimes n}&\rightmap{\sigma_A^{\otimes n}}&A[1]^{\otimes n}\\
     \lmapdown{h_n}&&\rmapdown{\widehat{h}_n}\\
     B&\rightmap{\sigma_B}&B[1].\\
    \end{matrix}$$
Then,  consider the $\overline{\Psi}(f)\g \overline{\Psi}(g)$-coderivation $\overline{\Delta}(\widehat{h}):\overline{T}_S(A[1])\rightmap{}\overline{T}_S(B[1])$ given by $\widehat{h}$ and the universal property of the reduced tensor $S$-coalgebra $\overline{T}_S(A[1])$. Finally,  extend $\overline{\Delta}(\widehat{h})$ to a $\Psi(f)\g \Psi(g)$-coderivation of degree $d$ 
$$\Delta(h):T_S(A[1])\rightmap{}T_S(B[1]).$$ 
\item If $h\in {\cal H}^{-1}_\infty(A,B)$, with the notation of (\ref{D: homotopy for Alg-inf}), we have  
$$h^{\odot}:= H(h)-H_{f,g}(h)\in {\cal H}_\infty^0(A,B).$$
\item We have  $\Psi(f)-\Psi(g)\in\Coder^0_{\phi,\psi}({\cal B}_A,{\cal B}_B)$.
\item For any $\phi\g \psi$-coderivation $\xi\in \Coder^{-1}_{\phi,\psi}({\cal B}_A,{\cal B}_B)$, we have $$\xi^{\odot}:=\delta_B\xi+\xi\delta_A\in \Coder^{0}_{\phi,\psi}({\cal B}_A,{\cal B}_B).$$
\item For any $h\in {\cal H}^{-1}_\infty(A,B)$, we have $\Delta(h^{\odot})=\Delta(h)^{\odot}$.
\item We have $\Delta(f-g)=\phi-\psi$.
\item The map $\Delta$ induces a bijection between the homotopy sets 
${\cal H}(f,g)$ and ${\cal H}(\phi,\psi)$. 
\end{enumerate}
As a consequence, the morphisms $f$ and $g$ are homotopic in $\Alg_\infty$  if and only if the morphisms $\phi$ and $\psi$ are homotopic in $\DGCoalg$. 
 \end{lemma}

 \begin{proof} The map ${\cal H}^d_\infty(A,B)\rightmap{}\Hom^d_{\GM\g S\g S}(\overline{T}_S(A[1]),B[1])$ is clearly an isomorphism, and so is the map 
  $$\overline{\Delta}:\Hom^d_{\GM\g S\g S}(\overline{T}_S(A[1]),B[1])\rightmap{}
  \Coder^d_{\overline{\Psi}(f),\overline{\Psi}(g)}(\overline{T}_S(A[1]),\overline{T}_S(B[1]))$$
  given by the universal property of the reduced tensor $S$-coalgebra. Recall that the inverse of $\overline{\Delta}$ maps each coderivation $\xi$ onto $p_B\xi$, where $p_B:\overline{T}_S(B[1])\rightmap{}B[1]$ is the projection. 
  Finally, the extension $\overline{\Delta}(\widehat{h})\mapsto \Delta(h)$ is also bijective. 
  
  (2)--(4) are easy to show. In order to show (5), take $h\in {\cal H}^{-1}_\infty(A,B)$. Make $\overline{\Delta}(\widehat{h})^\odot:=\overline{\delta}_B\overline{\Delta}(\hat{h})+\overline{\Delta}(\hat{h})\overline{\delta}_A$. Notice that it will be enough to show the equality  
  $\overline{\Delta}(\widehat{h^\odot})=\overline{\Delta}(\widehat{h})^{\odot}$. Since both terms are $\phi\g \psi$-coderivations, it will be enough to show that 
  $p_B\overline{\Delta}(\widehat{h^\odot})_n=p_B\overline{\Delta}(\widehat{h})^{\odot}_n$, for all $n\geq 1$. 
  This means that 
  $$(\widehat{h^{\odot}})_n=p_B(\overline{\delta}_B\overline{\Delta}(\widehat{h})+\overline{\Delta}(\widehat{h})\overline{\delta}_A)_n, \hbox{ for all } n\geq 1.$$
   We have 
 $p_B(\overline{\delta}_B\overline{\Delta}(\hat{h})
 +\overline{\Delta}(\hat{h})\overline{\delta}_A)_n=D_n+R_n$
 where
 $$D_n=p_B\overline{\delta}_B\overline{\Delta}(\hat{h})_n \hbox{ \ and \  }
 R_n=p_B\overline{\Delta}(\hat{h})(\overline{\delta}_A)_n.$$
From explicit description of $\overline{\Delta}(\widehat{h})$, we have  
 $$\begin{matrix}
   D_n
      &=& \sum_{\scriptsize
  \begin{matrix}
            r,t\geq 0;s\geq 1\\
            i_1+\cdots+i_r+s\\
            +j_1+\cdots +j_t=n\\
         \end{matrix}}  p_B\overline{\delta}_B
         (\hat{f}_{i_1}\otimes\cdots\otimes \hat{f}_{i_r}\otimes 
         \hat{h}_s\otimes
         \hat{g}_{j_1}\otimes\cdots\otimes \hat{g}_{j_t})\hfill\\
       &=& \sum_{\scriptsize
  \begin{matrix}
            r,t\geq 0;s\geq 1\\
            i_1+\cdots+i_r+s\\
            +j_1+\cdots +j_t=n\\
         \end{matrix}} \widehat{m}_{r+1+t}^{B}
         (\hat{f}_{i_1}\otimes\cdots\otimes \hat{f}_{i_r}\otimes 
         \hat{h}_s\otimes
         \hat{g}_{j_1}\otimes\cdots\otimes \hat{g}_{j_t})\hfill\\
   \end{matrix}$$
and 
 $$\begin{matrix}
   R_n
   &=&
   \sum_{\scriptsize\begin{matrix} r+s+t=n\\ r,t\geq 0; s\geq 1 \end{matrix}} 
p_B\overline{\Delta}(\hat{h})(id^{\otimes r}\otimes \widehat{m}_s^{A}\otimes id^{\otimes t})\hfill\\

 &=&
   \sum_{\scriptsize\begin{matrix} r+s+t=n\\ r,t\geq 0; s\geq 1 \end{matrix}} 
\hat{h}_{r+1+t}(id^{\otimes r}\otimes \widehat{m}_s^A\otimes id^{\otimes t}).\hfill\\
 \end{matrix}$$

 We have  
 $$\sigma_B^{-1}R_n\sigma_A^{\otimes n}=\sum_{\scriptsize
  \begin{matrix} r+s+t=n\\
            r,t\geq 0; s\geq 1\\
         \end{matrix}}\hbox{\,}(-1)^{r+st}h_{r+1+t}
         (id^{\otimes r}\otimes m^A_s\otimes id^{\otimes t})
         =H(h)_n.$$
Moreover, we have 
$$ \sigma_B^{-1}D_n\sigma_A^{\otimes n} =
   \sum_{\scriptsize
  \begin{matrix}
            r,t\geq 0;s\geq 1\\
            i_1+\cdots+i_r+s\\
            +j_1+\cdots +j_t=n\\
         \end{matrix}} \sigma_B^{-1}\widehat{m}_{r+1+t}^BZ_n(i_1,\ldots,i_r,s,j_1,\ldots,j_t)$$
 where 
 $$Z_n(i_1,\ldots,i_r,s,j_1,\ldots,j_t)=               
         (\hat{f}_{i_1}\otimes\cdots\otimes \hat{f}_{i_r}\otimes 
         \hat{h}_s\otimes
         \hat{g}_{j_1}\otimes\cdots\otimes \hat{g}_{j_t})\sigma_A^{\otimes n}.
         $$
 Fix a vector $v=(i_1,\ldots,i_r,s,j_1,\ldots,j_t)$ and let us examine $Z_n(v)$. 
 We have 
 $$
(\hat{f}_{i_1}\otimes\cdots\otimes \hat{f}_{i_r}\otimes 
         \hat{h}_s\otimes
         \hat{g}_{j_1}\otimes\cdots\otimes \hat{g}_{j_t})(\sigma_A^{\otimes i_1}\otimes\cdots\otimes\sigma_A^{\otimes i_r}\otimes
         \sigma_A^{\otimes s}\otimes\sigma_A^{\otimes j_1}\otimes \cdots\otimes\sigma_A^{\otimes j_t})\hfill\\
 $$
 equals
 $$(-1)^{\sgn'}(\hat{f}_{i_1}\sigma_A^{\otimes i_1}\otimes\cdots\otimes
 \hat{f}_{i_r}\sigma_A^{\otimes i_r}\otimes 
         \hat{h}_s\sigma_A^{\otimes k}\otimes
         \hat{g}_{j_1}\sigma_A^{\otimes j_1}\otimes\cdots\otimes 
         \hat{g}_{j_t}\sigma_A^{\otimes j_t})$$
 and 
  $$(-1)^{\sgn'}(\sigma_Bf_{i_1}\otimes\cdots\otimes
 \sigma_Bf_{i_r}\otimes 
         \sigma_Bh_s\otimes
         \sigma_Bg_{j_1}\otimes\cdots\otimes \sigma_Bg_{j_t}),$$
 where $\sgn'=\sum_{u=1}^r i_u$. The last term equals 
 $$(-1)^{\sgn(i_1,\ldots,i_r,s,j_1,\ldots,j_t)}\sigma_B^{\otimes (r+1+t)}(f_{i_1}\otimes\cdots\otimes
 f_{i_r}\otimes h_s\otimes g_{j_1}\otimes\cdots\otimes g_{j_t}),$$
 where the sign was defined in (\ref{D: homotopy for Alg-inf}). Hence,
 $ \sigma_B^{-1}D_n\sigma_A^{\otimes n} =H_{f,g}(h)_n$
  
 Therefore,   for $n\geq 1$, we get 
   $\sigma_{B}^{-1}(D_n+R_n)\sigma_A^{\otimes n}=(h^{\odot})_n,$
or, equivalently, 
$$p_B(\overline{\delta}_B\overline{\Delta}(\hat{h})
 +\overline{\Delta}(\hat{h})\overline{\delta}_A)_n=(\widehat{h^{\odot}})_n.$$
  
(6): Since $\overline{\Psi}(f)-\overline{\Psi}(g)$ is a $\overline{\Psi}(f)\g \overline{\Psi}(g)$-coderivation, 
 from (1), 
in order to prove that $\overline{\Psi}(f)-\overline{\Psi}(g)=\overline{\Delta}(\widehat{f-g})$, we have to show that 
$$p_B(\overline{\Psi}(f)-\overline{\Psi}(g))_n=p_B\overline{\Delta}(\widehat{f-g})_n,
\hbox{  \ for all \ } n\in \hueca{N}.$$
This is clear, because, for $n\in \hueca{N}$, both terms coincide with $\widehat{f}_n-\widehat{g}_n$.

(7): If $h$ is a homotopy from $f$ to $g$, then $f-g=h^{\odot}$. Then, we have  
$$\overline{\Psi}(f)-\overline{\Psi}(g)=\overline{\Delta}(\widehat{f-g})=
\overline{\Delta}(\widehat{h^{\odot}})=\overline{\Delta}(\widehat{h})^{\odot}=
\overline{\delta}_B\xi+\xi\overline{\delta}_A,$$
where $\xi=\overline{\Delta}(\widehat{h})$ is a homotopy from $\overline{\Psi}(f)$ to $\overline{\Psi}(g)$. 

Conversely, if $\xi:\overline{C}_A\rightmap{}\overline{C}_B$ is a  
homotopy from $\overline{\Psi}(f)$ to $\overline{\Psi}(g)$, 
we have $\xi=\overline{\Delta}(\widehat{h})$, for some $h\in {\cal H}_\infty^{-1}(A,B)$ and 
$$\overline{\Delta}(\widehat{f-g})=\overline{\Psi}(f)-\overline{\Psi}(g)=
\overline{\delta}_A\xi+\xi\overline{\delta}_B=
\overline{\Delta}(\widehat{h})^{\odot}=\overline{\Delta}(\widehat{h^{\odot}}).$$
It follows that $f-h=h^{\odot}$, so $h$ is a homotopy from $f$ to $g$. 
 \end{proof}

\begin{proposition}\label{P: funtor restriccion de un infty-morfismo A-->B}
 Let $A$ and $B$ be $A_\infty$-algebras and consider a morphism 
 $\phi:A\rightmap{}B$ of $A_\infty$-algebras. Then, 
 \begin{enumerate}
  \item The morphism $\phi$ induces a 
 \emph{restriction functor} of differential graded categories 
 $$R_\phi:\GM\g B\rightmap{}\GM\g A$$
 defined as follows. Given $M\in \GM\g B$,
 the $A$-module $R_\phi(M)=M$.  
 Given a morphism $f:M\rightmap{}N$ in $\GM\g B$, the morphism 
$R_\phi(f)=\{R_\phi(f)_n\}_{n\in \hueca{N}}:R_\phi(M)\rightmap{}R_\phi(N)$ in $\GM\g A$ 
is given, for $n\geq 2$, by 
$$ R_\phi(f)_n= \sum_{\scriptsize
  \begin{matrix} 1\leq r\leq n-1\\
           i_1+\cdots +i_r=n-1\\
          \end{matrix}}
   (-1)^{\sgn(i_1,\ldots,i_r)}f_{r+1}(id_M\otimes \phi_{i_1}\otimes\cdots\otimes \phi_{i_r}),$$ 
   and $R_\phi(f)_1=f_1$, where $\sgn(i_1,\ldots,i_r)$ is as in (\ref{A-infinite algebra}). 
   
\item  The morphism $\phi$ induces a 
 \emph{restriction functor} of graded categories
 $$R_\phi:\TGM\g B\rightmap{}\TGM\g A.$$
 If $(M,\{m_n^M\})\in \TGM\g B$,
 by definition, $R_\phi(M,m^M)=(M,R_\phi(m^M))$.  
 Given a morphism $f:(M,m^M)\rightmap{}(N,m^N)$ in $\TGM\g B$, the morphism 
$R_\phi(f)=\{R_\phi(f)_n\}_{n\in \hueca{N}}:R_\phi(M)\rightmap{}R_\phi(N)$ in $\GM\g A$ 
is in fact a morphism $R_\phi(M,m^M)\rightmap{}R_\phi(N,m^N)$ in $\TGM\g A$. 
   
\item Let $\overline{{\cal B}}_A$ and $\overline{{\cal B}}_B$ denote the reduced tensor $S$-coalgebras  
associated to $A$ and $B$ respectively, and denote by ${\cal B}_A$ and ${\cal B}_B$ the corresponding 
tensor $S$-coalgebras. Then the morphism of $A_\infty$-algebras $\phi:A\rightmap{}B$ determines
a morphism of differential $S$-coalgebras $\overline{\Psi}(\phi):\overline{{\cal B}}_A\rightmap{}\overline{{\cal B}}_B$, 
which extends to a morphism 
$\Psi(\phi):{\cal B}_A\rightmap{}{\cal B}_B$, as in 
(\ref{L: homotopias de morfismos de overline C a C}). Then, we have  commutative 
squares of functors 
$$\begin{matrix}
   \GM\g B&\rightmap{\hueca{G}_B}&\GM\g {\cal B}_B\\
   \lmapdown{R_\phi}&&\lmapdown{R_{\Psi(\phi)}}\\
   \GM\g A&\rightmap{\hueca{G}_A}&\GM\g {\cal B}_A\\
  \end{matrix} \hbox{  \ and  \ } 
  \begin{matrix}
   \TGM\g B&\rightmap{\hueca{G}_B}&\TGM\g {\cal B}_B\\
   \rmapdown{R_\phi}&&\rmapdown{R_{\Psi(\phi)}}\\
   \TGM\g A&\rightmap{\hueca{G}_A}&\TGM\g {\cal B}_A.\\
  \end{matrix} 
  $$
where $\hueca{G}_A$ and $\hueca{G}_B$ are the functors defined in (\ref{P: equiv entre GMod-A y GMod-B(A)}) and  (\ref{P: equiv entre TGMod-A y TGMod-B(A)}).   
 \item Moreover, $R_\phi$ maps null-homotopic morphisms of $\Mod_\infty\g B$ onto 
 null-homotopic morphisms of $\Mod_\infty\g A$. So, it induces a  functor 
 $\underline{R}_\phi$ and a commutative 
 square of functors 
 $$\begin{matrix}
   \underline{\Mod}_\infty\g B&\rightmap{\underline{\hueca{G}}_B}&\underline{\TGM}^0\g {\cal B}_B\\
   \lmapdown{\underline{R}_\phi}&&\rmapdown{\underline{R}_{\Psi(\phi)}}\\
   \underline{\Mod}_\infty\g A&\rightmap{\underline{\hueca{G}}_A}&\underline{\TGM}^0\g {\cal B}_A.\\
  \end{matrix}$$
  \end{enumerate} 
\end{proposition}
 
 \begin{proof} In order to prove (1), it is enough to show the commutativity of the first diagram in (3), because we already know that $R_{\Psi(\phi)}$ is a functor and $\hueca{G}_A$ and $\hueca{G}_B$ are equivalences, and they all preserve the differentials of the categories.  Clearly, $\hueca{G}_A R_\phi$ and $R_{\Psi(\phi)}\hueca{G}_B$
coincide on objects. Take a morphism $f:M\rightmap{}N$  in $\GM\g B$ and let us show that $\hueca{G}_A R_\phi(f)= R_{\Psi(\phi)}\hueca{G}_B(f)$. For this we have to show that, for each $n\geq 0$, we have 
$\hueca{G}_A(R_\phi(f))_n=R_{\Psi(\phi)}\hueca{G}_B(f)_n$. Since $R_{\Psi(\phi)}\hueca{G}_B(f)_n=R_{\Psi(\phi)}(\widehat{f})_n=\widehat{f}(id_{M[1]}\otimes \Psi(\phi))_n$, the former is equivalent to show 
 that $\widehat{R_\phi(f)}_n=\widehat{f}(id_{M[1]}\otimes \Psi(\phi))_n$. For this we will show that the following diagrams commute 
  $$\begin{matrix}
  M\otimes_S A^{\otimes n}&
  \rightmap{ \ \sigma_M\otimes \sigma^{\otimes n} \ }&
  M[1]\otimes_S A[1]^{\otimes n}\\
     \lmapdown{R_\phi(f)_{n+1}}&&\lmapdown{\hat{f}(id_{M[1]}\otimes \Psi(\phi))_{n}}\\
  N&\rightmap{ \  \ \ \sigma_N \ \  \ }&N[1]\\
  \end{matrix} \hbox{ and  \  \   } 
  \begin{matrix}
  M&\rightmap{ \ \zeta\sigma_M \ \ }&M[1]\otimes_S S\\
     \rmapdown{R_\phi(f)_1}&&\rmapdown{\hat{f}(id_{M[1]}\otimes \Psi(\phi))_0}\\
  N&\rightmap{ \ \ \  \sigma_N \   \ \ }&N[1],\\
  \end{matrix}$$
  where $n$ runs in $\hueca{N}$ and $\zeta:M[1]\rightmap{}M[1]\otimes_SS$ is the canonical isomorphism.
 
 Recall that, for $n\geq 1$,  we have  
  $$\Psi(\phi)_n=\sum_{\scriptsize
  \begin{matrix} 1\leq r\leq n\\
           i_1+\cdots +i_r=n\\
          \end{matrix}}\widehat{\phi}_{i_1}\otimes \cdots\otimes \widehat{\phi}_{i_r},$$
  where $\widehat{\phi}_{i}\sigma^{\otimes i}=\sigma \phi_{i}$, for all $i\in \hueca{N}$.         
 Then, if  $Q:=\widehat{f}(id_{M[1]}\otimes \Psi(\phi))_n(\sigma_M\otimes \sigma^{\otimes n})$, we have 
 $$\begin{matrix}
Q
&=&
\sum_{\scriptsize
  \begin{matrix} 1\leq r\leq n\\
           i_1+\cdots +i_r=n\\
          \end{matrix}}\widehat{f}_r(id_{M[1]}\otimes\widehat{\phi}_{i_1}\otimes \cdots\otimes \widehat{\phi}_{i_r})(\sigma_M\otimes \sigma^{\otimes n})\hfill\\
&=&  
\sum_{\scriptsize
  \begin{matrix} 1\leq r\leq n\\
           i_1+\cdots +i_r=n\\
          \end{matrix}}\widehat{f}_r(\sigma_{M}\otimes\widehat{\phi}_{i_1}\sigma^{\otimes {i_1}}\otimes \cdots\otimes \widehat{\phi}_{i_r}\sigma^{\otimes i_r})\hfill\\
&=& 
\sum_{\scriptsize
  \begin{matrix} 1\leq r\leq n\\
           i_1+\cdots +i_r=n\\
          \end{matrix}}\widehat{f}_r(\sigma_{M}\otimes \sigma\phi_{i_1}\otimes \cdots\otimes \sigma\phi_{i_r})\hfill\\
&=& 
\sum_{\scriptsize
  \begin{matrix} 1\leq r\leq n\\
           i_1+\cdots +i_r=n\\
          \end{matrix}}(-1)^{\sgn}\widehat{f}_r(\sigma_{M}\otimes \sigma^{\otimes r})(id_M\otimes \phi_{i_1}\otimes \cdots\otimes \phi_{i_r})\hfill\\
&=& 
\sum_{\scriptsize
  \begin{matrix} 1\leq r\leq n\\
           i_1+\cdots +i_r=n\\
          \end{matrix}}(-1)^{\sgn}\sigma_Nf_{r+1}(id_M\otimes \phi_{i_1}\otimes \cdots\otimes \phi_{i_r})\hfill\\
&=& 
\sigma_NR_\phi(f)_{n+1}.\hfill\\
\end{matrix}$$          
For $n=0$, we have 
$$\widehat{f}(id_{M[1]}\otimes \Psi(\phi))_0\zeta\sigma_M=\widehat{f}_0(\sigma_M\otimes 1)=\widehat{f}_0\zeta\sigma_M=\sigma_Nf_1=\sigma_NR_\phi(f)_1.$$

(2) and the commutativity of the second square in (3) follow from the commutativity of the first square and the 
fact that $\TGM\g A$ is constructed from $\GM\g A$ in a similar way that $\TGM\g {\cal B}_A$ is constructed from 
$\GM\g {\cal B}_A$, using the corresponding differentials, and the  functor 
$R_\phi:\TGM\g B\rightmap{}\TGM\g A$ is constructed from 
$R_\phi:\GM\g B\rightmap{}\GM\g A$ in a similar way that the functor 
$R_{\Psi(\phi)}:\TGM\g {\cal B}_B\rightmap{}\TGM\g {\cal B}_A$ is constructed from the functor 
$R_{\Psi(\phi)}:\GM\g {\cal B}_B\rightmap{}\GM\g {\cal B}_A$. 

(4) Given morphisms $f,g:(M,m^M)\rightmap{}(N,m^N)$ in $\Mod_\infty\g A$, hence in $\TGM^0\g A$, any homotopy $h$ from $f$ to $g$ is a homogeneous morphism in $\GM\g A$ with degree $\vert h\vert=-1$ and is mapped by the restriction functor $R_\phi:\GM\g B\rightmap{}\GM\g A$ onto the homotopy $R_\psi(h)$ from $R_\psi(f)$ to $R_\psi(g)$, see (\ref{D: homotopia para GMod0-A es la original}). Hence, there is an induced functor $\underline{R}_\phi$ which clearly makes the diagram of (4) to commute. 
 \end{proof}

 \begin{corollary}\label{C: teo sobre homotopia y restricciones para morf de A-infinito algs}
Let $f,g:A\rightmap{}B$ be homotopic morphisms of $A_\infty$-algebras. Then,
there is an isomorphism of functors 
$$ \underline{R}_f\cong \underline{R}_g$$
where $ \underline{R}_f, \underline{R}_g: 
\underline{\Mod}_\infty\g B\rightmap{}\underline{\Mod}_\infty\g A$ are the functors 
induced on the homotopy categories by the restriction functors 
$R_f,R_g:\Mod_\infty\g B\rightmap{}\Mod_\infty\g A$, respectively. Any homotopy equivalence $f:A\rightmap{}B$ of $A_\infty$-algebras determines an equivalence of categories $\underline{R}_f:\underline{\Mod}_\infty\g B\rightmap{}\underline{\Mod}_\infty\g A$. 
\end{corollary}
  
\begin{proof} Let ${\cal B}_A$ and ${\cal B}_B$ be the differential tensor 
$S$-coalgebras (or differential tensor $S$-bocses) associated to $A$ and $B$, respectively.  Then, $\phi=\Psi(f)$ and $\psi=\Psi(g)$ are homotopic morphisms of graded $S$-coalgebras 
${\cal B}_A\rightmap{}{\cal B}_B$. We have the following commutative diagram 
$$\begin{matrix}
 \underline{\Mod}_\infty\g B&\rightmap{\underline{\hueca{G}}_B}&
 \underline{\TGM}^0\g {\cal B}_B\\
 \lmapdown{\underline{R}_f}&&\rmapdown{\underline{R}_\phi}\\
 \underline{\Mod}_\infty\g A&\rightmap{\underline{\hueca G}_A}&
 \underline{\TGM}^0\g {\cal B}_A,\\
  \end{matrix}$$ 
and similarly for the morphisms $g$ and $\psi$. 
  From (\ref{P: equiv de funtores restriccion de morfismos homotopicos}), we know there is an isomorphism of functors 
$\underline{R}_\phi\cong\underline{R}_\psi$.
Denote by $\hueca{G}'_A$ a quasi inverse for the equivalence $\hueca{G}_A$.  From the isomorphism  of 
functors $\underline{R}_\phi\underline{\hueca{G}}_B\cong \underline{R}_\psi
 \underline{\hueca{G}}_B$, we obtain  
$$\underline{R}_f
\cong
\underline{\hueca{G}}_A'\underline{\hueca{G}}_A\underline{R}_f
=
\underline{\hueca{G}}_A'\underline{R}_\phi
\underline{\hueca{G}}_B
\cong
\underline{\hueca{G}}_A'\underline{R}_\psi\underline{\hueca{G}}_B
=
\underline{\hueca{G}}_A'\underline{\hueca{G}}_A\underline{R}_g
\cong 
\underline{R}_g.$$
\end{proof}

\noindent{\bf  Acknowledgements.}  The second author acknowledges the hospitality of
Centro de Ciencias Matem\'aticas, UNAM,  during his sabbatical year  and  the support of CONACyT
sabbatical grant 2018-000008-01NACV.

 \hskip2cm

\vbox{\noindent R. Bautista\\
Centro de Ciencias Matem\'aticas\\
Universidad Nacional Aut\'onoma de M\'exico\\
Morelia, M\'exico\\
raymundo@matmor.unam.mx\\}

\vbox{\noindent E. P\'erez\\
Facultad de Matem\'aticas\\
Universidad Aut\'onoma de Yucat\'an\\
M\'erida, M\'exico\\
jperezt@correo.uady.mx\\}

\vbox{\noindent L. Salmer\'on\\
Centro de Ciencias  Matem\'aticas\\
Universidad Nacional Aut\'onoma de M\'exico\\
Morelia, M\'exico\\
salmeron@matmor.unam.mx\\}

 \end{document}